%% file: main.tex
\documentclass[twocolumn]{autart}

\usepackage[top=.75in, bottom=.85in, left=.65in, right=.65in,paperwidth=8.5in,paperheight=11in]{geometry}

\usepackage{mathtools}
\usepackage{enumitem}
\usepackage{amsmath,amssymb}
\usepackage{graphicx}
\usepackage{epstopdf,cite}
\usepackage{url}
\usepackage{bm}
\usepackage{multirow,booktabs,multicol}
\usepackage{mathtools}
\usepackage{setspace}
\usepackage{hyperref}
\usepackage{makecell}
\usepackage{pifont}
\newcommand{\xmark}{\ding{55}}%
\usepackage{graphicx,color}
\usepackage{subfigure,balance}

\usepackage[english]{babel}

\usepackage{tikz}

\usepackage[justification=centering]{caption}
\newcommand{\tr}{{{\mathsf T}}}

\newtheorem{definition}{Definition}
\newtheorem{theorem}{Theorem}

\newtheorem{proposition}{Proposition}
\newtheorem{remark}{Remark}
\newtheorem{lemma}{Lemma}
\newtheorem{example}{Example}

\newtheorem{corollary}{Corollary}
\newtheorem{assumption}{Assumption}

\newenvironment{proof}{\textbf{Proof:}}{\hfill$\square$}

\let\OLDthebibliography\thebibliography
\renewcommand\thebibliography[1]{
  \OLDthebibliography{#1}
  \setlength{\parskip}{0pt}
  \setlength{\itemsep}{3pt plus 0.3ex}
}

\makeatletter
\def\endthebibliography{%
  \def\@noitemerr{\@latex@warning{Empty `thebibliography' environment}}%
  \endlist
}
\makeatother

\begin{document}

\setlength{\abovedisplayskip}{9pt}
\setlength{\belowdisplayskip}{9pt}

	\begin{frontmatter}
		
		\title{System-level, Input-output and New Parameterizations of Stabilizing Controllers, and Their Numerical {Computation}}
		
		\thanks[footnoteinfo]{The material of this paper was not presented at any conference. This work is supported by NSF career 1553407, AFOSR Young Investigator Program, and ONR Young Investigator Program. L. Furieri and M. Kamgarpour are gratefully supported by ERC Starting Grant CONENE. }
		
		\author[ucsd]{Yang~Zheng}\ead{zhengy@eng.ucsd.edu},
		\author[EPFL,ETH]{Luca Furieri}\ead{luca.furieri@epfl.ch},
		\author[EPFL]{Maryam Kamgarpour}\ead{maryam.kamgarpour@epfl.ch},
		\author[Harvard1]{Na~Li}\ead{nali@seas.harvard.edu}
		
		\address[ucsd]{Department of Electrical and Computer Engineering, University of California San Diego, CA 92093.} \vspace{-.1in}
		\address[Harvard1]{School of Engineering and Applied Sciences, Harvard University, Boston, MA, 02138, U.S.} \vspace{-.1in} 
		
		\address[ETH]{Automatic Control Laboratory,  ETH Zurich, Switzerland.}  \vspace{-.1in}             
		\address[EPFL]{Institute of Mechanical Engineering, \'Ecole Polytechnique F\'ed\'erale de Lausanne (EPFL), Switzerland} \vspace{-.1in}
	 \vspace{-.2in}
		

		\begin{abstract}
		It is known that the set of internally stabilizing controller $\mathcal{C}_{\text{stab}}$ is non-convex, but it admits convex characterizations using certain closed-loop maps: a classical result is the {Youla parameterization}, and two recent notions are the {system-level parameterization} (SLP) and the {input-output parameterization} (IOP). In this paper, we address the existence of new convex parameterizations and discuss potential tradeoffs of each parametrization in different scenarios. Our main contributions are: 1) We  reveal that only four groups of stable closed-loop transfer matrices are equivalent to internal stability: one of them is used in the SLP, another one is used in the IOP, and the other two are new, leading to two new convex parameterizations of $\mathcal{C}_{\text{stab}}$. 2) We  investigate the properties of these parameterizations after imposing the finite impulse response (FIR) approximation, revealing that the IOP has the best ability of approximating $\mathcal{C}_{\text{stab}}$ given FIR constraints. 3) These four parameterizations require no \emph{a priori} doubly-coprime factorization of the plant, but impose a set of equality constraints. However, these equality constraints will never be satisfied exactly in floating-point arithmetic computation and/or implementation. We prove that the IOP is numerically robust for open-loop stable plants, in the sense that small mismatches in the equality constraints do not compromise the closed-loop stability; but a direct IOP implementation will fail to stabilize open-loop unstable systems in practice. The SLP is known to enjoy numerical robustness in the state feedback case; here, we show that numerical robustness of the four-block SLP controller requires case-by-case analysis even the plant is open-loop stable. 
		\end{abstract}
		
	\end{frontmatter}

\section{Introduction}
Feedback systems must be stable in some appropriate sense for practical deployment, and thus one fundamental problem in control theory is to design a feedback controller that stabilizes a given dynamical system~\cite{zhou1996robust}. Indeed, many control synthesis problems include stability as a constraint while optimizing some performance~\cite{dullerud2013course}. However, it is well-known that the set of stabilizing controllers is non-convex, and hence, hard to search directly over. One standard approach is to parameterize all stabilizing controllers and the corresponding closed-loop responses in a convex way, and then to optimize the performance over the new parameter(s) using convex optimization~\cite{boyd1991linear}.

A classical method for parameterizing  the set of all internally stabilizing controllers is based on the celebrated \emph{Youla parameterization}~\cite{youla1976modern} which relies on a doubly-coprime  factorization  of the system. It is shown in~\cite{boyd1991linear} that many performance specifications on the closed-loop system can be expressed in the Youla parameterization framework via convex optimization. Moreover, the foundational results of robust and optimal control are built on the Youla parameterization~\cite{zhou1996robust,francis1987course}. Recently, a \emph{system-level parameterization} (SLP)~\cite{wang2019system} and an \emph{input-output parameterization} (IOP)~\cite{furieri2019input} were proposed to characterize the set of internally stabilizing controllers, 
without relying on the doubly-coprime factorization technique. In principle, Youla, the SLP, and the IOP all directly treat certain closed-loop responses as design parameters and thus shift the controller synthesis from the design of a controller to the design of the closed-loop responses. We note that an open-source Python-based implementation for the SLP and the IOP is available~\cite{2020arXiv200412565T}.

Besides the classical control synthesis problems~\cite{zhou1996robust,francis1987course}, closed-loop parameterizations are powerful tools in other areas, including distributed optimal control~\cite{rotkowitz2006characterization,qi2004structured,lessard2015convexity,shah2013cal,matni2016regularization,sabuau2014youla} and quantifying the performance of learning in control~\cite{dean2017sample,dean2018regret, simchowitz2020improper, lale2020logarithmic,furieri2019learning,zheng2021sample}. In distributed control, the goal is to  design sub-controllers relying on locally available information, which is crucial for many cyber-physical systems. Enforcing these information constraints, however, may make the problems~computationally intractable~\cite{Witsenhausen,tsitsiklis1985complexity}. Nevertheless, it is well-known that a notion of \emph{quadratic invariance} (QI)~\cite{rotkowitz2006characterization,lessard2015convexity,sabuau2014youla} allows equivalently translating information constraints on the controller to convex constraints on the Youla parameter, thus preserving the convexity of distributed controller synthesis. The QI notion can also be integrated with the SLP and the IOP, resulting in equivalent convex reformulations~\cite{wang2019system,furieri2019input}. Together with a recent notion of \emph{Sparsity Invariance}~\cite{furieri2019sparsity}, these closed-loop parameterizations enable deriving convex approximations 
for problems with general sparsity constraints 
beyond QI; see~\cite[Remark 5]{zheng2019equivalence} for example.

For learning in control, the SLP was integrated within a \emph{Coarse-ID} control procedure to derive a 
sample complexity bound for {learning} the classical linear quadratic regulator (LQR)~\cite{dean2017sample}.~This procedure was {exploited} in~\cite{dean2018regret} to derive high probability guarantees of sub-linear regret {using an adaptive LQR control architecture}. In~{\cite{simchowitz2020improper, lale2020logarithmic}}, based on the Youla parameterization, an online gradient descent algorithm was proposed to achieve sub-linear regret {for} learning the linear quadratic gaussian (LQG) controller.
In~\cite{furieri2019learning}, the Youla framework was used to derive a sample complexity~bound on learning the \emph{globally optimal} distributed controller subject to QI constraints. The IOP was also used to derive an end-to-end sample complexity bound on learning LQG controllers for stable systems~\cite{zheng2021sample}.
{The results in \cite{simchowitz2020improper, lale2020logarithmic,furieri2019learning,zheng2021sample} motivate the shift from static controllers to dynamic ones in complex learning-based control tasks.}

Youla~\cite{youla1976modern}, the SLP~\cite{wang2019system}, and the IOP~\cite{furieri2019input} are fundamental building blocks for distributed controller synthesis and learning-based control applications. Nevertheless, a few critical issues have been left unexplored. First, while it is known that Youla, the SLP, and the IOP are all equivalent~\cite{zheng2019equivalence}, it remains unclear whether there exist other equivalent parameterizations beyond them. Second, the decision variables in these closed-loop parameterizations are, in general, infinite-dimensional. The works \cite{wang2019system,anderson2019system,furieri2019input} enforce \emph{finite impulse response} (FIR) constraints on the decision variables to enable formulation as finite-dimensional convex optimization problems. However, these works do not characterize the conservatism introduced by the FIR approximation using different parameterizations. Third, unlike Youla, the SLP and the IOP do not need to know a doubly-coprime factorization \emph{a priori}, and instead introduce a set of equality constraints for achievable closed-loop responses. A fact that is not investigated in the SLP~\cite{wang2019system,anderson2019system}, the IOP~\cite{furieri2019input} or the recent work~\cite{zheng2019equivalence} is that the set of equality constraints can never be satisfied exactly in numerical computation, potentially affecting the closed-loop stability.

\subsection{Contributions}

This paper aims to investigate the issues raised above and provide a complete understanding of closed-loop parameterizations. We introduce new parameterizations beyond SLP/IOP and discuss tradeoffs among these parameterizations in different scenarios. Note that our previous work~\cite{zheng2019equivalence} established explicit affine mappings among Youla, SLP, and IOP parameters, but it provides no investigation of the issues mentioned above. Specifically, the contributions of this paper are as follows.

First, we examine all possible parameterizations for the set of internally stabilizing controllers $\mathcal{C}_{\text{stab}}$ using closed-loop responses from the disturbances 
$(\bm{\delta}_\textbf{x}, \bm{\delta}_\textbf{y}, \bm{\delta}_\textbf{u})$ to state, output, control signals $(\textbf{x}, \textbf{y}, \textbf{u})$; see Figure~\ref{Fig:LTI} for illustration. Our strategy is to examine the cases where the stability of external transfer matrices is equivalent to internal stability. We reveal that only four groups of stable disturbance-to-signal maps can guarantee internal stability (see Theorem~\ref{Theo:mainresult}): one of them is used in the SLP~\cite{wang2019system}, another one is a classical result and is used in the IOP~\cite{furieri2019input}, and the other two have not been discussed before and thus can be used to derive two new parameterizations (Propositions~\ref{prop:slp3} and~\ref{prop:slp4}). Our results are \emph{complete and exclusive}, in the sense that there are no other parameterizations for $\mathcal{C}_{\text{stab}}$ using closed-loop responses from $(\bm{\delta}_\textbf{x}, \bm{\delta}_\textbf{y}, \bm{\delta}_\textbf{u})$ to $(\textbf{x}, \textbf{y}, \textbf{u})$.

Second, we 
investigate the impact of imposing FIR constraints on the closed-loop parameterizations. We show that the IOP provides the tightest approximation of $\mathcal{C}_{\text{stab}}$ 
after imposing FIR constraints (Theorem~\ref{theo:FIR}). This result is enabled by the fact that the IOP {directly} deals with the~maps from inputs to outputs without {passing through} the system state, while the SLP and the two new parametrizations explicitly involve the system state or disturbances on the system state. Motivated by~\cite{anderson2017structured}, we characterize state-space realizations for the controllers in closed-loop parameterizations after imposing FIR approximations (Theorem~\ref{Theo:IOPstatespace}). The state-space realizations provide easily implementable controllers for practical deployment.

Third, we quantify the numerical robustness of closed-loop parameterizations due to floating-point arithmetic in both numerical computation and controller implementation.  
We prove that in the IOP, small numerical mismatches in the equality constraints do not compromise~closed-loop stability for open-loop stable plants, but will destabilize the closed-loop system for unstable plants (Theorem~\ref{theo:IOProbustness}). This result holds similarly for the two new closed-loop parameterizations. We also show that, in general, the four-block SLP controller in the output-feedback case is vulnerable to destabilization due to small mismatches in the equality constraints (Theorem~\ref{theo:SLProbustness}), no matter whether the plant is open-loop stable or unstable.~This~issue arises irrespective of which SLP controller implementation is used \cite{wang2019system,anderson2019system}. These results suggest that it is unwise to directly apply the IOP or related parameterizations for open-loop unstable systems; instead, utilizing a pre-stabilizing controller is more appealing since it avoids the destabilization issue. 

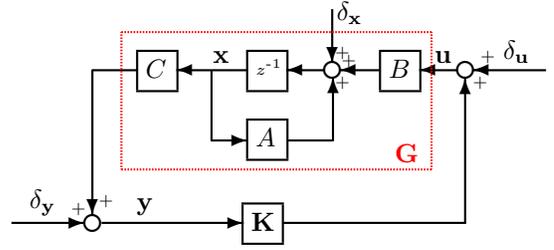
\begin{figure}[t]
\setlength{\belowcaptionskip}{-10pt}
\setlength{\abovecaptionskip}{-10pt}
  \centering
  \begin{center}
  \input{interconnection.tex}
  \end{center}
  \caption{Interconnection of the plant $\mathbf{G}$ and the controller $\mathbf{K}$.}\label{Fig:LTI}
\end{figure}

\subsection{Paper Structure}

The rest of this paper is organized as follows. We {state the problem} in Section~\ref{Section:preliminaries}. The relationship between the stability of external transfer matrices and internal stability is revealed in Section~\ref{Section:IOS}. Four parameterizations of stabilizing controllers using closed-loop responses, including the SLP and the IOP, are presented in Section~\ref{section:paramterization}. Numerical computation using the FIR constraints and controller implementation are discussed in Section~\ref{section:FIR}. We investigate the numerical robustness of closed-loop parameterizations in Section~\ref{section:robustness}. A numerical application is shown in Section~\ref{section:example}, and we conclude the paper in Section~\ref{section:conclusion}. 

\subsection{Notation}
{The symbols $\mathbb{R}$ and $\mathbb{N}$ refer to the set of real and integer numbers, respectively.} We use lower and upper case letters (\emph{e.g.} $x$ and $A$) to denote vectors and matrices, respectively. Lower and upper case boldface letters (\emph{e.g.} $\mathbf{x}$ and $\mathbf{G}$) are used to denote signals and transfer matrices, respectively.  We denote the set of real-rational proper stable\footnote{Throughout the paper, ``stable'' means ``asymptotically stable'', \emph{i.e.}, all eigenvalues/poles have strictly negative real parts in continuous time (magnitudes less than 1 in discrete time).} transfer matrices as $\mathcal{RH}_{\infty}$.  We use the notation $\mathbf{G} \in \frac{1}{z}\mathcal{RH}_{\infty}$ to denote that $\mathbf{G}$ is stable and strictly proper. Given $\mathbf{G} \in \mathcal{RH}_{\infty}$, we denote its $\mathcal{H}_{\infty}$ norm by $\|\mathbf{G}\|_{\infty}:= \sup_{\omega} {\sigma}_{\max} (\mathbf{G}(e^{j\omega}))$, where ${\sigma}_{\max}(\cdot)$ denotes the maximum singular value. Given a stable transfer matrix $\mathbf{G}(z)$, the square of its $\mathcal{H}_2$ norm is $\|\mathbf{G}\|^2_{\mathcal{H}_2}:= {\frac{1}{2\pi}\int_{-\pi}^{\pi} \text{Trace}\left(\mathbf{G}^*(e^{j\omega})\mathbf{G}(e^{j\omega})\right)d\omega}$. For simplicity, we omit the dimension of transfer matrices, which shall be clear in the context. Also, we use $I$ (resp. $0$) to denote the identity matrix (resp. zero matrix) of compatible dimension. In Section~\ref{section:statespace}, to avoid ambiguity, we explicitly write the matrix dimension and use $I_p$ to denote the identity matrix of dimension $p$.  Finally, the state-space realization $C(zI - A)^{-1}B + D$ is denoted as
 $ \left[\begin{array}{c|c} A & B \\\hline
     C & D\end{array}\right].$

\section{Problem statement}\label{Section:preliminaries}


\subsection{System model}

We consider strictly proper discrete-time linear time-invariant (LTI) plants of the form
\begin{equation}\label{eq:LTI}
    \begin{aligned}
        x[t+1] &= A x[t] + B u[t] + \delta_x[t],\\
        y[t]   &= Cx[t] + \delta_y[t],
    \end{aligned}
\end{equation}
where $x[t] \in \mathbb{R}^{n},u[t]\in \mathbb{R}^{m},y[t]\in \mathbb{R}^{p}$ are the state vector, control action, and measurement vector at time $t$, respectively; $\delta_x[t]\in \mathbb{R}^{n}$ and $\delta_y[t]\in \mathbb{R}^{p}$ are external disturbances on the state and measurement vectors at time $t$, respectively. The transfer matrix from $\mathbf{u}$ to $\mathbf{y}$ is
$
    \mathbf{G} = C(zI - A)^{-1}B.
$
Consider an LTI dynamical controller
\begin{equation}\label{eq:LTIController}
    \mathbf{u} = \mathbf{K}\mathbf{y} + \bm{\delta}_u,
\end{equation}
where $\bm{\delta}_u $ is the external disturbance on the control action. A state-space realization of~\eqref{eq:LTIController} is
\begin{equation}\label{eq:state_space_controller}
    \begin{aligned}
        \xi[t+1] &= A_k \xi[t] + B_k y[t],\\
        u[t]   &= C_k\xi[t] + D_k y[t]+ \delta_u[t],
    \end{aligned}
\end{equation}
where $\xi[t]\in \mathbb{R}^{q}$ is the internal state of the controller at time $t$. {The formulation~\eqref{eq:state_space_controller} above reduces to a static controller when $(A_k,B_k,C_k,D_k)=(0,0,0,K)$ for some $K \in \mathbb{R}^{m \times p}$. }In this paper, we make the following standard assumption.
\begin{assumption} \label{assumption:stabilizability}
    Both the plant and controller realizations are stabilizable and detectable, \emph{i.e.}, $(A, B)$ and $(A_k, B_k)$ are stabilizable, and $(A,C)$ and $(A_k,C_k)$ are detectable.
\end{assumption}

 Applying the controller~\eqref{eq:LTIController} to the plant~\eqref{eq:LTI} leads to a closed-loop system shown in Fig.~\ref{Fig:LTI}. 
 Since the plant is strictly proper, the closed-loop system is always well-posed~\cite[Lemma 5.1]{zhou1996robust}.

\subsection{Internal stability}
{Internal stability is defined as follows}~\cite[Chapter 5.3]{zhou1996robust}:
\begin{definition}
    The system in~Fig.~\ref{Fig:LTI} is \emph{internally stable} if it is well-posed, and the states $(x[t],\xi[t])$ converge to zero as $t\rightarrow \infty$ for all initial states $(x[0],\xi[0])$ when $\delta_x[t] = 0, \delta_y[t]=0, \delta_u[t] = 0, \forall t$.
\end{definition}

We say the controller $\mathbf{K}$ \emph{internally stabilizes} the plant $\mathbf{G}$ if the closed-loop system in Fig.~\ref{Fig:LTI} is {internally stable}. The set of all LTI internally stabilizing controllers is defined as
\begin{equation}
    \mathcal{C}_{\text{stab}} := \{\mathbf{K} \mid \mathbf{K} \; \text{internally stabilizes} \; \mathbf{G}\}.
\end{equation}
Note that {when an infinite time-horizon is considered}, a feedback system must at least be stable, and  any controller synthesis will implicitly or explicitly involve a constraint $\mathbf{K} \in \mathcal{C}_{\text{stab}}$. Therefore, it is of fundamentally importance to characterize $ \mathcal{C}_{\text{stab}}$. Indeed, it is well-known that $\mathcal{C}_{\text{stab}}$ is non-convex, and it is not difficult to find explicit examples where $\mathbf{K}_1, \mathbf{K}_2\in \mathcal{C}_{\text{stab}}$ and $\frac{1}{2}(\mathbf{K}_1 + \mathbf{K}_2)\notin \mathcal{C}_{\text{stab}}$. Accordingly, it is not easy to directly search over $\mathbf{K} \in \mathcal{C}_{\text{stab}}$ for control synthesis, and a suitable change of variables is used in many control synthesis procedures~\cite{zhou1996robust,dullerud2013course,boyd1991linear,youla1976modern,francis1987course,wang2019system,furieri2019input,zheng2019equivalence}.

A standard state-space characterization of
internal stabilization is as follows.
\begin{lemma}[{\cite[Lemma 5.2]{zhou1996robust}}] \label{lemma:internalstability}
    Under Assumption~\ref{assumption:stabilizability}, $\mathbf{K}$ \emph{internally stabilizes}  $\mathbf{G}$ if and only if
    \begin{equation} \label{eq:closed_loop_matrix}
    A_{\text{\emph{cl}}} := \begin{bmatrix}
    A+BD_kC & BC_k\\
  B_kC & A_k
    \end{bmatrix}
\end{equation}
is stable.
\end{lemma}

Note that the result in Lemma~\ref{lemma:internalstability} {is a simplified version of~\cite[Lemma 5.2]{zhou1996robust}} {because} we {focus on} strictly proper plants for simplicity. {Lemma~\ref{lemma:internalstability}} leads to an explicit state-space characterization of the set $\mathcal{C}_{\text{stab}}$ {as follows:}
\begin{equation}\label{eq:state_space_representation}
    \mathcal{C}_{\text{stab}} = \left\{\mathbf{K} \mid A_{\text{cl}} =\begin{bmatrix}
    A+BD_kC & BC_k\\
  B_kC & A_k
    \end{bmatrix} \text{is stable}\right\},
\end{equation}
where $\mathbf{K} = C_k(zI - A_k)^{-1}B_k + D_k$. Unfortunately, the stability condition on $A_{\text{cl}}$ in~\eqref{eq:state_space_representation} is still non-convex in terms of the parameters $A_k,B_k,C_k,D_k$.

\subsection{Characterizations based on transfer functions}

Unlike the state-space parameterization~\eqref{eq:state_space_representation}, there are several frequency-domain characterizations for $\mathcal{C}_{\text{stab}}$, where only convex constraints are involved in certain transfer functions. A classical approach is the celebrated \emph{Youla parameterization}~\cite{youla1976modern}, where a doubly-coprime factorization  of the plant is used.
 \begin{definition}
        A collection of stable transfer matrices, $\mathbf{U}_l, \mathbf{V}_l,\mathbf{N}_l,\mathbf{M}_l,\mathbf{U}_r, \mathbf{V}_r,\mathbf{N}_r,\mathbf{M}_r \in \mathcal{RH_{\infty}}$ is called a doubly-coprime factorization of $\mathbf{G}$ if
        $
            \mathbf{P}_{22} = \mathbf{N}_r\mathbf{M}_r^{-1} = \mathbf{M}_l^{-1}\mathbf{N}_l
        $
        and
        $$
            \begin{bmatrix} \mathbf{U}_l & -\mathbf{V}_l \\ -\mathbf{N}_l & \mathbf{M}_l\end{bmatrix}
            \begin{bmatrix} \mathbf{M}_r & \mathbf{V}_r \\ \mathbf{N}_r & \mathbf{U}_r\end{bmatrix} = I.
        $$
    \end{definition}
    Such doubly-coprime factorization can always be computed if 
    the state-space realization of $\mathbf{G}$ is stabilizable and detectable~\cite{nett1984connection}. We have the following equivalence~\cite{youla1976modern}
    \begin{equation} \label{eq:youla}
            \mathcal{C}_{\text{stab}} = \{\mathbf{K} = (\mathbf{V}_r - \mathbf{M}_r\mathbf{Q})(\mathbf{U}_r - \mathbf{N}_r\mathbf{Q})^{-1} \mid   \mathbf{Q} \in \mathcal{RH}_{\infty} \},
    \end{equation}
    where $\mathbf{Q}$ is denoted as the \emph{Youla parameter}. Note that the Youla parameter $\mathbf{Q}$ can be freely  chosen in $\mathcal{RH}_{\infty}$. We refer the interested reader to~\cite{youla1976modern,francis1987course,zhou1996robust} for more details on {the} Youla parameterization.

  Two recent approaches are the \emph{system-level parameterization} (SLP)~\cite{wang2019system} and {the} \emph{input-output parameterization} (IOP)~\cite{furieri2019input}, where no doubly-coprime factorization is required \emph{a priori}.
 Both the SLP and the IOP use certain closed-loop responses for parameterizing $\mathcal{C}_{\text{stab}}$ with the addition of a set of equality constraints.
 In particular, the SLP utilizes the closed-loop responses from $(\bm{\delta}_x,\bm{\delta}_y)$ to $(\mathbf{x},\mathbf{u})$, and the IOP relies on the closed-loop responses from $(\bm{\delta}_y,\bm{\delta}_u)$ to $(\mathbf{y},\mathbf{u})$. It is known that Youla parameterization, the SLP, and the IOP are equivalent to each other~\cite{zheng2019equivalence}. Inspired by these results~\cite{wang2019system,furieri2019input,zheng2019equivalence}, we aim to investigate the following questions:
 \begin{itemize}
    \item What are all the possible parameterizations for $\mathcal{C}_{\text{stab}}$ using closed-loop responses, beyond the SLP and the IOP? We present a complete and exclusive answer by examining all closed-loop responses from $(\bm{\delta}_x,\bm{\delta}_y,\bm{\delta}_u)$ to $(\mathbf{x},\mathbf{y},\mathbf{u})$ (Sections~\ref{Section:IOS} and~\ref{section:paramterization}).
   \item All the closed-loop parameterizations for $\mathcal{C}_{\text{stab}}$ are equivalent in theory. When applying FIR approximations for practical computation, are they still equivalent? What are the corresponding state-space realizations for controller implementation? Both of them are answered in Section~\ref{section:FIR}.
   \item Finally, floating-point  arithmetic computation and/or implementation always introduce small numerical mismatches for equality constraints. What are the impact of these numerical mismatches for closed-loop parameterizations? In Section~\ref{section:robustness}, we investigate a property termed as \textit{numerical robustness} of closed-loop parameterizations.
 \end{itemize}


\subsection{Optimal controller synthesis}
As a particular application, we consider the following optimal control problem
\begin{equation} \label{eq:H2_s2}
    \begin{aligned}
        \min_{\mathbf{K}}  \quad &  \lim_{T\rightarrow \infty }\mathbb{E}\left[\frac{1}{T}\sum_{t=0}^T\left(y_t^\tr Qy_t + u_t^\tr R u_t\right)\right] \\
        \text{s.t.} \quad & x[t+1] = Ax[t] +B(u[t] + \delta_u[t]), \\
        & \quad \,\;\; y[t] = Cx[t] + \delta_y[t], \\
        & \quad \;\,\;\;\;\; \mathbf{u} = \mathbf{K} \mathbf{y},
    \end{aligned}
\end{equation}
where $\delta_u[t] \sim \mathcal{N}(0,I), \delta_y[t] \sim \mathcal{N}(0,I), Q \succ 0$ and $R \succ 0$ are performance-weight matrices with compatible dimensions.
Problem~\eqref{eq:H2_s2} can be reformulated into a problem in the frequency domain (see e.g.,~\cite[Appendix G]{zheng2021sample})
    \begin{align}
        \min_{\mathbf{K}}  \quad &  \left\|\begin{bmatrix} Q^{\frac{1}{2}} & \\  & R^{\frac{1}{2}}  \end{bmatrix} \begin{bmatrix} (I-\mathbf{G}\mathbf{K})^{-1} & (I-\mathbf{G}\mathbf{K})^{-1}\mathbf{G} \\ \mathbf{K}(I-\mathbf{G}\mathbf{K})^{-1} & (I-\mathbf{K}\mathbf{G})^{-1} \end{bmatrix}\right\|^2_{\mathcal{H}_2} \nonumber\\
        \text{s.t.} \quad & \mathbf{K} \in \mathcal{C}_{\text{stab}}, \label{eq:OptimalControlv2}
    \end{align}
where $\mathbf{G} = C(zI - A)^{-1}B$.
It is easy to see that the optimal synthesis problem~\eqref{eq:OptimalControlv2} is non-convex in terms of $\mathbf{K}$ since both the cost function and constraint are non-convex. Instead of optimizing over the controller directly, we will establish a variety of equivalent convex reformulations based upon optimizing over closed-loop responses. Our results will reveal computational properties and numerical robustness of such reformulations based on closed-loop parameterizations.

We conclude this section by stating the following classical result, which will be frequently used.
\begin{lemma}[{\cite[Chapter 3]{zhou1996robust}}] \label{lemma:3}
Given a transfer matrix $
    \mathbf{T}(z) = C(zI - A)^{-1}B + D
$, we have
\begin{itemize}
\item If $(A,B,C)$ is detectable and stabilizable, then $\mathbf{T}(z) \in \mathcal{RH}_{\infty}$ if and only if $A$ is stable;
\item If $(A,B)$ is not stabilizable, or $(A,C)$ is not detectable, then the stability of $A$ is sufficient but not necessary for $ \mathbf{T}(z) \in \mathcal{RH}_{\infty}$.
\end{itemize}
\end{lemma}

\section{External transfer matrix characterization of internal stability} \label{Section:IOS}
In this section, we revisit the external transfer matrix characterization of internal stability,  which will be applied to characterize $\mathcal{C}_{\text{stab}}$ in the next section.

%
Combining~\eqref{eq:LTI} with~\eqref{eq:LTIController}, we can write the closed-loop responses from $(\bm{\delta}_x,\bm{\delta}_y,\bm{\delta}_u)$ to $(\mathbf{x},\mathbf{y},\mathbf{u})$  as
\begin{equation} \label{eq:closedloop}
        \begin{bmatrix}
            \mathbf{x} \\
            \mathbf{y} \\
            \mathbf{u}
        \end{bmatrix} = \begin{bmatrix}
            \bm{\Phi}_{xx} & \bm{\Phi}_{xy} & \bm{\Phi}_{xu} \\
             \bm{\Phi}_{yx} & \bm{\Phi}_{yy} & \bm{\Phi}_{yu}  \\
            \bm{\Phi}_{ux} & \bm{\Phi}_{uy} & \bm{\Phi}_{uu}
        \end{bmatrix}\begin{bmatrix}
            \bm{\delta}_x \\
            \bm{\delta}_y  \\
            \bm{\delta}_u
        \end{bmatrix},
\end{equation}
where we have $\bm{\Phi}_{xx} = (zI - A - B\mathbf{K}C)^{-1}$ and
\begin{equation} \label{eq:responses}
    \begin{aligned}
        \bm{\Phi}_{xy} &= \bm{\Phi}_{xx}B\mathbf{K},
        & \bm{\Phi}_{xu} &= \bm{\Phi}_{xx}B, \\ \bm{\Phi}_{yx} &= C\bm{\Phi}_{xx},  & \bm{\Phi}_{yy} &= C\bm{\Phi}_{xx}B\mathbf{K} + I,  \\ \bm{\Phi}_{yu} &= C\bm{\Phi}_{xx}B,
        & \bm{\Phi}_{ux} &= \mathbf{K}C\bm{\Phi}_{xx},  \\ \bm{\Phi}_{uy} &= \mathbf{K}(C\bm{\Phi}_{xx}B\mathbf{K} + I),   &\bm{\Phi}_{uu} &= \mathbf{K}C\bm{\Phi}_{xx}B + I.  \\
    \end{aligned}
\end{equation}
We define the closed-loop response transfer matrix as
\begin{equation}\label{eq:PHI}
    \bm{\Phi} :=  \begin{bmatrix}
            \bm{\Phi}_{xx} & \bm{\Phi}_{xy} & \bm{\Phi}_{xu} \\
             \bm{\Phi}_{yx} & \bm{\Phi}_{yy} & \bm{\Phi}_{yu}  \\
            \bm{\Phi}_{ux} & \bm{\Phi}_{uy} & \bm{\Phi}_{uu}
        \end{bmatrix}.
\end{equation}

A notion of external transfer matrix stability is defined as follows.
\begin{definition}[{\cite[Chapter 5]{zhou1996robust}}]
    The closed-loop system is {\emph{disturbance-to-signal stable}} if the closed-loop responses from $(\bm{\delta}_x,\bm{\delta}_y,\bm{\delta}_u)$ to $(\mathbf{x},\mathbf{y},\mathbf{u})$ are all stable, \emph{i.e.}, $\bm{\Phi} \in \mathcal{RH}_{\infty}$.
\end{definition}

\subsection{General plant}

Under Assumption~\ref{assumption:stabilizability}, it is known that the internal stability in Definition 1 and the {disturbance-to-signal} stability in Definition 2 are equivalent~\cite[Chapter 5]{zhou1996robust}, \emph{i.e.}, we have
\begin{equation}
     \mathcal{C}_{\text{stab}} = \{\mathbf{K} \mid \bm{\Phi} \in \mathcal{RH}_{\infty}, \, \text{with}\; \bm{\Phi} \;\text{defined in~\eqref{eq:PHI}} \}.
\end{equation}
In fact, it is sufficient to enforce a subset of elements in $\bm{\Phi}$ to be stable, as shown in~\cite[Lemma 5.3]{zhou1996robust}.
\begin{lemma}\label{lemma:equivalence}
    Under Assumption~\ref{assumption:stabilizability}, $\mathbf{K}$ \emph{internally stabilizes}  $\mathbf{G}$ if and only if the closed-loop responses from $(\bm{\delta}_y,\bm{\delta}_u)$ to $(\mathbf{y},\mathbf{u})$ are stable, \emph{i.e.},
    $$
        \begin{bmatrix}
          \bm{\Phi}_{yy} & \bm{\Phi}_{yu}  \\
            \bm{\Phi}_{uy} & \bm{\Phi}_{uu}
        \end{bmatrix} \in \mathcal{RH}_{\infty}.
    $$
\end{lemma}

For notational simplicity, we denote
$$
    \left(\begin{bmatrix}
            \bm{\delta}_y \\
            \bm{\delta}_u  \\
        \end{bmatrix} \rightarrow \begin{bmatrix}
            \mathbf{y} \\
            \mathbf{u} \\
        \end{bmatrix}\right) := \begin{bmatrix}
          \bm{\Phi}_{yy} & \bm{\Phi}_{yu}  \\
            \bm{\Phi}_{uy} & \bm{\Phi}_{uu}
        \end{bmatrix}.
$$
The result in Lemma~\ref{lemma:equivalence} motivates the question of whether we can select different {minimal sets of} elements in $\bm{\Phi}$ for internal stability. For example, if the closed-loop responses from $(\bm{\delta}_x, \bm{\delta}_y)$ to $(\mathbf{x}, \mathbf{y})$ are stable, \emph{i.e.},
$$
    \left(\begin{bmatrix}
            \bm{\delta}_x \\
            \bm{\delta}_y  \\
        \end{bmatrix} \rightarrow \begin{bmatrix}
            \mathbf{x} \\
            \mathbf{y} \\
        \end{bmatrix}\right) := \begin{bmatrix}
          \bm{\Phi}_{xx} & \bm{\Phi}_{xy}  \\
            \bm{\Phi}_{yx} & \bm{\Phi}_{yy}
        \end{bmatrix} \in \mathcal{RH}_{\infty},
$$
can we guarantee {that} the closed-loop system is internally stable? The answer is negative, as proved in Theorem~\ref{Theo:mainresult} below.

In particular, we consider all possible combinations of four closed-loop responses that may guarantee internal stability. When choosing two disturbances and two outputs from~\eqref{eq:closedloop}, we have in total  ${3 \choose 2} \times {3 \choose 2} = 9$ choices, \emph{i.e.},
\begin{equation} \label{eq:input-ouput}
     \begin{aligned}
        {\color{black}\left(\begin{bmatrix}
            \bm{\delta}_x \\
            \bm{\delta}_y  \\
        \end{bmatrix} \rightarrow \begin{bmatrix}
            \mathbf{x} \\
            \mathbf{y} \\
        \end{bmatrix}\right)}, &{\color{blue}\left(\begin{bmatrix}
            \bm{\delta}_x \\
            \bm{\delta}_y  \\
        \end{bmatrix} \rightarrow \begin{bmatrix}
            \mathbf{y} \\
            \mathbf{u} \\
        \end{bmatrix}\right)},  {\color{blue}\left(\begin{bmatrix}
            \bm{\delta}_x \\
            \bm{\delta}_y  \\
        \end{bmatrix} \rightarrow \begin{bmatrix}
            \mathbf{x} \\
            \mathbf{u} \\
        \end{bmatrix}\right)},  \\
         {\color{black}\left(\begin{bmatrix}
            \bm{\delta}_y \\
            \bm{\delta}_u  \\
        \end{bmatrix} \rightarrow \begin{bmatrix}
            \mathbf{x} \\
            \mathbf{y} \\
        \end{bmatrix}\right)}, &{\color{blue}\left(\begin{bmatrix}
            \bm{\delta}_y \\
            \bm{\delta}_u  \\
        \end{bmatrix} \rightarrow \begin{bmatrix}
            \mathbf{y} \\
            \mathbf{u} \\
        \end{bmatrix}\right)},  {\color{blue}\left(\begin{bmatrix}
            \bm{\delta}_y \\
            \bm{\delta}_u  \\
        \end{bmatrix} \rightarrow \begin{bmatrix}
            \mathbf{x} \\
            \mathbf{u} \\
        \end{bmatrix}\right)}, \\
         {\color{black}\left(\begin{bmatrix}
            \bm{\delta}_x \\
            \bm{\delta}_u  \\
        \end{bmatrix} \rightarrow \begin{bmatrix}
            \mathbf{x} \\
            \mathbf{y} \\
        \end{bmatrix}\right)}, &{\color{black}\left(\begin{bmatrix}
            \bm{\delta}_x \\
            \bm{\delta}_u  \\
        \end{bmatrix} \rightarrow \begin{bmatrix}
            \mathbf{y} \\
            \mathbf{u} \\
        \end{bmatrix}\right)},  {\color{black}\left(\begin{bmatrix}
            \bm{\delta}_x \\
            \bm{\delta}_u  \\
        \end{bmatrix} \rightarrow \begin{bmatrix}
            \mathbf{x} \\
            \mathbf{u} \\
        \end{bmatrix}\right)}.
    \end{aligned}
\end{equation}
Note that it is in general not sufficient to select less than four close-loop responses since there are two dynamical parts in system~\eqref{eq:LTI} and controller~\eqref{eq:state_space_controller}.
One main result of this section shows that the stability of any of the groups of four closed-loop responses in the top-right corner of~\eqref{eq:input-ouput}, highlighted in blue, is equivalent to internal stability.
\begin{theorem}\label{Theo:mainresult}
  Consider the LTI system~\eqref{eq:LTI}, evolving under a dynamic control policy~\eqref{eq:state_space_controller}. Under Assumption 1, the following statements are equivalent:
   \begin{enumerate}
   \setlength\itemsep{1mm}
       \item $\mathbf{K}$ internally stabilizes $\mathbf{G}$;
       \item $\left(\begin{bmatrix}
            \bm{\delta}_x \\
            \bm{\delta}_y  \\
        \end{bmatrix} \rightarrow \begin{bmatrix}
            \mathbf{y} \\
            \mathbf{u} \\
        \end{bmatrix}\right) \in \mathcal{RH}_{\infty}$;
        \item $\left(\begin{bmatrix}
            \bm{\delta}_x \\
            \bm{\delta}_y  \\
        \end{bmatrix} \rightarrow \begin{bmatrix}
            \mathbf{x} \\
            \mathbf{u} \\
        \end{bmatrix}\right)\in \mathcal{RH}_{\infty}$;

        \item $\left(\begin{bmatrix}
            \bm{\delta}_y \\
            \bm{\delta}_u  \\
        \end{bmatrix} \rightarrow \begin{bmatrix}
            \mathbf{y} \\
            \mathbf{u} \\
        \end{bmatrix}\right)\in \mathcal{RH}_{\infty}$;

        \item $\left(\begin{bmatrix}
            \bm{\delta}_y \\
            \bm{\delta}_u  \\
        \end{bmatrix} \rightarrow \begin{bmatrix}
            \mathbf{x} \\
            \mathbf{u} \\
        \end{bmatrix}\right)\in \mathcal{RH}_{\infty}$.
   \end{enumerate}
   {Moreover, the stability of any other group of four closed-loop responses in~\eqref{eq:input-ouput} is not sufficient for internal stability.}
\end{theorem}

\begin{proof}
The idea of our proof is to use a state-space representation of the closed-loop system, which is motivated by~\cite[Lemma 5.3]{zhou1996robust}.
From~\eqref{eq:LTI} and~\eqref{eq:state_space_controller}, we have
\begin{equation} \label{eq:statespace_s1}
    \begin{bmatrix} x[t+1] \\\xi[t+1] \end{bmatrix} \! =\!
    \begin{bmatrix} A & 0 \\ 0 & A_k \end{bmatrix}\!
    \begin{bmatrix} x[t] \\\xi[t] \end{bmatrix} \!+\!
    \begin{bmatrix} B & 0\\ 0 & B_k \end{bmatrix}\!
    \begin{bmatrix} u[t] \\y[t] \end{bmatrix} \!+\! \begin{bmatrix} \delta_x[t] \\0 \end{bmatrix},
\end{equation}
and
\begin{equation} \label{eq:statespace_s2}
    \begin{bmatrix} I & -D_k \\ 0 & I \end{bmatrix}\begin{bmatrix} u[t] \\y[t] \end{bmatrix} = \begin{bmatrix} 0 & C_k \\ C & 0 \end{bmatrix}\begin{bmatrix} x[t] \\\xi[t] \end{bmatrix} + \begin{bmatrix} \delta_u[t] \\\delta_y[t] \end{bmatrix}.
\end{equation}
Substituting~\eqref{eq:statespace_s2} into~\eqref{eq:statespace_s1} leads to
\begin{equation*} 
    \begin{bmatrix}
    x[t+1] \\
    \xi[t+1]
    \end{bmatrix} = A_{\text{cl}}\begin{bmatrix}
    x[t] \\
    \xi[t]
    \end{bmatrix}
     + \begin{bmatrix}
    B & BD_k \\
    0 & B_k
    \end{bmatrix} \begin{bmatrix}
    \delta_u[t] \\
    \delta_{y}[t]
    \end{bmatrix} + \begin{bmatrix}
    \delta_x[t] \\
    0
    \end{bmatrix},
    \end{equation*}
\emph{i.e.}, we have
\begin{equation} \label{eq:state_space_s3}
\begin{aligned}
    &\begin{bmatrix}
    x[t+1] \\
    \xi[t+1]
    \end{bmatrix} = A_{\text{cl}}\begin{bmatrix}
    x[t] \\
    \xi[t]
    \end{bmatrix}
    + \begin{bmatrix} I & BD_k & B \\
                     0  & B_k & 0
                     \end{bmatrix} \begin{bmatrix} \delta_x[t] \\ \delta_y[t] \\ \delta_u[t] \end{bmatrix}, \\
     &\begin{bmatrix} x[t] \\ y[t] \\ u[t] \end{bmatrix} = \begin{bmatrix}
    I & 0 \\
    C & 0 \\
    D_kC  & C_k
    \end{bmatrix}  \begin{bmatrix}
    x[t] \\
    \xi[t]
    \end{bmatrix}  + \begin{bmatrix}
    0 & 0 & 0 \\
    0 & I & 0\\
    0  & D_k & I
    \end{bmatrix}  \begin{bmatrix} \delta_x[t] \\ \delta_y[t] \\ \delta_u[t] \end{bmatrix}.
\end{aligned}
    \end{equation}
Therefore, the closed-loop responses from $(\bm{\delta}_x,\bm{\delta}_y,\bm{\delta}_u) \rightarrow (\mathbf{x},\mathbf{y},\mathbf{u})$ are
\begin{equation} \label{eq:CLresponses}
        \begin{bmatrix}
    I & 0 \\
    C & 0 \\
    D_kC  & C_k
    \end{bmatrix} (zI - A_{\text{cl}})^{-1}\begin{bmatrix} I & BD_k & B \\
                     0  & B_k & 0
                     \end{bmatrix}  + \begin{bmatrix}
    0 & 0 & 0 \\
    0 & I & 0\\
    0  & D_k & I
    \end{bmatrix},
\end{equation}
    from which, we get state-space realizations of the following closed-loop responses
    \begin{subequations} \label{eq:clstatespace}
    \begin{align}
    \left(\begin{bmatrix}
            \bm{\delta}_x \\
            \bm{\delta}_y  \\
        \end{bmatrix} \rightarrow \begin{bmatrix}
            \mathbf{y} \\
            \mathbf{u} \\
        \end{bmatrix}\right) &=  \hat{C}_1 (zI - A_{\text{cl}})^{-1}\hat{B}_1  + \begin{bmatrix}
    0 & I \\
    0  & D_k
    \end{bmatrix},\label{eq:clxy2yu}\\
    \left(\begin{bmatrix}
            \bm{\delta}_x \\
            \bm{\delta}_y  \\
        \end{bmatrix} \rightarrow \begin{bmatrix}
            \mathbf{x} \\
            \mathbf{u} \\
        \end{bmatrix}\right) &=  \hat{C}_2 (zI - A_{\text{cl}})^{-1}\hat{B}_1  + \begin{bmatrix}
    0 & 0 \\
    0  & D_k
    \end{bmatrix},\label{eq:clxy2xu}\\
    \left(\begin{bmatrix}
            \bm{\delta}_y \\
            \bm{\delta}_u  \\
        \end{bmatrix} \rightarrow \begin{bmatrix}
            \mathbf{y} \\
            \mathbf{u} \\
        \end{bmatrix}\right) &=  \hat{C}_1 (zI - A_{\text{cl}})^{-1}\hat{B}_2  + \begin{bmatrix}
    I & 0 \\
    D_k & I
    \end{bmatrix}, \label{eq:clyu2yu}
    \\
     \left(\begin{bmatrix}
            \bm{\delta}_y \\
            \bm{\delta}_u  \\
        \end{bmatrix} \rightarrow \begin{bmatrix}
            \mathbf{x} \\
            \mathbf{u} \\
        \end{bmatrix}\right) &=  \hat{C}_2 (zI - A_{\text{cl}})^{-1}\hat{B}_2  + \begin{bmatrix}
    0 & 0 \\
    D_k & I
    \end{bmatrix} \label{eq:clyu2xu},
    \end{align}
    \end{subequations}
where
\begin{equation} \label{eq:clmatrices}
\begin{aligned}
\hat{B}_1 &= \begin{bmatrix} I & BD_k  \\
    0  & B_k
    \end{bmatrix}, & \hat{B}_2 &= \begin{bmatrix} BD_k & B  \\
    B_k & 0
    \end{bmatrix}, \; \\
    \hat{C}_1 &= \begin{bmatrix} C & 0  \\
    D_kC  & C_k
    \end{bmatrix}, &\hat{C}_2 &= \begin{bmatrix} I  & 0  \\
    D_kC & C_k
    \end{bmatrix}.
\end{aligned}
\end{equation}

By Lemma 1, we know that $\mathbf{K}$ internally stabilizes $\mathbf{G}$ if and only if the closed-loop matrix $A_{\text{cl}}$ defined in~\eqref{eq:closed_loop_matrix} is stable. It is obvious true that $(1)\Rightarrow(2)$, $(1)\Rightarrow(3)$, $(1)\Rightarrow(4)$, and $(1)\Rightarrow(5)$.

Next, we prove if anyone of $(2)$ -- $(5)$ is true, the matrix $A_{\text{cl}}$ is stable. According to Lemma~\ref{lemma:3}, it remains to prove that the state-space realizations~\eqref{eq:clxy2yu}--\eqref{eq:clyu2xu} are all stabilizable and detectable. This is equivalent to {showing} that
$
(A_{\text{cl}},\hat{B}_1), (A_{\text{cl}},\hat{B}_2)
$
are stabilizable and that $
    (A_{\text{cl}},\hat{C}_1), (A_{\text{cl}},\hat{C}_2)
$ are detectable.
We let
$
 \hat{F}_1 = \begin{bmatrix} F & 0 \\ -C & F_k \end{bmatrix}
$
where $F$ and $F_k$ are chosen such that $A + F$ and $A_k + B_kF_k$ are stable (since $(A, B_k)$ is stabilizable).
Then, we have that
$$
 A_{\text{cl}} + \hat{B}_1\hat{F}_1 = \begin{bmatrix} A + F & BC_k+BD_kF_k \\
 0 & A_k + B_kF_k\end{bmatrix}
$$
is stable, and thus $(A_{\text{cl}}, \hat{B}_1)$ is stablizable. Similar arguments show that $(A_{\text{cl}},\hat{B}_2)$ is stabilizable, and $
    (A_{\text{cl}},\hat{C}_1)$, $(A_{\text{cl}},\hat{C}_2)
$ are detectable.

For the second part {of} Theorem~\ref{Theo:mainresult}, we first prove that the stability of  $\left(\begin{bmatrix}
            \bm{\delta}_x \\
            \bm{\delta}_y  \\
        \end{bmatrix} \rightarrow \begin{bmatrix}
            \mathbf{x} \\
            \mathbf{y} \\
        \end{bmatrix}\right)$ is not sufficient for internal stability.
        From~\eqref{eq:CLresponses}, a state-space realization of the closed-loop responses from $(\bm{\delta}_x,\bm{\delta}_y)$ to $(\mathbf{x},\mathbf{y})$ is
        \begin{equation*}
           \left(\begin{bmatrix}
            \bm{\delta}_x \\
            \bm{\delta}_y  \\
        \end{bmatrix} \rightarrow \begin{bmatrix}
            \mathbf{x} \\
            \mathbf{y} \\
        \end{bmatrix}\right)= \begin{bmatrix}
       I & 0\\
        C & 0
    \end{bmatrix}(zI - A_{\text{cl}})^{-1}\hat{B}_1 +\begin{bmatrix}
       0 & 0\\
       0 & I
    \end{bmatrix}.
        \end{equation*}
    Since
    $
    \left(A_{\text{cl}}, \begin{bmatrix}
       I & 0\\
        C & 0
    \end{bmatrix}\right)
    $
    is not detectable in general, the stability of $\left(\begin{bmatrix}
            \bm{\delta}_x \\
            \bm{\delta}_y  \\
        \end{bmatrix} \rightarrow \begin{bmatrix}
            \mathbf{x} \\
            \mathbf{y} \\
        \end{bmatrix}\right)$ cannot guarantee the stability of $A_{\text{cl}}$. Therefore, it  is not sufficient for internal stability either. The other claims can be proved in a similar way: the corresponding state-space realization of the closed-loop transfer matrix is not stabilizable {and/}or detectable.
\end{proof}

As shown in Theorem~\ref{Theo:mainresult}, to guarantee internal stability for general plants, {it is always required} to select $\bm{\delta}_y$ as an input and $\mathbf{u}$ as an output, leading to four possible groups of closed-loop responses. The groups of closed-loop responses in~\eqref{eq:input-ouput}, except those in blue, do not have either $\bm{\delta}_y$ or $\mathbf{u}$, and thus fail to guarantee internal stability.
%
Note that Theorem~\ref{Theo:mainresult} is exclusive in the sense that there exist no other combinations of stable closed-loop responses that are equivalent to internal stability, and  Lemma~\ref{lemma:equivalence} is included as the equivalence between 1) and 4) in Theorem~\ref{Theo:mainresult}. Thus, Theorem~\ref{Theo:mainresult} offers a complete picture between internal stability and stable closed-loop responses.

\subsection{Two special cases: Stable plants and State feedback }

Here, we show that the transfer matrix characterization of internal stability can be simplified for special cases: 1) open-loop stable plants; 2) the state feedback case. To guarantee internal stability, instead of considering four closed-loop responses in Theorem~\ref{Theo:mainresult}, the stability of one particular closed-loop response is sufficient in the case of open-loop stable plants, and the stability of two particular closed-loop responses is sufficient in the state feedback case.

The following result is classical, which is the same as~\cite[Corollary 5.5]{zhou1996robust}. For completeness, we provide a proof from a state-space perspective.

\begin{corollary} \label{prop:stable}
    Consider the LTI system~\eqref{eq:LTI}, evolving under a dynamic control policy~\eqref{eq:state_space_controller}. If the LTI system is open-loop stable (i.e., $A$ is stable), then $\mathbf{K} \in \mathcal{C}_{\text{stab}}$ if and only if $(\bm{\delta}_y \rightarrow \mathbf{u}):= \bm{\Phi}_{uy} \in \mathcal{RH}_{\infty}$.
\end{corollary}
\begin{proof}
    The ``only if'' direction is true by definition. We now prove the sufficiency. From~\eqref{eq:CLresponses}, we have
    $$
        \bm{\Phi}_{uy} = \begin{bmatrix}
       D_kC  & C_k
    \end{bmatrix} (zI - A_{\text{cl}})^{-1}\begin{bmatrix}  BD_k \\
     B_k
                     \end{bmatrix}  + D_k.
    $$
    Considering the fact that the following matrix
    $$
        A_{\text{cl}} + \begin{bmatrix}  BD_k \\
     B_k
                     \end{bmatrix}\begin{bmatrix}  -C &   F_k
                     \end{bmatrix} = \begin{bmatrix}
                     A & BC_k+BD_kF_k \\ 0 & A_k + B_kF_k
                     \end{bmatrix},
    $$
    is stable when $A$ and $A_k +B_kF_k$ are stable, we know that $\left(A_{\text{cl}},\begin{bmatrix}  BD_k \\
     B_k
                     \end{bmatrix}\right)$ is stabilizable. Similarly, we can show that $\left(A_{\text{cl}},\begin{bmatrix}  D_kC &
     C_k                \end{bmatrix}\right)$ is detectable. Therefore, if $\bm{\Phi}_{uy} \in \mathcal{RH}_{\infty}$, we have $A_{\text{cl}}$ is stable, meaning that $\mathbf{K} \in \mathcal{C}_{\text{stab}}$. 
\end{proof}

In the state feedback case, we have the following result.
\begin{corollary} \label{prop:state}
    Consider the LTI system~\eqref{eq:LTI}, evolving under a dynamic control policy~\eqref{eq:state_space_controller}. If $C = I$, then $\mathbf{K} \in \mathcal{C}_{\text{stab}}$ if and only if $\left(\bm{\delta}_x \rightarrow \begin{bmatrix} \mathbf{x} \\\mathbf{u} \end{bmatrix}\right):=\begin{bmatrix}\bm{\Phi}_{xx} \\ \bm{\Phi}_{ux} \end{bmatrix} \in \mathcal{RH}_{\infty}$.
\end{corollary}
\begin{proof}
    When $C = I$, from~\eqref{eq:CLresponses}, we have
    $$
        \begin{bmatrix}\bm{\Phi}_{xx} \\ \bm{\Phi}_{ux} \end{bmatrix} = \begin{bmatrix} I & 0\\
       D_k & C_k
    \end{bmatrix} \left(zI -  \begin{bmatrix} A + BD_k & BC_k \\
        B_k &A_k
        \end{bmatrix}\right)^{-1}\begin{bmatrix}  I \\
     0
                     \end{bmatrix},
    $$
    From the proof of Theorem~\ref{Theo:mainresult}, we know
    $$
        \left(\begin{bmatrix} A + BD_k & BC_k \\
        B_k &A_k
        \end{bmatrix}, \begin{bmatrix} I & 0\\
       D_k & C_k
    \end{bmatrix}\right)
    $$
    is detectable. It is not difficult to prove that
    \begin{equation} \label{eq:statecontrollable}
        \left(\begin{bmatrix} A + BD_k & BC_k \\
        B_k &A_k
        \end{bmatrix}, \begin{bmatrix} I \\
       0
    \end{bmatrix}\right)
    \end{equation}
    is stabilizable (see Appendix~\ref{app:statefeedback} for details).
    Therefore, if $\bm{\Phi}_{xx} \in \mathcal{RH}_{\infty}, \bm{\Phi}_{ux}  \in \mathcal{RH}_{\infty}$, we have $A_{\text{cl}}$ is stable, meaning that $\mathbf{K} \in \mathcal{C}_{\text{stab}}$.
\end{proof}

The result in Corollary~\ref{prop:state} has been used in the state feedback case of the system-level parametrization~\cite{wang2019system}. The proof in~\cite{wang2019system} used a frequency-based method. Here, we provided an alternative proof from a state-space perspective, which is consistent with the proofs for Theorem~\ref{Theo:mainresult} and Corollary~\ref{prop:stable}.

\section{Parameterizations of stabilizing controllers} \label{section:paramterization}
The results in Theorem~\ref{Theo:mainresult} can be used to parameterize the set of internally stabilizing controllers $\mathcal{C}_{\text{stab}}$, leading to four equivalent parameterizations. One of them corresponds to the SLP~\cite{wang2019system}, and another one is the IOP~\cite{furieri2019input}. The remaining two parameterizations are new and, to the best of the authors' knowledge, have not been characterized before. The results in Corollaries~\ref{prop:stable} and~\ref{prop:state} can also be used to parameterize $\mathcal{C}_{\text{stab}}$ in a simplified way.

\subsection{Four equivalent parameterizations for general plants}
The closed-loop responses from $(\bm{\delta}_x,\bm{\delta}_y)$ to $(\mathbf{x},\mathbf{u})$ {have} been utilized in the SLP~\cite{wang2019system}. Specifically, consider
\begin{equation} \label{eq:slp1}
        \begin{bmatrix}
            \mathbf{x} \\
            \mathbf{u}
        \end{bmatrix} = \begin{bmatrix}
            \bm{\Phi}_{xx} & \bm{\Phi}_{xy} \\
            \bm{\Phi}_{ux} & \bm{\Phi}_{uy}
        \end{bmatrix}\begin{bmatrix}
            \bm{\delta}_x \\
            \bm{\delta}_y  \\
        \end{bmatrix}.
\end{equation}
We have the following system-level parameterization (SLP).
\begin{proposition}[SLP{~\cite[Theorem 2]{wang2019system}}]\label{prop:slp1}
     Consider the LTI system~\eqref{eq:LTI}, evolving under a dynamic control policy~\eqref{eq:state_space_controller}. The following {statements} are true:
    \begin{enumerate}
        \item For any $\mathbf{K} \in \mathcal{C}_{\text{{\emph{stab}}}}$, the resulting closed-loop responses~\eqref{eq:slp1} are in the following affine subspace

        \begin{equation}\label{eq:slp1_constraint}
    \begin{aligned}
    \begin{bmatrix} zI - A& - B \end{bmatrix}\begin{bmatrix}
            \bm{\Phi}_{xx} & \bm{\Phi}_{xy} \\
            \bm{\Phi}_{ux} & \bm{\Phi}_{uy}
        \end{bmatrix} &= \begin{bmatrix} I & 0 \end{bmatrix}, \\
            \begin{bmatrix}
            \bm{\Phi}_{xx} & \bm{\Phi}_{xy} \\
            \bm{\Phi}_{ux} & \bm{\Phi}_{uy}
        \end{bmatrix}\begin{bmatrix} zI - A\\ - C \end{bmatrix} &= \begin{bmatrix} I \\ 0 \end{bmatrix}, \\
           \bm{\Phi}_{xx}, \bm{\Phi}_{ux},
          \bm{\Phi}_{xy},\bm{\Phi}_{uy}
        &\in \mathcal{RH}_{\infty}.
        \end{aligned}
\end{equation}
        \item For any transfer matrices $ \bm{\Phi}_{xx}, \bm{\Phi}_{ux},
            \bm{\Phi}_{xy}, \bm{\Phi}_{uy}$ satisfying~\eqref{eq:slp1_constraint}, $\mathbf{K} = \bm{\Phi}_{uy} - \bm{\Phi}_{ux}\bm{\Phi}_{xx}^{-1}\bm{\Phi}_{xy} \in \mathcal{C}_{\text{{\emph{stab}}}}$.
    \end{enumerate}
\end{proposition}
We refer to $\mathbf{K} = \bm{\Phi}_{uy} - \bm{\Phi}_{ux}\bm{\Phi}_{xx}^{-1}\bm{\Phi}_{xy}$ as the four-block SLP controller. Also, the closed-loop responses from $(\bm{\delta}_y,\bm{\delta}_u)$ to $(\mathbf{y},\mathbf{u})$ {have} been used in the IOP~\cite{furieri2019input}. Specifically, consider
\begin{equation} \label{eq:slp2}
        \begin{bmatrix}
            \mathbf{y} \\
            \mathbf{u}
        \end{bmatrix} = \begin{bmatrix}
            \bm{\Phi}_{yy} & \bm{\Phi}_{yu} \\
            \bm{\Phi}_{uy} & \bm{\Phi}_{uu}
        \end{bmatrix}\begin{bmatrix}
            \bm{\delta}_y \\
            \bm{\delta}_u  \\
        \end{bmatrix}\,.
\end{equation}
We have the following input-output parameterization~(IOP).
\begin{proposition}[IOP{~\cite[Theorem 1]{furieri2019input}}] \label{prop:slp2}
    Consider the LTI system~\eqref{eq:LTI}, evolving under a dynamic control policy~\eqref{eq:state_space_controller}. The following {statements} are true:
    \begin{enumerate}
        \item For any $\mathbf{K} \in \mathcal{C}_{{\text{\emph{stab}}}}$, the resulting closed-loop responses~\eqref{eq:slp2} are in the following affine subspace
        \begin{equation}\label{eq:iop}
    \begin{aligned}
    \begin{bmatrix} I & -\mathbf{G} \end{bmatrix}\begin{bmatrix}
           \bm{\Phi}_{yy} & \bm{\Phi}_{yu} \\
            \bm{\Phi}_{uy} & \bm{\Phi}_{uu}
        \end{bmatrix} &= \begin{bmatrix} I & 0 \end{bmatrix}, \\
            \begin{bmatrix}
            \bm{\Phi}_{yy} & \bm{\Phi}_{yu} \\
            \bm{\Phi}_{uy} & \bm{\Phi}_{uu}
        \end{bmatrix}\begin{bmatrix}  -\mathbf{G} \\I \end{bmatrix} &= \begin{bmatrix} 0 \\ I\end{bmatrix}, \\
       \bm{\Phi}_{yy}, \bm{\Phi}_{uy},
            \bm{\Phi}_{yu}, \bm{\Phi}_{uu} &\in \mathcal{RH}_{\infty}.
        \end{aligned}
\end{equation}
        \item For any transfer matrices $ \bm{\Phi}_{yy}, \bm{\Phi}_{uy},
            \bm{\Phi}_{yu}, \bm{\Phi}_{uu}$ satisfying~\eqref{eq:iop}, $\mathbf{K} = \bm{\Phi}_{uy}\bm{\Phi}_{yy}^{-1} \in \mathcal{C}_{{\text{\emph{stab}}}}$.
    \end{enumerate}
\end{proposition}

Next, we consider the following closed-loop responses
\begin{equation}\label{eq:slp3}
        \begin{bmatrix}
            \mathbf{y} \\
            \mathbf{u} \\
        \end{bmatrix} =   \begin{bmatrix}
            \bm{\Phi}_{yx}  & \bm{\Phi}_{yy}\\
            \bm{\Phi}_{ux}  & \bm{\Phi}_{uy}\\
        \end{bmatrix} \begin{bmatrix}
            \bm{\delta}_x \\
            \bm{\delta}_y  \\
        \end{bmatrix}.
\end{equation}
We have 
a new parametrization of $\mathcal{C}_{\text{stab}}$. 
\begin{proposition}[Mixed I]\label{prop:slp3}
     Consider the LTI system~\eqref{eq:LTI}, evolving under a dynamic control policy~\eqref{eq:state_space_controller}. The following {statements} are true:
    \begin{enumerate}
        \item For any $\mathbf{K} \in \mathcal{C}_{{\text{\emph{stab}}}}$, the resulting closed-loop responses~\eqref{eq:slp3} are in the following affine subspace
        \begin{equation} \label{eq:slp3constraint}
             \begin{aligned}
           \begin{bmatrix} I & - \mathbf{G} \end{bmatrix} \begin{bmatrix}
            \bm{\Phi}_{yx}  & \bm{\Phi}_{yy}\\
            \bm{\Phi}_{ux}  & \bm{\Phi}_{uy}\\
        \end{bmatrix} &= {\begin{bmatrix} C(zI - A)^{-1}  &  I\end{bmatrix}},\\
        \begin{bmatrix}
            \bm{\Phi}_{yx}  & \bm{\Phi}_{yy}\\
            \bm{\Phi}_{ux}  & \bm{\Phi}_{uy}\\
        \end{bmatrix}  \begin{bmatrix} zI - A \\ -C \end{bmatrix} &= 0,
        \\
            \bm{\Phi}_{yx}, \bm{\Phi}_{ux},
            \bm{\Phi}_{yy}, \bm{\Phi}_{uy} &\in \mathcal{RH}_{\infty}.
            \end{aligned}
        \end{equation}
        \item For any transfer matrices $ \bm{\Phi}_{yx}, \bm{\Phi}_{ux},
            \bm{\Phi}_{yy}, \bm{\Phi}_{uy}$ satisfying~\eqref{eq:slp3constraint}, $\mathbf{K} = \bm{\Phi}_{uy}\bm{\Phi}_{yy}^{-1} \in \mathcal{C}_{{\text{\emph{stab}}}}$.
    \end{enumerate}
\end{proposition}

The proof is provided in Appendix~\ref{App:proposition3}.
Finally, we consider {the} case
\begin{equation}\label{eq:slp4}
        \begin{bmatrix}
            \mathbf{x} \\
            \mathbf{u} \\
        \end{bmatrix} =   \begin{bmatrix}
            \bm{\Phi}_{xy}  & \bm{\Phi}_{xu}\\
            \bm{\Phi}_{uy}  & \bm{\Phi}_{uu}\\
        \end{bmatrix} \begin{bmatrix}
            \bm{\delta}_y \\
            \bm{\delta}_u  \\
        \end{bmatrix}.
\end{equation}
The following result {mirrors} Proposition~\ref{prop:slp3} {for an additional new parametrization of $\mathcal{C}_{\text{stab}}$}. 
\begin{proposition}[Mixed II] \label{prop:slp4}
     Consider the LTI system~\eqref{eq:LTI}, evolving under a dynamic control policy~\eqref{eq:state_space_controller}. The following {statements} are true:
    \begin{enumerate}
        \item For any $\mathbf{K} \in \mathcal{C}_{{\text{\emph{stab}}}}$, the resulting closed-loop responses~\eqref{eq:slp4} are in the following affine subspace
        \begin{equation} \label{eq:slp4constraint}
             \begin{aligned}
           \begin{bmatrix} zI - A & - B \end{bmatrix} \begin{bmatrix}
            \bm{\Phi}_{xy}  & \bm{\Phi}_{xu}\\
            \bm{\Phi}_{uy}  & \bm{\Phi}_{uu}\\
        \end{bmatrix}  &= 0,\\
       \begin{bmatrix}
            \bm{\Phi}_{xy}  & \bm{\Phi}_{xu}\\
            \bm{\Phi}_{uy}  & \bm{\Phi}_{uu}\\
        \end{bmatrix} \begin{bmatrix} -\mathbf{G}\\ I \end{bmatrix} &= \begin{bmatrix}
            (zI - A)^{-1}B  \\
            I
        \end{bmatrix},
        \\
            \bm{\Phi}_{xy}, \bm{\Phi}_{uy},
            \bm{\Phi}_{xu}, \bm{\Phi}_{uu} &\in \mathcal{RH}_{\infty}.
            \end{aligned}
        \end{equation}
        \item For any transfer matrices $ \bm{\Phi}_{xy}, \bm{\Phi}_{uy},
            \bm{\Phi}_{xu}, \bm{\Phi}_{uu}$ satisfying~\eqref{eq:slp4constraint}, $\mathbf{K} = \bm{\Phi}_{uu}^{-1}\bm{\Phi}_{uy} \in \mathcal{C}_{{\text{\emph{stab}}}}$.
    \end{enumerate}
\end{proposition}
The proof of Proposition~\ref{prop:slp4} is similar to that of Proposition~\ref{prop:slp3}, 
which is provided in Appendix~\ref{app:proof_p4}  for completeness.

To summarize, Propositions~\ref{prop:slp1}--~\ref{prop:slp4} {establish four} equivalent methods to parameterize the set of internally stabilizing controllers using closed-loop responses:
 \begin{equation*}
        \begin{aligned}
            \mathcal{C}_{\text{stab}}  &= \{\mathbf{K} = \bm{\Phi}_{uy} - \bm{\Phi}_{ux}\bm{\Phi}_{xx}^{-1}\bm{\Phi}_{xy}  \mid  \bm{\Phi}_{xx},  \bm{\Phi}_{ux},
              \bm{\Phi}_{xy},  \bm{\Phi}_{uy}\, \\
        & \qquad \qquad \qquad \qquad \quad       \text{are in the affine subspace~\eqref{eq:slp1_constraint}}   \}, \\
            \mathcal{C}_{\text{stab}}  &= \{\mathbf{K} = \bm{\Phi}_{uy}\bm{\Phi}_{yy}^{-1}  \mid \bm{\Phi}_{yy}, \bm{\Phi}_{uy}, \bm{\Phi}_{yu},  \bm{\Phi}_{uu} \;  \\
          & \qquad \qquad \qquad \qquad \quad      \text{are in the affine subspace~\eqref{eq:iop}}   \}, \\
            \mathcal{C}_{\text{stab}}  &= \{\mathbf{K} = \bm{\Phi}_{uy}\bm{\Phi}_{yy}^{-1}  \mid\bm{\Phi}_{yx}, \bm{\Phi}_{ux},
            \bm{\Phi}_{yy} , \bm{\Phi}_{uy} \;  \\
          & \qquad \qquad \qquad \qquad \quad      \text{are in the affine subspace~\eqref{eq:slp3constraint}}   \}, \\
            \mathcal{C}_{\text{stab}}  &= \{\mathbf{K} = \bm{\Phi}_{uu}^{-1}\bm{\Phi}_{uy}  \mid \bm{\Phi}_{xy}, \bm{\Phi}_{uy},
            \bm{\Phi}_{xu} , \bm{\Phi}_{uu} \; \\
         & \qquad \qquad \qquad \qquad \quad       \text{are in the affine subspace~\eqref{eq:slp4constraint}}   \}.
        \end{aligned}
    \end{equation*}
Unlike the state-space characterization~\eqref{eq:state_space_representation}, the constraints~\eqref{eq:slp1_constraint},~\eqref{eq:iop},~\eqref{eq:slp3constraint}, and~\eqref{eq:slp4constraint} are all affine in the new parameters. Based on~\eqref{eq:slp1_constraint},~\eqref{eq:iop},~\eqref{eq:slp3constraint}, and~\eqref{eq:slp4constraint}, convex optimization problems can be derived for the classical optimal controller synthesis; see~\cite{wang2019system,furieri2019input, zheng2019equivalence} for details. We will present a case study in Section~\ref{section:example}.

\begin{remark}[Equivalence with Youla]
    The explicit~equ-ivalence between Propositions~\ref{prop:slp1} \&~\ref{prop:slp2} and the Youla parameterization has been derived in~{\cite{zheng2019equivalence}}. It is not difficult to derive {the} explicit relationship {between} Propositions~\ref{prop:slp3} \& \ref{prop:slp4} and the Youla parameterization~\eqref{eq:youla} using the approach of~{\cite{furieri2019input,zheng2019equivalence}}. While there are four parameters in~\eqref{eq:slp1_constraint},~\eqref{eq:iop},~\eqref{eq:slp3constraint}, or~\eqref{eq:slp4constraint}, there is only one freedom due to the affine constraints. This is consistent with the Youla parameterization, where only one parameter is involved with no explicit affine constraints. In Proposition~\ref{prop:clpyoula}, we will show that any doubly-coprime factorization of the plant can exactly eliminate the affine constraints~\eqref{eq:slp1_constraint},~\eqref{eq:iop},~\eqref{eq:slp3constraint}, and~\eqref{eq:slp4constraint}. 
\end{remark}

\begin{remark}[Numerical computation]
   \label{remark:numerical}
We note that while being convex, the decision variables in~\eqref{eq:slp1_constraint},~\eqref{eq:iop},~\eqref{eq:slp3constraint}, and~\eqref{eq:slp4constraint} are infinite-dimensional. Thus, finite-dimensional approximations are usually needed for numerical computations, which will be discussed in Section~\ref{section:FIR}. However, the affine constraints~\eqref{eq:slp1_constraint},~\eqref{eq:iop},~\eqref{eq:slp3constraint}, and~\eqref{eq:slp4constraint} can never be exactly satisfied in numerical computation. Section~\ref{section:robustness} will formally discuss the issue of numerical robustness.

\end{remark}
\subsection{Two special cases: stable plants and state feedback}
The results in Corollaries~\ref{prop:stable} and~\ref{prop:state} can be exploited to derive simplified versions of Propositions~\ref{prop:slp1}--\ref{prop:slp4}. 
We will later show that these simplified parametrizations enjoy provable numerical robustness. When the plant is open-loop stable, the IOP (Proposition~\ref{prop:slp2}) and the Mixed I (Proposition~\ref{prop:slp3}) are simplified as follows.

\begin{corollary}\label{coro:stablepara}
    Consider the LTI system~\eqref{eq:LTI}, evolving under a dynamic controller policy~\eqref{eq:state_space_controller}. If the LTI system is open-loop stable, then we have
    \begin{equation*}
        \mathcal{C}_{\text{{\emph{stab}}}} \! =\! \left\{\mathbf{K}= \bm{\Phi}_{uy}\bm{\Phi}_{yy}^{-1} \left| \begin{bmatrix} I & - \mathbf{G} \end{bmatrix}\begin{bmatrix} \bm{\Phi}_{yy} \\ \bm{\Phi}_{uy} \end{bmatrix}  = I, \right.\! \bm{\Phi}_{uy} \!\in\! \mathcal{RH}_{\infty}\right\}.
    \end{equation*}
\end{corollary}
 This result is consistent with the classical one in~\cite[Theorem 12.7]{zhou1996robust}. Note that for open-loop stable plants, the transfer matrix $\bm{\Phi}_{uy}$ from the measurement disturbance $\bm{\delta}_y$ to the control input $\mathbf{u}$ is the same as the Youla parameter $\mathbf{Q}$. Under the condition in Corollary~\ref{coro:stablepara}, the Mixed II (Proposition~\ref{prop:slp4}) can be simplified as well:
 \begin{equation*}
        \mathcal{C}_{\text{{stab}}}\! = \!\left\{ \! \mathbf{K}= \bm{\Phi}_{uu}^{-1}\bm{\Phi}_{uy} \left| \begin{bmatrix} \bm{\Phi}_{uy} & \bm{\Phi}_{uu} \end{bmatrix}\!\!\begin{bmatrix} - \mathbf{G} \\ I \end{bmatrix}  = I, \right. \! \! \bm{\Phi}_{uy} \!\in \!\mathcal{RH}_{\infty}\! \right\}.
    \end{equation*}

If the state is directly measurable for control, \emph{i.e.}, $C=I$, Corollary~\ref{prop:state} leads to the following simplified version of SLP.
\begin{corollary}[{\cite[Theorem 1]{wang2019system}}]\label{coro:statepara}
    Consider the LTI system~\eqref{eq:LTI}, evolving under a dynamic controller policy~\eqref{eq:state_space_controller}. If $C = I$, then we have
    \begin{equation*}
    \begin{aligned}
        \mathcal{C}_{\text{\emph{stab}}} = \bigg\{\mathbf{K}= \mathbf{\Phi}_{ux}\mathbf{\Phi}_{xx}^{-1} &\left| (zI -A)\mathbf{\Phi}_{xx} - B\mathbf{\Phi}_{ux} = I, \right.
        \\ &\qquad \qquad \mathbf{\Phi}_{ux}, \mathbf{\Phi}_{xx} \in \frac{1}{z}\mathcal{RH}_{\infty}\bigg\}.
    \end{aligned}
    \end{equation*}
\end{corollary}

{Note that the simplified IOP/Mixed I/Mixed II requires the stability of only one parameter, while the simplified SLP requires the stability of two parameters.} The proofs for Corollaries~\ref{coro:stablepara} and~\ref{coro:statepara} are similar to that of Proposition~\ref{prop:slp3}.

\section{Numerical computation and controller implementation}
\label{section:FIR}
This section investigates the numerical computation and controller implementation using the closed-loop parameterization for $\mathcal{C}_{\text{stab}}$. As noted in Remark~\ref{remark:numerical}, since the decision variables in the affine constraints~\eqref{eq:slp1_constraint},~\eqref{eq:iop},~\eqref{eq:slp3constraint}, and~\eqref{eq:slp4constraint} are infinite-dimensional, it is not immediately obvious to derive efficient numerical computation to search over the feasible region. One practical method is to apply the finite impulse response (FIR) approximation, which is extensively used in~\cite{wang2019system,furieri2019input}. As we will see, the SLP, the IOP and the two new mixed parametrizations are not equivalent to each other after imposing FIR constraints. In this section, we also present standard state-space realizations~\eqref{eq:state_space_controller} for the controllers using closed-loop responses.

\subsection{Numerical computation via FIR}

We denote the space of finite impulse response (FIR) transfer matrices with horizon $T$ as
$$
\mathcal{F}_T : = \left\{\mathbf{H} \in \mathcal{RH}_{\infty} \left|~~ \mathbf{H} = \sum_{k=0}^T \frac{1}{z^k}H_k\right.\right\},
$$
where $H_k$ denotes the $i$-th spectral component of the FIR transfer matrix $\mathbf{H}$. It is known that on letting the FIR length $T$ go to infinity, $\mathcal{F}_T$ converges to the space $\mathcal{RH}_{\infty}$~\cite[Theorem 4.7]{pohl2009advanced}.
It is not difficult to check that after imposing the decision variables to be FIR transfer matrices of horizon $T$, the constraints~\eqref{eq:slp1_constraint},~\eqref{eq:iop},~\eqref{eq:slp3constraint}, and~\eqref{eq:slp4constraint} all become finite-dimensional affine constraints in terms of the spectral components of the closed-loop responses. Specifically, the constraints are obtained by matching the coefficients associated with the terms $z^{-k}$. Thus, searching {for} an internally stabilizing controller only requires solving a linear program (LP) under the FIR assumption\footnote{Depending on the choice of the cost function, optimal controller synthesis may be cast as a quadratic program (QP) under the FIR assumption; see a case study in Section~\ref{section:example}.}.

Here, we show that imposing the FIR assumption has different effects depending on the chosen closed-loop parametrization.
\begin{theorem} \label{theo:FIR}
Given the LTI system~\eqref{eq:LTI}, evolving under a dynamic control policy~\eqref{eq:state_space_controller}, we consider the statements:
 \begin{enumerate}[label=(\roman*)]
   \setlength\itemsep{1mm}
\item $\bm{\Phi} \in \mathcal{F}_T$;
        \item $\left(\begin{bmatrix}
            \bm{\delta}_x \\
            \bm{\delta}_y  \\
        \end{bmatrix} \rightarrow \begin{bmatrix}
            \mathbf{x} \\
            \mathbf{u} \\
        \end{bmatrix}\right) \in \mathcal{F}_T; \hfill \text{(SLP)} \qquad $

        \item $\left(\begin{bmatrix}
            \bm{\delta}_x \\
            \bm{\delta}_y  \\
        \end{bmatrix} \rightarrow \begin{bmatrix}
            \mathbf{y} \\
            \mathbf{u} \\
        \end{bmatrix}\right) \in \mathcal{F}_T; \hfill \text{(Mixed I)} \qquad $

        \item $\left(\begin{bmatrix}
            \bm{\delta}_y \\
            \bm{\delta}_u  \\
        \end{bmatrix} \rightarrow \begin{bmatrix}
            \mathbf{x} \\
            \mathbf{u} \\
        \end{bmatrix}\right) \in \mathcal{F}_T;  \hfill \text{(Mixed II)} \qquad$

        \item $\left(\begin{bmatrix}
            \bm{\delta}_y \\
            \bm{\delta}_u  \\
        \end{bmatrix} \rightarrow \begin{bmatrix}
            \mathbf{y} \\
            \mathbf{u} \\
        \end{bmatrix}\right) \in \mathcal{F}_T.  \hfill \text{(IOP) \qquad}$

   \end{enumerate}
  If $(A, B, C)$ and $(A_k, B_k, C_k)$ are both stabilizable and detectable, we have  $(i) \Leftrightarrow (ii)\Rightarrow (iii)\Rightarrow (v)$
      and $(i) \Leftrightarrow (ii) \Rightarrow (iv)\Rightarrow (v)$. In addition, if $(A, B, C)$ and $(A_k, B_k, C_k)$ are both controllable and observable, we have $(i) \Leftrightarrow (ii) \Leftrightarrow (iii) \Leftrightarrow (iv)\Leftrightarrow (v)$.
\end{theorem}
The proof is not mathematically involved, and we provide it in Appendix~\ref{App:theorem2}. Note that minimal state-space realizations of the plant and the controller deserve more investigations in distributed control when particular structures are required; see~\cite{rantzer2019realizability,vamsi2015optimal} for details.

Upon defining the following sets
         \begin{equation*}
        \begin{aligned}
            \mathcal{C}_{\text{SLP}} \! &= \!\{\mathbf{K} = \bm{\Phi}_{uy} - \bm{\Phi}_{ux}\bm{\Phi}_{xx}^{-1}\bm{\Phi}_{xy}  \mid  \bm{\Phi}_{xx},  \bm{\Phi}_{ux},
              \bm{\Phi}_{xy},  \bm{\Phi}_{uy} \in \mathcal{F}_T\, \\
              & \qquad \qquad \qquad \qquad \qquad \text{are in the affine subspace~\eqref{eq:slp1_constraint}}   \},\\
            \mathcal{C}_{\text{M1}}  \!&= \!\{\mathbf{K} = \bm{\Phi}_{uy}\bm{\Phi}_{yy}^{-1}  \mid\bm{\Phi}_{yx}, \bm{\Phi}_{yy},
            \bm{\Phi}_{ux} , \bm{\Phi}_{uy} \in \mathcal{F}_T \; \\
              & \qquad \qquad \qquad \qquad \qquad
            \text{are in the affine subspace~\eqref{eq:slp3constraint}}   \},\\
            \mathcal{C}_{{\text{M2}}} \! &= \!\{\mathbf{K} = \bm{\Phi}_{uu}^{-1}\bm{\Phi}_{uy}  \mid \bm{\Phi}_{xy}, \bm{\Phi}_{uy},
            \bm{\Phi}_{xu} , \bm{\Phi}_{uu} \in \mathcal{F}_T \;  \\
              & \qquad \qquad \qquad \qquad \qquad
            \text{are in the affine subspace~\eqref{eq:slp4constraint}}   \},\\
            \mathcal{C}_{\text{IOP}} \! &= \!\{\mathbf{K} = \bm{\Phi}_{uy}\bm{\Phi}_{yy}^{-1}  \mid \bm{\Phi}_{yy}, \bm{\Phi}_{uy}, \bm{\Phi}_{yu},  \bm{\Phi}_{uu} \in \mathcal{F}_T \;  \\
              & \qquad \qquad \qquad \qquad \qquad
            \text{are in the affine subspace~\eqref{eq:iop}}   \},
        \end{aligned}
    \end{equation*}
it is easy to derive the following corollary.

\begin{corollary}
\label{co:inclusions}
If $(A, B, C)$ and $(A_k, B_k, C_k)$ are both stabilizable and detectable, 
we have
$
   {\mathcal{C}_{\text{SLP}}} \subset {\mathcal{C}_{\text{M1}}} \subset {\mathcal{C}_{\text{IOP}}} \subset \mathcal{C}_{\text{stab}}$ and $
    {\mathcal{C}_{\text{SLP}}} \subset {\mathcal{C}_{\text{M2}}} \subset {\mathcal{C}_{\text{IOP}}} \subset \mathcal{C}_{\text{stab}}$. If $(A, B, C)$ and $(A_k, B_k, C_k)$ are both controllable and observable, we have $
    {\mathcal{C}_{\text{SLP}}} = {\mathcal{C}_{\text{M1}}} =  {\mathcal{C}_{\text{M2}}} = {\mathcal{C}_{\text{IOP}}} \subset \mathcal{C}_{\text{stab}}$.
\end{corollary}

Theoretically, the closed-loop parameterizations in Propositions~\ref{prop:slp1}-\ref{prop:slp4} are equivalent to each other. However, after imposing the FIR approximation on the decision variables, {Corollary~\ref{co:inclusions}} shows that the IOP~\cite{furieri2019input} in Proposition~\ref{prop:slp2} has the best ability to approximate the set of stabilizing controllers $\mathcal{C}_{\text{stab}}$, as it exclusively deals with the maps from inputs to outputs without passing through the system state; see Figure~\ref{fig:Venn_diagram} for illustration. Precisely, when there are some stable uncontrollable and/or unobservable modes in~\eqref{eq:LTI}, these modes cannot be changed by any feedback controller and will be reflected in the closed-loop responses involving the state $\mathbf{x}$. Therefore, for systems with stable uncontrollable and/or unobservable 
modes, the parameters in the SLP~\cite{wang2019system}, or the new parameterization in Proposition~\ref{prop:slp3}/\ref{prop:slp4} (Mixed I/II), cannot be made FIR by definition, since these parameterizations involve the state $\mathbf{x}$ and/or the disturbance on the state $\bm{\delta}_x$ explicitly.
For example, consider an LTI system~\eqref{eq:LTI} with matrices as
$$
    A = \begin{bmatrix}
        0.5 & 0 \\
        0 & 1
    \end{bmatrix}, \; B = \begin{bmatrix} 0 \\ 1 \end{bmatrix}, \; C = \begin{bmatrix} 0 & 1 \end{bmatrix}.
$$
There is one uncontrollable and unobservable mode $z = 0.5$, and this mode is stable. The affine constraints~\eqref{eq:slp1_constraint}, ~\eqref{eq:slp3constraint},~\eqref{eq:slp4constraint} are all infeasible for any FIR approximation with finite horizon $T$ since the mode $z = 0.5$ cannot be represented by FIR exactly, while the IOP in Proposition~\ref{prop:slp2} is feasible as long as the horizon $T \geq 1$.

\begin{remark}
    Note that if there are some stable uncontrollable and/or unobservable modes in~\eqref{eq:LTI}, one may perform a model reduction to get an equivalent state-space realization that is controllable and observable. Then, all the closed-loop parameterizations in Propositions~\ref{prop:slp1}-\ref{prop:slp4} have the same ability for approximating $\mathcal{C}_{\text{stab}}$ when imposing the FIR assumption. We note that model reduction generally destroys the underlying sparsity structure in the original system~\eqref{eq:LTI}, which may be unfavourable for distributed controller synthesis~\cite{wang2018separable}. 
\end{remark}

\begin{figure}
    \centering
    \setlength{\abovecaptionskip}{0pt}
    \setlength{\belowcaptionskip}{0pt}
    \begin{tikzpicture}[set/.style={fill=cyan,fill opacity=0.1}]
\draw[set,
      yshift=0cm,
    rotate =90] (0,0) ellipse (1.2cm and 0.6cm);

\draw[set,
     xshift=0.5cm,
     yshift=0.7cm,
    rotate =70] (0,0) ellipse (2cm and 1cm);

\draw[set,
     xshift=-0.5cm,
     yshift=0.7cm,
    rotate =-70] (0,0) ellipse (2cm and 1cm);

\draw[set,
    yshift=1.2cm,
    rotate =90] (0,0) ellipse (2.4cm and 1.8cm);

\draw[set,
    yshift=1.6cm,
    rotate =90] (0,0) ellipse (2.8cm and 2.2cm);
 \node at (0,0.3) {\small ${\mathcal{C}_{\text{SLP}}}$};
 \node at (1,1.7) {\small ${\mathcal{C}_{\text{M1}}}$};
 \node at (-1,1.7) {\small ${\mathcal{C}_{\text{M2}}}$};
 \node at (0,3) {\small ${\mathcal{C}_{\text{IOP}}}$};
 \node at (0,4) {\small ${\mathcal{C}_{\text{stab}}}$};
\end{tikzpicture}






 \vspace{-2mm}
    \caption{The IOP provides the best inner approximation of $\mathcal{C}_{\text{stab}}$ using FIR approximations: $
   {\mathcal{C}_{\text{SLP}}} \subset {\mathcal{C}_{\text{M1}}} \subset {\mathcal{C}_{\text{IOP}}} \subset \mathcal{C}_{\text{stab}}$, and $
    {\mathcal{C}_{\text{SLP}}} \subset {\mathcal{C}_{\text{M2}}} \subset {\mathcal{C}_{\text{IOP}}} \subset \mathcal{C}_{\text{stab}}$.}
    \label{fig:Venn_diagram}
\end{figure}

\subsection{Controller implementation via state-space realization}
\label{section:statespace}
In Propositions~\ref{prop:slp1}--\ref{prop:slp4}, to get the controller $\mathbf{K}$, we need to compute the inverse of some transfer matrix as well as the product of transfer matrices.
For the SLP in~\cite{wang2019system}, the authors proposed the following implementation of the controller $\mathbf{K} = \bm{\Phi}_{uy} - \bm{\Phi}_{ux}\bm{\Phi}_{xx}^{-1}\bm{\Phi}_{xy}$ from the system responses matrices
$ \bm{\Phi}_{xx}, \bm{\Phi}_{xy},            \bm{\Phi}_{ux}, \bm{\Phi}_{uy} $:
\begin{equation} \label{eq:SLPcontrollerimplementation}
    \begin{aligned}
        z\bm{\beta} &= z(I - z \bm{\Phi}_{xx})\bm{\beta} - z\bm{\Phi}_{xy} \mathbf{y},\\
        \mathbf{u} &= z\bm{\Phi}_{ux}\beta + \bm{\Phi}_{uy}\mathbf{y}.
    \end{aligned}
\end{equation}
The implementation~\eqref{eq:SLPcontrollerimplementation} avoids the explicit computation of matrix inverse and matrix product. We note that a few other realizations have been discussed in~\cite{rantzer2019realizability,jensen2021explicit}. However, the controller matrices in~\eqref{eq:SLPcontrollerimplementation} still contain transfer matrices.
Motivated by~\cite{anderson2017structured}, this subsection provides a standard state-space realization~\eqref{eq:state_space_controller} for the controller in closed-loop parameterizations after imposing the FIR approximation.

We consider the controller $\mathbf{K} = \bm{\Phi}_{uy}\bm{\Phi}_{yy}^{-1}$ in Proposition~\ref{prop:slp2} and~\ref{prop:slp3} (IOP and Mixed I). 
We assume that the system response $\bm{\Phi}_{uy}$ and $\bm{\Phi}_{yy}$ are FIR with horizon $T$, \emph{i.e.},
\begin{equation} \label{eq:FIRiop}
    \bm{\Phi}_{uy} = \sum_{t=0}^T U_t \frac{1}{z^t} \in \mathcal{RH}_{\infty}, \;   \bm{\Phi}_{yy} = \sum_{t=0}^T Y_t \frac{1}{z^t} \in \mathcal{RH}_{\infty}.
\end{equation}
Upon defining the following real matrices
\begin{equation}\label{eq:ZpIp}
\begin{aligned}
    \hat{U} &= \begin{bmatrix} U_1 & U_2 & \ldots & U_T\end{bmatrix} \in  \mathbb{R}^{m \times pT}, \; \\
                        \hat{Y} &= \begin{bmatrix} Y_1 & Y_2 & \ldots & Y_T\end{bmatrix} \in  \mathbb{R}^{p \times pT},
\end{aligned}
\end{equation}
and $Z_p \in \mathbb{R}^{pT \times pT}$ as the down shift operator with sub-diagonal containing identity matrices of dimension $p \times p$ 
and $
    \mathcal{I}_p = [ I_p, 0 , \ldots,  0 ]^\tr \in \mathbb{R}^{pT \times p},
$
we have the following result.
\begin{theorem} \label{Theo:IOPstatespace}
Suppose that 
$\bm{\Phi}_{uy}$ and $\bm{\Phi}_{yy}$ are FIR transfer matrices with horizon $T$ in~\eqref{eq:FIRiop}. A state-space realization for the output feedback controller $\mathbf{K} = \bm{\Phi}_{uy}\bm{\Phi}_{yy}^{-1}$  is given by
\vspace{-2mm}
\begin{equation} \label{eq:IOPstatespace}
 \mathbf{K}= \left[\begin{array}{c|c} Z_p - \mathcal{I}_p\hat{Y} & -\mathcal{I}_p \\\hline
    U_0\hat{Y}- \hat{U} & U_0\end{array}\right].
    \vspace{-2mm}
\end{equation}
\end{theorem}
A state-space realization for the controller $\mathbf{K} = \bm{\Phi}_{uu}^{-1}\bm{\Phi}_{uy}$ in Proposition~\ref{prop:slp4} (Mixed II) can be developed similarly. The proof of Theorem~\ref{Theo:IOPstatespace} is motivated by~\cite{anderson2017structured}, and is based on some standard operations on dynamical systems (see, e.g.,~\cite[Chapter 3.6]{zhou1996robust}). We provide the proofs in Appendix~\ref{app:statespace} for completeness. In Appendix~\ref{app:statespace}, we also provide a state-space realization for the SLP controller $\mathbf{K} =  \bm{\Phi}_{uy} - \bm{\Phi}_{ux}\bm{\Phi}_{xx}^{-1}\bm{\Phi}_{xy}$. Finally, we note that the state-space realization in~\eqref{eq:IOPstatespace} is in general not minimal.

\begin{table*}[t]
  \centering
    \renewcommand\arraystretch{1.2}
  \caption{Comparison of numerical robustness among different closed-loop parameterizations}
  {\small
  \begin{tabular}{c c c l c c c}
  \toprule
     & \makecell{Coprime \\ factorization } & \makecell{Equality \\ constraints } & \makecell{Controller \\ recovery $\mathbf{K}$} & \makecell{Open-loop \\ stable plants} & \makecell{Open-loop \\ unstable plants} & \makecell{Pre-stabilizing \\ the plant$^{2}$}  \\
    \cline{2-7}
    \multirow{3}{*}{SLP~\cite{wang2019system}} & \multirow{3}{*}{No} & \multirow{3}{*}{Yes} & ${\bm{\Phi}}_{uy} - {\bm{\Phi}}_{ux}{\bm{\Phi}}_{xx}^{-1}{\bm{\Phi}}_{xy} $ & $*$ & $*$ & $*$  \\
         & &  & ${\bm{\Phi}}_{uy}(I + C{\bm{\Phi}}_{xy})^{-1}$ & \checkmark & \xmark & \checkmark  \\
       & & & ${\bm{\Phi}}_{ux}{\bm{\Phi}}_{xx}^{-1}$ (when $C = I$)$^{1}$ & \checkmark & \checkmark & \checkmark  \\
       \hline
    IOP~\cite{furieri2019input} & No & Yes & \multicolumn{1}{c}{$\mathbf{\Phi}_{uy}\mathbf{\Phi}_{uy}^{-1} $} & \checkmark &  \xmark & \checkmark\\
    Mixed I & No &  Yes & \multicolumn{1}{c}{$\mathbf{\Phi}_{uy}\mathbf{\Phi}_{uy}^{-1} $} & \checkmark & \xmark & \checkmark\\
    Mixed II &No & Yes &  \multicolumn{1}{c}{$\mathbf{\Phi}_{uu}^{-1}\mathbf{\Phi}_{uy} $} & \checkmark & \xmark & \checkmark\\
    Youla~\cite{youla1976modern} &Yes&  No & \multicolumn{1}{c}{\eqref{eq:youla}} &  \checkmark & \checkmark & \checkmark \\
    \bottomrule
  \end{tabular}
  }
  \raggedright

{\scriptsize
{\begin{spacing}{1}
$^{1}$\;: This only works for the state feedback case, \emph{i.e.}, $C = I$.\newline
$^{2}$\;: This applies an initial stabilizing controller that is stable itself (see Proposition~\ref{prop:initialcontroller}). \newline
$*$\;: The situation requires care-by-case analysis; see Theorem~\ref{theo:SLProbustness} and Section~\ref{subsection:numerical} for details. \newline
\checkmark: The parameterization is numerically robust (see Corollary~\ref{coro:robustness}).\newline
\xmark\;: The parameterization cannot guarantee the closed-loop stability if small numerical mismatches in the equality constraints exist.\newline
\end{spacing}
}}
\label{tab:comparision}
\vspace{-3mm}
\end{table*}

\section{Numerical robustness of closed-loop parameterizations} \label{section:robustness}

The previous sections highlighted the benefits of closed-loop parameterizations: the set of internally stabilizing controllers can be fully characterized by a set of affine constraints on certain closed-loop responses, leading to finite-dimensional convex optimization problems for controller synthesis after imposing the FIR constraints. However, numerical solutions computed via arbitrarily precise floating point arithmetic can never solve the affine constraints exactly. This phenomenon is further exacerbated by  the finite stopping criteria used in common solvers, like SeDuMi~\cite{sturm1999using} and Mosek~\cite{andersen2000mosek}. Moreover, any controller implementation that uses floating-point arithmetic also introduce errors. Therefore, numerical mismatches in solving the affine constraints in Propositions~\ref{prop:slp1}--\ref{prop:slp4} always exist.

This section investigates how the numerical mismatches in the affine constraints affect the stability of the closed-loop system, a property termed as \emph{numerical robustness} of closed-loop parameterizations. 
An overview of the results in this section is presented in Table~\ref{tab:comparision}.

\subsection{Robustness results for the IOP and the SLP}
We begin with the IOP in Proposition~\ref{prop:slp2}. The transfer matrices $ \hat{\bm{\Phi}}_{yy}, \hat{\bm{\Phi}}_{uy},
            \hat{\bm{\Phi}}_{yu}, \hat{\bm{\Phi}}_{uu}$ only approximately satisfy the affine constraint~\eqref{eq:iop}, \emph{i.e.}, we have
 \begin{equation}\label{eq:iop_error}
    \begin{aligned}
    \begin{bmatrix} I & -\mathbf{G} \end{bmatrix}\begin{bmatrix}
           \hat{\bm{\Phi}}_{yy} & \hat{\bm{\Phi}}_{yu} \\
            \hat{\bm{\Phi}}_{uy} & \hat{\bm{\Phi}}_{uu}
        \end{bmatrix} &= \begin{bmatrix} I+\bm{\Delta}_1 & \bm{\Delta}_2 \end{bmatrix}, \\
            \begin{bmatrix}
            \hat{\bm{\Phi}}_{yy} & \hat{\bm{\Phi}}_{yu} \\
            \hat{\bm{\Phi}}_{uy} & \hat{\bm{\Phi}}_{uu}
        \end{bmatrix}\begin{bmatrix}  -\mathbf{G} \\I \end{bmatrix} &= \begin{bmatrix} \bm{\Delta}_3 \\ I+\bm{\Delta}_4\end{bmatrix}, \\
       \hat{\bm{\Phi}}_{yy}, \hat{\bm{\Phi}}_{uy},
            \hat{\bm{\Phi}}_{yu}, \hat{\bm{\Phi}}_{uu} &\in \mathcal{RH}_{\infty},
        \end{aligned}
\end{equation}
where the {residuals} are $\bm{\Delta}_1 = \hat{\bm{\Phi}}_{yy} - \mathbf{G}\hat{\bm{\Phi}}_{uy} - I,  \bm{\Delta}_2 = \hat{\bm{\Phi}}_{yu} - \mathbf{G}\hat{\bm{\Phi}}_{uu},
    \bm{\Delta}_3 = -\hat{\bm{\Phi}}_{yy}\mathbf{G} +\hat{\bm{\Phi}}_{yu},  \bm{\Delta}_4 = -\hat{\bm{\Phi}}_{uy}\mathbf{G}+\hat{\bm{\Phi}}_{uu} - I.$

\begin{remark}
    We note that the residuals $\bm{\Delta}_i (i = 1, \ldots, 4)$ can come from 1) numerical computation when enforcing the equality constraints numerically,  and 2) the controller implementation that uses floating-point arithmetic. We use $\bm{\Delta}_i (i = 1, \ldots, 4)$ to denote such mismatches from the equality constraints. The results in this section work regardless of the phenomenon  generating $\bm{\Delta}_i (i = 1, \ldots, 4)$.
\end{remark}

\begin{theorem} \label{theo:IOProbustness}
  Let $\hat{\bm{\Phi}}_{yy}, \hat{\bm{\Phi}}_{uy},
            \hat{\bm{\Phi}}_{yu}, \hat{\bm{\Phi}}_{uu}$ satisfy~\eqref{eq:iop_error}. Then, we have the following statements.
            \begin{enumerate}
                \item  In the case of $\mathbf{G} \in \mathcal{RH}_{\infty}$, the controller $\mathbf{K} = \hat{\bm{\Phi}}_{uy}\hat{\bm{\Phi}}_{yy}^{-1}$ internally stabilizes the plant $\mathbf{G}$ if and only if $(I + \bm{\Delta}_1)^{-1}$ is stable.
                \item In the case of $\mathbf{G} \notin \mathcal{RH}_{\infty}$, the controller $\mathbf{K} = \hat{\bm{\Phi}}_{uy}\hat{\bm{\Phi}}_{yy}^{-1}$
                fails to internally stabilize the closed-loop system due to non-zero residuals $\bm{\Delta}_i$.
            \end{enumerate}
\end{theorem}

\begin{proof}
Given a controller $\mathbf{K}$, the closed-loop responses from $(\bm{\delta}_y, \bm{\delta}_u)$ to $(\mathbf{y},\mathbf{u})$ are
\begin{equation*}
    \begin{bmatrix}
        \mathbf{y} \\
        \mathbf{u}
    \end{bmatrix} = \begin{bmatrix} (I - \mathbf{G}\mathbf{K})^{-1} & (I - \mathbf{G}\mathbf{K})^{-1}\mathbf{G} \\
    \mathbf{K}(I - \mathbf{G}\mathbf{K})^{-1} &                 I + \mathbf{K}(I - \mathbf{G}\mathbf{K})^{-1}\mathbf{G}\end{bmatrix}\begin{bmatrix}
        \bm{\delta}_{y} \\
        \bm{\delta}_{u}
    \end{bmatrix}.
\end{equation*}
Considering $\mathbf{K} = \hat{\bm{\Phi}}_{uy}\hat{\bm{\Phi}}_{yy}^{-1}$, where $\hat{\bm{\Phi}}_{yy}, \hat{\bm{\Phi}}_{uy},
            \hat{\bm{\Phi}}_{yu}, \hat{\bm{\Phi}}_{uu}$ satisfy~\eqref{eq:iop_error}, we can verify the following identities:

            \begin{equation} \label{eq:ioptf_error}
                \begin{aligned}
                    (I - \mathbf{G}\mathbf{K})^{-1} &= (I - \mathbf{G}\hat{\bm{\Phi}}_{uy}\hat{\bm{\Phi}}_{yy}^{-1})^{-1} \\
                    &=  \hat{\bm{\Phi}}_{yy}(I + \bm{\Delta}_1)^{-1}, \\
                (I - \mathbf{G}\mathbf{K})^{-1}\mathbf{G} &  = \hat{\bm{\Phi}}_{yy}(I + \bm{\Delta}_1)^{-1}\mathbf{G}, \\
                \mathbf{K}(I - \mathbf{G}\mathbf{K})^{-1} &= \hat{\bm{\Phi}}_{uy}(I + \bm{\Delta}_1)^{-1}, \\
                \end{aligned}
            \end{equation}
            and $ I + \mathbf{K}(I - \mathbf{G}\mathbf{K})^{-1}\mathbf{G} = I + \hat{\bm{\Phi}}_{uy}(I + \bm{\Delta}_1)^{-1}\mathbf{G}.$

    \emph{Proof of Statement 1:} Suppose that $\mathbf{G} \in \mathcal{RH}_{\infty}$. If $(I + \bm{\Delta}_1)^{-1}$ is stable, it is easy to verify that all transfer matrices in~\eqref{eq:ioptf_error} are stable. This means that
    \vspace{-2mm}
    $$
        \left(\begin{bmatrix}\delta_y \\ \delta_u\end{bmatrix} \rightarrow \begin{bmatrix}\mathbf{y} \\ \mathbf{u}\end{bmatrix}\right) \in \mathcal{RH}_{\infty}.
        \vspace{-2mm}
    $$
    By Theorem~\ref{Theo:mainresult}, we know $\mathbf{K} = \hat{\bm{\Phi}}_{uy}\hat{\bm{\Phi}}_{yy}^{-1}$ internally stabilizes the plant $\mathbf{G}$. If $(I+\bm{\Delta}_1)^{-1}$ is unstable, then the closed-loop response from $\bm{\delta}_y$ to $\mathbf{y}$ will be unstable in general, and thus the controller does not internally stabilize the system.

    \emph{Proof of Statement 2:} If $\mathbf{G}$ is unstable, the transfer matrices in~\eqref{eq:ioptf_error} cannot be guaranteed to be stable if $\bm{\Delta}_1 \neq 0$. When $\bm{\Delta}_1 = 0 $, we have $(I - \mathbf{G}\mathbf{K})^{-1}\mathbf{G} = \hat{\bm{\Phi}}_{yy}\mathbf{G} = \hat{\bm{\Phi}}_{yu} - \bm{\Delta}_3$ and
    $
     I + \mathbf{K}(I - \mathbf{G}\mathbf{K})^{-1}\mathbf{G} = I + \hat{\bm{\Phi}}_{uy}\mathbf{G} = \hat{\bm{\Phi}}_{uu} - \bm{\Delta}_4.
    $
    Note that the {residuals} $\bm{\Delta}_3, \bm{\Delta}_4$ in~\eqref{eq:iop_error} can be unstable if $\mathbf{G}$ is unstable (since the product of an unstable transfer matrix and a stable one can be unstable). Thus, the controller fails to guarantee the internal stability of the closed-loop system unless $\bm{\Delta}_1 = 0,\bm{\Delta}_2 = 0,\bm{\Delta}_3 = 0,\bm{\Delta}_4 = 0$.
\end{proof}

We now focus on the SLP in Proposition~\ref{prop:slp1}. The transfer matrices $ \hat{\bm{\Phi}}_{xx}, \hat{\bm{\Phi}}_{ux},
            \hat{\bm{\Phi}}_{xy}, \hat{\bm{\Phi}}_{uy}$ only approximately satisfy the affine constraint~\eqref{eq:slp1_constraint}, \emph{i.e.}, we have
            \vspace{-2mm}
 \begin{equation}\label{eq:slp_error}
    \begin{aligned}
    \begin{bmatrix} zI - A & -B \end{bmatrix}\begin{bmatrix}
           \hat{\bm{\Phi}}_{xx} & \hat{\bm{\Phi}}_{xy} \\
            \hat{\bm{\Phi}}_{ux} & \hat{\bm{\Phi}}_{uy}
        \end{bmatrix} &= \begin{bmatrix} I+\hat{\bm{\Delta}}_1 & \hat{\bm{\Delta}}_2 \end{bmatrix}, \\
            \begin{bmatrix}
            \hat{\bm{\Phi}}_{xx} & \hat{\bm{\Phi}}_{xy} \\
            \hat{\bm{\Phi}}_{ux} & \hat{\bm{\Phi}}_{uy}
        \end{bmatrix}\begin{bmatrix}  zI - A \\ -C \end{bmatrix} &= \begin{bmatrix} I + \hat{\bm{\Delta}}_3 \\ \hat{\bm{\Delta}}_4\end{bmatrix}, \\
       \hat{\bm{\Phi}}_{xx}, \hat{\bm{\Phi}}_{ux},
            \hat{\bm{\Phi}}_{xy}, \hat{\bm{\Phi}}_{uy} &\in \mathcal{RH}_{\infty},
        \end{aligned}
        \vspace{-2mm}
\end{equation}
where the {residuals} are $\hat{\bm{\Delta}}_1 = (zI - A)\hat{\bm{\Phi}}_{xx} - B\hat{\bm{\Phi}}_{ux} - I, $ and

    $$
    \begin{aligned}
     \hat{\bm{\Delta}}_2 &= (zI - A)\hat{\bm{\Phi}}_{xy} - B\hat{\bm{\Phi}}_{uy}, \\
    \hat{\bm{\Delta}}_3 &= \hat{\bm{\Phi}}_{xx}(zI - A) - \hat{\bm{\Phi}}_{xy}C - I \\
    \hat{\bm{\Delta}}_4 &= \hat{\bm{\Phi}}_{ux}(zI - A) - \hat{\bm{\Phi}}_{uy}C.
\end{aligned}
$$

Note that there are multiple ways to recover the controller $\mathbf{K}$ in the SLP framework. The {SLP} controller can also be recovered in another way as $\mathbf{K} = \hat{\bm{\Phi}}_{uy}(I + C\hat{\bm{\Phi}}_{xy})^{-1}$~\cite{wang2019system}. Now, we have the following result.

\begin{theorem} \label{theo:SLProbustness}
  Let $\hat{\bm{\Phi}}_{xx}, \hat{\bm{\Phi}}_{ux},
            \hat{\bm{\Phi}}_{xy}, \hat{\bm{\Phi}}_{uy}$ satisfy~\eqref{eq:slp_error}. We have the following statements.
            \begin{enumerate}
                \item  In the state feedback case, i.e., $C = I$, the controller $\mathbf{K} = \hat{\bm{\Phi}}_{ux}\hat{\bm{\Phi}}_{xx}^{-1}$ internally stabilizes the plant $\mathbf{G}$ if and only if $(I + \hat{\bm{\Delta}}_1)^{-1}$ is stable.
                \item The four-block SLP controller $\mathbf{K} = \hat{\bm{\Phi}}_{uy} - \hat{\bm{\Phi}}_{ux}\hat{\bm{\Phi}}_{xx}^{-1}\hat{\bm{\Phi}}_{xy}$ cannot guarantee the internal stability of the closed-loop system if
                $
                (I + \hat{\bm{\Delta}})^{-1}
                $
                is unstable, where\footnote{Note that because $\|\hat{\bm{\Delta}}\|_{\infty}$ may large than $1$, there is no guarantee that $(I + \hat{\bm{\Delta}})^{-1}$ is always stable. See Example~\eqref{eq:SLPcounterex}.}
                \begin{equation}\label{eq:Runcertainty} \hat{\bm{\Delta}}:=\hat{\bm{\Delta}}_3+ \hat{\bm{\Phi}}_{xx}(I + \hat{\bm{\Delta}}_1)^{-1}\big(B\hat{\bm{\Delta}}_4 - (zI - A)\hat{\bm{\Delta}}_3\big).
                \end{equation}
                \item For the controller $\mathbf{K} = \hat{\bm{\Phi}}_{uy}(I + C\hat{\bm{\Phi}}_{xy})^{-1}$,

                 \begin{enumerate}
                \item[a)] if $\mathbf{G} \in \mathcal{RH}_{\infty}$, $\mathbf{K}$  internally stabilizes the plant $\mathbf{G}$ if and only if $(I + C(zI - A)^{-1}\hat{\bm{\Delta}}_2)^{-1}$ is stable.
                \item[b)] if $\mathbf{G}\notin \mathcal{RH}_{\infty}$, $\mathbf{K}$ fails to internally stabilizes the closed-loop system due to non-zero residuals $\hat{\bm{\Delta}}_i, i = 1, \ldots, 4$.
            \end{enumerate}
                \end{enumerate}
\end{theorem}

\begin{proof}
The proof of Statement 1 is presented in~\cite[Theorem 4.3]{anderson2019system}. We prove the second statement here. Given $\hat{\bm{\Phi}}_{xx}, \hat{\bm{\Phi}}_{ux},
            \hat{\bm{\Phi}}_{xy}, \hat{\bm{\Phi}}_{uy}$ satisfying~\eqref{eq:slp_error} and the controller $\mathbf{K} = \hat{\bm{\Phi}}_{uy} - \hat{\bm{\Phi}}_{ux}\hat{\bm{\Phi}}_{xx}^{-1}\hat{\bm{\Phi}}_{xy}$, we consider the closed-loop response from $\bm{\delta}_x$ to $\mathbf{x}$. After some tedious algebra (see Appendix~\ref{app:slperror}), we derive
            \begin{equation} \label{eq:SLPuncertainty}
    \begin{aligned}
        (zI - A - B \mathbf{K} C)^{-1}
        = (I + \hat{\bm{\Delta}})^{-1}\hat{\bm{\Phi}}_{xx}(I + \hat{\bm{\Delta}}_1)^{-1},
    \end{aligned}
\end{equation}
with $\hat{\bm{\Delta}}$ defined in~\eqref{eq:Runcertainty}.
If $(I + \hat{\bm{\Delta}})^{-1}$ is unstable, there is no guarantee that the closed-loop response from $\bm{\delta}_x$ to $\mathbf{x}$ is stable. In this case, the controller $\mathbf{K} = \hat{\bm{\Phi}}_{uy} - \hat{\bm{\Phi}}_{ux}\hat{\bm{\Phi}}_{xx}^{-1}\hat{\bm{\Phi}}_{xy}$ cannot internally stabilize the plant.

For Statement 3, considering Corollary~\ref{prop:stable}, we only need to check the closed-loop response from $\bm{\delta}_y$ to $\mathbf{u}$, which is
    $$
    \begin{aligned}
    &\mathbf{K}(I - \mathbf{G}\mathbf{K})^{-1} \\
    = &\hat{\bm{\Phi}}_{uy}(I + C\hat{\bm{\Phi}}_{xy})^{-1} (I - \mathbf{G}\hat{\bm{\Phi}}_{uy}(I + C\hat{\bm{\Phi}}_{xy})^{-1})^{-1}  \\
    =& \hat{\bm{\Phi}}_{uy} (I + C\hat{\bm{\Phi}}_{xy} - C(zI - A)^{-1}B\hat{\bm{\Phi}}_{uy})^{-1} \\
    = & \hat{\bm{\Phi}}_{uy} (I + C(zI - A)^{-1}\hat{\bm{\Delta}}_2)^{-1}.
    \end{aligned}
    $$
   The rest of the proof is similar to Theorem~\ref{theo:IOProbustness}.
\end{proof}

Theorem~\ref{theo:SLProbustness} quantifies the numerical robustness of different controller recovery in the SLP due to the mismatches of the associated equality constraints. The mismatches come from  floating-point  arithmetic  in either numerical computation or controller implementation. This is irrespective of whether the FIR approximation is used for the closed-loop responses. Similar robustness results can be derived for the Mixed I/II parameterizations (Propositions~\ref{prop:slp3}/\ref{prop:slp4}); see Appendix~\ref{app:mixedI/II}.
Theorems~\ref{theo:IOProbustness} and \ref{theo:SLProbustness} can now be combined with the small gain theorem~\cite[Theorem 9.1]{zhou1996robust} to provide simple sufficient conditions for {numerical robustness}. 
\begin{corollary} \label{coro:robustness}
    {Let $\hat{\bm{\Phi}}_{yy}, \hat{\bm{\Phi}}_{uy},
            \hat{\bm{\Phi}}_{yu}, \hat{\bm{\Phi}}_{uu}$ satisfy~\eqref{eq:iop_error}. Then}
    \begin{itemize}
        \item[--] for open-loop stable plants, the IOP controller $\mathbf{K} = \hat{\bm{\Phi}}_{uy}\hat{\bm{\Phi}}_{yy}^{-1}$ internally stabilizes the plant if $\|\bm{\Delta}_1\|_{\infty} < 1$.
    \end{itemize}
         {Let $\hat{\bm{\Phi}}_{xx}, \hat{\bm{\Phi}}_{ux},
            \hat{\bm{\Phi}}_{xy}, \hat{\bm{\Phi}}_{uy}$ satisfy~\eqref{eq:slp_error}. Then} 
     \begin{itemize}
         \item[--] for the state feedback case, the SLP controller $\mathbf{K} = \hat{\bm{\Phi}}_{ux}\hat{\bm{\Phi}}_{xx}^{-1}$ internally stabilizes the plant  if $\|\hat{\bm{\Delta}}_1\|_{\infty} < 1$.
         \item[--] for open-loop stable plants,  the SLP controller $\mathbf{K} = \hat{\bm{\Phi}}_{uy}(I + C\hat{\bm{\Phi}}_{xy})^{-1}$ internally stabilizes the plant if
         $$\|\hat{\bm{\Delta}}_2\|_{\infty} \leq \frac{1}{\|C(zI - A)\|_{\infty}}.$$
     \end{itemize}

\end{corollary}

The sufficient condition for robustness of the SLP state feedback case 
first appeared in~\cite{matni2017scalable}, which is one key result in the recent learning-based control applications~\cite{dean2017sample,dean2018regret}.
{
\begin{remark} \label{remark:stability}
    The controller recovery in IOP/Mixed I/II, the two-block state-feedback SLP controller, and the SLP controller $\mathbf{K} = \hat{\bm{\Phi}}_{uy}(I + C\hat{\bm{\Phi}}_{xy})^{-1}$ only involve two parameters explicitly. Thus, their robustness analysis is straightforward. The corresponding analysis for the four-block SLP controller $\mathbf{K} = \hat{\bm{\Phi}}_{uy} - \hat{\bm{\Phi}}_{ux}\hat{\bm{\Phi}}_{xx}^{-1}\hat{\bm{\Phi}}_{xy}$, instead, does not provide a numerical robustness result. As shown in Theorem~\ref{theo:SLProbustness}, the residuals $\hat{\bm{\Delta}}_i, i = 1, \ldots, 4$ play a more complex role in the resulting closed-loop responses, irrespective of state- or output-feedback, or open-loop stability of the plant. Since the closed-loop responses from $\bm{\delta}_x, \bm{\delta}_y, \bm{\delta}_u$ to $\bm{x}, \bm{y}, \bm{u}$ can be computed using these residuals $\hat{\bm{\Delta}}_i$, one may find sophisticated sufficient conditions on the residuals $\hat{\bm{\Delta}}_i$ to ensure the internally stability of the closed-loop system. These conditions might relate these residuals with
   numerical solutions such as $\hat{\bm{\Phi}}_{xx}$. Deriving such conditions and finding tractable ways to enforce these conditions are beyond the scope of this paper and left as future work.
\vspace{-1mm}

\end{remark}

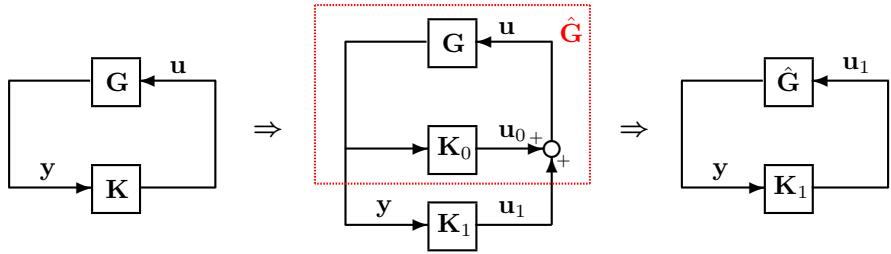
\begin{figure*}[t]
\vspace*{-3mm}
   \begin{center}
      \raisebox{7mm}{\input{Basic_Loop_GK}}
      \hspace{-12mm}
      \raisebox{20mm}{$\bm{\Rightarrow}$}
      \hspace{-3mm}
      \input{Double_Loop.tex}
      \hspace{-8mm}
      \raisebox{20mm}{$\bm{\Rightarrow}$}
       \hspace{-7mm}
      \raisebox{7mm}{\input{Basic_Loop_hatGK}}
  \end{center}
  \vspace*{-6mm}
  \caption{Given an initial controller $\mathbf{K}_0 \in \mathcal{C}_{\text{stab}}\cap \mathcal{RH}_{\infty}$, we search for $\mathbf{K}_1$ to stabilize the new stable plant $\hat{\mathbf{G}} := (I - \mathbf{G}\mathbf{K}_0)^{-1}\mathbf{G}$. }
    \label{fig:initialStabilizing}
\end{figure*}

\subsection{Implications in numerical computation and controller implementation} \label{subsection:numerical}
Here, we discuss the implication of Theorems~\ref{theo:IOProbustness}, \ref{theo:SLProbustness} and Corollary~\ref{coro:robustness}. 
In practice, the {residuals} $\bm{\Delta}_1 = \hat{\bm{\Phi}}_{yy} - \mathbf{G}\hat{\bm{\Phi}}_{uy} - I$ ({when $\mathbf{G}$ is stable}) and  $\hat{\bm{\Delta}}_1 = (zI - A)\hat{\bm{\Phi}}_{xx} - B\hat{\bm{\Phi}}_{xy} - I $ are very small numerically. It is fairly safe to say that {$\|\bm{\Delta}_1\|_{\infty} < 1$} and  $\|\hat{\bm{\Delta}}_1\|_{\infty} < 1$ in floating-point implementation and numerical computation using any common interior-point solvers, such as SeDuMi~\cite{sturm1999using} and Mosek~\cite{andersen2000mosek}.} Similar statements are true for the Mixed I/II parameterizations. This observation leads to the following summary (see Table~\ref{tab:comparision} for an overview).

 \vspace{2mm}

 \noindent{\bf Numerical Robustness}. {\it
  Consider closed-loop parameterizations~(SLP, IOP, Mixed I/II) in numerical computation. We have
\begin{enumerate}
    \item[(i)] the SLP with controller $\mathbf{K} = \hat{\bm{\Phi}}_{ux}\hat{\bm{\Phi}}_{xx}^{-1}$ is numerically robust in the state feedback case;
    \item[(ii)] the IOP, Mixed I/II, and SLP with controller $\hat{\bm{\Phi}}_{uy}(I + C\hat{\bm{\Phi}}_{xy})^{-1}$ are numerically robust for open-loop stable plants.
\end{enumerate}
On the other hand, we have
\begin{enumerate}
\item[(I)] the IOP, Mixed I/II, and SLP with controller $\hat{\bm{\Phi}}_{uy}(I + C\hat{\bm{\Phi}}_{xy})^{-1}$ are not numerically robust for open-loop unstable plants, irrespective of having state- or output-feedback;
\item[(II)] the SLP with controller $\mathbf{K} =\hat{\bm{\Phi}}_{uy} - \hat{\bm{\Phi}}_{ux}\hat{\bm{\Phi}}_{xx}^{-1}\hat{\bm{\Phi}}_{xy}$ is not numerically robust {in general}, irrespective of open-loop stability of the plant. 
\end{enumerate}

\vspace{1mm}
}
The statements (i), (ii) and (I) are easy to see from the previous section.
The statement (II) comes form Theorem~\ref{theo:SLProbustness} but requires more attention. Although the computational residuals $\hat{\bm{\Delta}}_1, \hat{\bm{\Delta}}_2, \hat{\bm{\Delta}}_3, \hat{\bm{\Delta}}_4$ in~\eqref{eq:slp_error} are typically very small element-wise by interior-point solvers, we still cannot guarantee that $\|\hat{\bm{\Delta}}\|_{\infty} < 1$ (where $\hat{\bm{\Delta}}$ is defined in~\eqref{eq:Runcertainty}), since $\hat{\bm{\Delta}}$ involves $\hat{\bm{\Phi}}_{xx}$ explicitly.  Consequently, it is possible that $(I + \hat{\bm{\Delta}})^{-1}$ is unstable in numerical computation. Therefore, one may argue that the four-block SLP controller $\mathbf{K} =  \hat{\bm{\Phi}}_{uy} - \hat{\bm{\Phi}}_{ux}\hat{\bm{\Phi}}_{xx}^{-1}\hat{\bm{\Phi}}_{xy}$ is not numerically robust {in general}\footnote{As discussed in Remark~\ref{remark:stability},
sufficient conditions could exist to ensure internal stability, and they will depend on $\hat{\bm{\Phi}}_{xx}$.}. Further, we notice that the controller implementation \eqref{eq:SLPcontrollerimplementation} proposed in~\cite{wang2019system} also suffers the issue of numerical instability, as the right-hand-side of  \eqref{eq:SLPuncertainty} represents the corresponding closed-loop response using~\eqref{eq:SLPcontrollerimplementation}.

\begin{example} \label{ex:slp}
To understand the role of the residuals, we present a simple example. Consider a stable LTI system~\eqref{eq:LTI} with
$
    A = 0, \quad B = 1, \quad C = 1.
$
It can be verified that the following transfer functions
\begin{equation} \label{eq:SLPcounterex}
\begin{aligned}
\hat{\bm{\Phi}}_{xx} &=  \frac{1}{z}+\frac{(z - 5)(z + 6)^2}{z^5}, \\
    \hat{\bm{\Phi}}_{ux} &=\frac{(z - 5)(z + 6)^2}{z^4},\\
    \hat{\bm{\Phi}}_{xy} &=\frac{(z - 5)(z + 6)^2}{z^4} - \frac{1}{1000}\frac{(z + 2)^2}{z^3},\\
    \hat{\bm{\Phi}}_{uy} &= \frac{(z - 5)(z + 6)^2}{z^3}, \\
\end{aligned}
\end{equation}
satisfy~\eqref{eq:slp_error} with {residuals}
\begin{equation*} 
\begin{aligned}
\hat{\bm{\Delta}}_1 = 0,
\hat{\bm{\Delta}}_2 = -\frac{(z + 2)^2}{1000z^2},
\hat{\bm{\Delta}}_3 =\frac{(z + 2)^2}{1000z^3},
\hat{\bm{\Delta}}_4 = 0.
\end{aligned}
\end{equation*}
For this example, we verify that $(I + \hat{\bm{\Delta}})^{-1}$ has a pair of unstable poles $z = 0.9522 \pm 0.5226i$, {despite the norm $\|\hat{\Delta}_3\|_{\infty}=9\times 10^{-3}$ being very small}, and that this pair of unstable poles also {appears} in the closed-loop system $(zI - A - B\mathbf{K}C)^{-1}$ using the controller $\mathbf{K} =  \hat{\bm{\Phi}}_{uy} - \hat{\bm{\Phi}}_{ux}\hat{\bm{\Phi}}_{xx}^{-1}\hat{\bm{\Phi}}_{xy}$.
This example is open-loop stable and it is also in state feedback form. Nonetheless, a small residual can destabilize the closed-loop using the four-block SLP controller.
\end{example}

We remark that in Example~\ref{ex:slp}, since $A = 0$, the optimal LQR controller will be $\mathbf{K} = 0$ for any weight matrices $Q$ and $R$. Thus, any sensible formulation of optimal control problems using the SLP will not lead to the highly suboptimal solution~\eqref{eq:SLPcounterex}. However, we emphasize that the numerical residuals play a complex role in the closed-loop system using the four-block SLP controller, and residuals with a small norm may lead to an undesirable destabilization situation. Indeed, numerical instability is tightly linked to the specific controller recovery.
Unlike the four-block SLP controller, from Corollary~\ref{coro:robustness}, the state feedback SLP controller $\mathbf{K} = \hat{\bm{\Phi}}_{ux}\hat{\bm{\Phi}}_{xx}^{-1}$ is numerically robust as long as $\|\hat{\bm{\Delta}}_1\|_{\infty}< 1$. Since $\hat{\bm{\Delta}}_1 = 0$ in Example~\ref{ex:slp}, the closed-loop system $(zI - A - B\mathbf{K}C)^{-1}$ has all zero eigenvalues using $\mathbf{K} = \hat{\bm{\Phi}}_{ux}\hat{\bm{\Phi}}_{xx}^{-1}$. Meanwhile, we can verify that $\|C(zI - A)^{-1}\hat{\bm{\Delta}}_2\|_{\infty} = 0.009 < 1$, thus it is guaranteed that the controller $\mathbf{K} = \hat{\bm{\Phi}}_{uy}(I +C \hat{\bm{\Phi}}_{xy})^{-1} $ internally stabilizes the plant (the largest norm of the closed-loop eigenvalues is 0.1675).

The question remains whether the phenomenon highlighted in Example~\ref{ex:slp} may lead to numerical instability when solving optimal control formulation in practice. We observed several cases where the four-block SLP controller failed to stabilize the plant even using the default setting (high precision) in Mosek~\cite{andersen2000mosek} for numerical computation\footnote{See the examples at~\url{https://github.com/zhengy09/h2_clp}, where the system matrices $A \in \mathbb{R}^{3 \times 3}, B \in \mathbb{R}^{3 \times 1}, C \in \mathbb{R}^{1 \times 3}$ have integer elements randomly generated from $-5$ to $5$, and the weight matrices are chosen $Q = I, R = I$ in~\eqref{eq:H2_s2}. }. This is likely due to $\|\hat{\bm{\Phi}}_{xx}\|_{\infty}$ being high, despite solving an optimal control formulation. How to avoid this issue requires more investigations, which is left for future work.

\subsection{Open-loop unstable plants and relation with the Youla parameterization}
To characterize the set of internally stabilizing controllers $\mathcal{C}_{\text{stab}}$, the closed-loop parameterizations in Proposition~\ref{prop:slp1}--\ref{prop:slp4} can avoid computing the doubly co-prime factorization of the plant \emph{a priori}, but they all need to impose a set of affine constraints for achievable closed-loop responses. As discussed above, any small mismatch in the additional affine constraints can make the resulting controller un-implementable when the plant is open-loop unstable (IOP, Mixed I/II), and the four-block SLP controller requires a case-by-case investigation.

{For the case of open-loop unstable plants, there exists a valid remedy by pre-stabilizing the plant.} Suppose that $\mathbf{G}$ is unstable, and that a \textit{stable} and \textit{stabilizing} controller $\mathbf{K}_0$ is known \emph{a priori}. We can split the control signal as%
 $
    \mathbf{u} = \mathbf{K}_0\mathbf{y} + \mathbf{u}_1,
 $
 and design $\mathbf{u}_1$. This is equivalent to applying the closed-loop parameterization to the new stable plant $\hat{\mathbf{G}} := (I - \mathbf{G}\mathbf{K}_0)^{-1}\mathbf{G}$ (see Figure~\ref{fig:initialStabilizing} for illustration).
Upon defining
$$
     \hat{\mathcal{C}}_{\text{stab}} := \{\mathbf{K}_0 + \mathbf{K}_1 \mid \mathbf{K}_1\;\; \text{internally stabilizes} \;\; \hat{\mathbf{G}}\},
$$
we have the following result.

\begin{proposition} \label{prop:initialcontroller}
Given an initial controller $\mathbf{K}_0 \in \mathcal{C}_{\text{stab}} \cap \mathcal{RH}_{\infty}$, we have
   $
         \mathcal{C}_{\text{stab}} =  \hat{\mathcal{C}}_{\text{stab}}.
   $
\end{proposition}
The proof is based on algebra verification; see Appendix~\ref{app:initialcontroller}. Proposition~\ref{prop:initialcontroller} shows that searching over $\hat{\mathcal{C}}_{\text{stab}}$ has no conservatism. The new plant $\hat{\mathbf{G}} = (I - \mathbf{G}\mathbf{K}_0)^{-1}\mathbf{G}$ is stable, and thus any closed-loop parameterization in Propositions~\ref{prop:slp1}--\ref{prop:slp4} for this plant has good numerical robustness\footnote{For the SLP, we use the controller $\mathbf{K} = \hat{\bm{\Phi}}_{uy}(I + C\hat{\bm{\Phi}}_{xy})^{-1}$.}. As shown in~\cite[Theorem 17]{rotkowitz2006characterization}, giving $\mathbf{K}_0 \in \mathcal{C}_{\text{stab}} \cap \mathcal{RH}_{\infty}$, the Youla parameterization~\eqref{eq:youla} has a simple form as well, since one can choose an explicit doubly-coprime factorization as
$$
\begin{aligned}
    \mathbf{M}_l &= (I - \mathbf{G}\mathbf{K}_0)^{-1}, \quad\quad \mathbf{M}_r = -(I - \mathbf{K}_0\mathbf{G})^{-1} \\
     \mathbf{N}_l &= \mathbf{G}(I - \mathbf{K}_0\mathbf{G})^{-1}, \quad\; \mathbf{N}_r = -\mathbf{G}(I - \mathbf{K}_0\mathbf{G})^{-1}, \\
    \mathbf{U}_l &= -I, \quad \mathbf{V}_l = - \mathbf{K}_0, \;\;\;  \mathbf{U}_r = I, \;\;\; \mathbf{V}_r = \mathbf{K}_0.
\end{aligned}
$$
If the plant is open-loop stable, we can choose $\mathbf{K}_0 = 0$.

Unlike the closed-loop parameterizations in Propositions~\ref{prop:slp1}--\ref{prop:slp4}, the Youla parameterization~\eqref{eq:youla} allows 
the parameter $\mathbf{Q}$ to be freely chosen in $\mathcal{RH}_{\infty}$ with no equality constraints. Indeed, any doubly-coprime factorization of the plant can be used to eliminate the affine constraints in Propositions~\ref{prop:slp1}--\ref{prop:slp4} exactly, as shown below.

\begin{proposition} \label{prop:clpyoula}
     Let $\mathbf{U}_r,\mathbf{V}_r,\mathbf{U}_l,\mathbf{V}_l,\mathbf{M}_r,\mathbf{M}_l,\mathbf{N}_r,\mathbf{N}_l$ be any  doubly-coprime factorization of $\mathbf{G}$. For any $\mathbf{Q} \in \mathcal{RH}_\infty$, the following transfer matrices
    \begin{equation} 
    \begin{aligned}
    \bm{\Phi}_{yy}&=(\mathbf{U}_r-\mathbf{N}_r\mathbf{Q})\mathbf{M}_l\,,\\
    \bm{\Phi}_{uy}&=(\mathbf{V}_r-\mathbf{M}_r\mathbf{Q})\mathbf{M}_l\,,\\
    \bm{\Phi}_{yu}&=(\mathbf{U}_r-\mathbf{N}_r\mathbf{Q})\mathbf{N}_l\,,\\
    \bm{\Phi}_{uu}&=I+(\mathbf{V}_r-\mathbf{M}_r\mathbf{Q})\mathbf{N}_l\,, \\
    \end{aligned}
    \end{equation}
    and $\bm{\Phi}_{xx} = (zI - A)^{-1}\! +
    (zI - A)^{-1}B\bm{\Phi}_{uy}C(zI - A)^{-1}, $

    \begin{equation}
    \begin{aligned}
            \bm{\Phi}_{ux} & = \bm{\Phi}_{uy}C(zI - A)^{-1},\\
            \bm{\Phi}_{xy} &= (zI - A)^{-1}B\bm{\Phi}_{uy}, \\
            \bm{\Phi}_{xu} &= (zI - A)^{-1}B\bm{\Phi}_{uu},\\
            \bm{\Phi}_{yx} &= \bm{\Phi}_{uu}C(zI - A)^{-1},
    \end{aligned}
    \end{equation}
satisfy the affine constraints~\eqref{eq:slp1_constraint},~\eqref{eq:iop},~\eqref{eq:slp3constraint},~\eqref{eq:slp4constraint}.
\end{proposition}
The proof is based on direct verification, which is omitted here; see~\cite{zheng2019equivalence} for further discussions on the equivalence of the Youla parameterization, the IOP, and the SLP. We note that a doubly-coprime factorization can be found in the state-space domain~\cite{nett1984connection}, and this pre-process might introduce  numerical issues that affect closed-loop stability, which is beyond the scope of this paper.

\section{Case studies}
\label{section:example}

In this section, we present a case study of optimal controller synthesis for \emph{open-loop stable} plants using Propositions~\ref{prop:slp1}-\ref{prop:slp4}. We show that the optimal controller synthesis problem can be cast into a quadratic program (QP) after imposing the FIR constraint.\footnote{Code is available at \url{https://github.com/soc-ucsd/h2_clp}.}

\subsection{Application to optimal controller synthesis}


Consider the optimal controller synthesis~\eqref{eq:OptimalControlv2}. Using a change of variables, as suggested in Propositions~\ref{prop:slp1}-\ref{prop:slp4}, it is equivalent to replace $\mathbf{K} \in \mathcal{C}_{\text{stab}}$ with the affine constraints~\eqref{eq:slp1_constraint},~\eqref{eq:iop},~\eqref{eq:slp3constraint}, or~\eqref{eq:slp4constraint}. It remains to reformulate the cost function in terms of these new variables. Simple algebra shows that
$$
\begin{aligned}
   (I-\mathbf{G}\mathbf{K})^{-1} &= \mathbf{\Phi}_{yy} = C\mathbf{\Phi}_{xy}+ I, \\
    \mathbf{K}(I-\mathbf{G}\mathbf{K})^{-1}  &= \mathbf{\Phi}_{uy}, \\
    (I-\mathbf{K}\mathbf{G})^{-1}  &= \mathbf{\Phi}_{uu} = \mathbf{\Phi}_{ux}B + I,
\end{aligned}
$$
and
$
    (I-\mathbf{G}\mathbf{K})^{-1}\mathbf{G} = \mathbf{\Phi}_{yu} = C\mathbf{\Phi}_{xx}B = C\mathbf{\Phi}_{xu} = \mathbf{\Phi}_{yx}B.
$

Therefore, problem~\eqref{eq:OptimalControlv2} is equivalent to any of the following convex optimization problems~\eqref{eq:OptimalControlsls}-\eqref{eq:OptimalControlty4}  corresponding to Propositions~\ref{prop:slp1}--\ref{prop:slp4}, respectively.
\begin{equation} \label{eq:OptimalControlsls}
    \begin{aligned}
        \min  \quad &  \left\|\begin{bmatrix} Q^{\frac{1}{2}} & 0 \\ 0& R^{\frac{1}{2}} \end{bmatrix} \begin{bmatrix} C\mathbf{\Phi}_{xy}+ I & C\mathbf{\Phi}_{xx}B\\ \Phi_{uy}  & \mathbf{\Phi}_{ux}B + I \end{bmatrix}\right\|^2_{\mathcal{H}_2} \\
        \text{s.t.} \quad 	 &
	 \mathbf{\Phi}_{xx},\mathbf{\Phi}_{xy},\mathbf{\Phi}_{ux},\mathbf{\Phi}_{uy}\;\; \text{satisfy}~\eqref{eq:slp1_constraint}.
    \end{aligned}
\end{equation}
\begin{equation} \label{eq:OptimalControliop}
    \begin{aligned}
        \min  \quad &  \left\|\begin{bmatrix} Q^{\frac{1}{2}} & 0 \\0 & R^{\frac{1}{2}} \end{bmatrix} \begin{bmatrix} \mathbf{\Phi}_{yy} & \mathbf{\Phi}_{yu}\\ \mathbf{\Phi}_{uy}  & \mathbf{\Phi}_{uu} \end{bmatrix}\right\|^2_{\mathcal{H}_2} \\
        \text{s.t.} \quad & \mathbf{\Phi}_{yy}, \mathbf{\Phi}_{yu}, \mathbf{\Phi}_{uy}, \mathbf{\Phi}_{uu}\;\; \text{satisfy}~\eqref{eq:iop}.
    \end{aligned}
\end{equation}
 \begin{equation} \label{eq:OptimalControlty3}
    \begin{aligned}
        \min  \quad &  \left\|\begin{bmatrix} Q^{\frac{1}{2}} & 0\\ 0& R^{\frac{1}{2}} \end{bmatrix} \begin{bmatrix} \mathbf{\Phi}_{yy} & \mathbf{\Phi}_{yx}B\\ \Phi_{uy}  & \mathbf{\Phi}_{ux}B + I \end{bmatrix}\right\|^2_{\mathcal{H}_2} \\
        \text{s.t.} \quad 	 &
	 \mathbf{\Phi}_{yy},\mathbf{\Phi}_{uy},\mathbf{\Phi}_{yx},\mathbf{\Phi}_{ux}\;\;\text{satisfy}~\eqref{eq:slp3constraint}.
    \end{aligned}
\end{equation}
 \begin{equation} \label{eq:OptimalControlty4}
    \begin{aligned}
        \min \quad &  \left\|\begin{bmatrix} Q^{\frac{1}{2}} & 0\\0 & R^{\frac{1}{2}} \end{bmatrix} \begin{bmatrix} C\mathbf{\Phi}_{xy}+I & C\mathbf{\Phi}_{xu}\\ \Phi_{uy}  & \mathbf{\Phi}_{uu} \end{bmatrix}\right\|^2_{\mathcal{H}_2} \\
        \text{s.t.} \quad 	 &
	 \mathbf{\Phi}_{xy},\mathbf{\Phi}_{uy},\mathbf{\Phi}_{xu},\mathbf{\Phi}_{uu}\;\;\text{satisfy}~\eqref{eq:slp4constraint}.
    \end{aligned}
\end{equation}

Note that the $\mathcal{H}_2$ norm of an FIR transfer matrix $\mathbf{H} = \sum_{k=1}^T \frac{1}{z^k}H_k$ admits the following expression
$
    \|\mathbf{H}\|^2_{\mathcal{H}_2} = \sum_{k=1}^T \text{Trace}(H_k^\tr H_k).
$
Thus, after imposing the FIR constraint on the decision variables, problems~\eqref{eq:OptimalControlsls}-\eqref{eq:OptimalControlty4} can be reformulated into QPs, for which very efficient solvers exist. 


\begin{figure}[t]
    \centering
    \includegraphics[width = 0.36\textwidth]{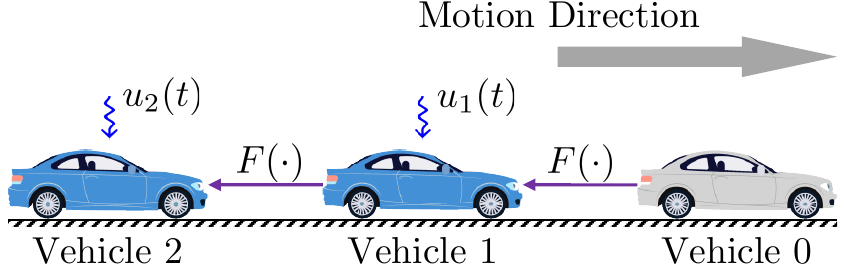}
    \caption{Each vehicle has a pre-existing car-following dynamics $F(\cdot)$ and the goal is to design an additional input $u_i(t), i = 1, 2$ to improve the car-following performance. }
    \label{fig:example}
\end{figure}

\subsection{Numerical experiments}

Here, we use a car-following control scenario~\cite{zheng2015stability} (see Figure~\ref{fig:example} for illustration) to demonstrate the numerical performance of the parameterizations in Propositions~\ref{prop:slp1}-\ref{prop:slp4}.

\textit{Modelling:} We denote the position and velocity of vehicle $i$ as $p_i$ and $v_i$. The spacing of vehicle $i$, \emph{i.e.}, its relative distance from vehicle $i-1$, is defined as $s_i = p_{i-1} - p_i, i=1, 2$. Without loss of generality, the vehicle length is ignored. It is assumed that the leading vehicle 0 runs at a constant velocity $v_0$. Each vehicle has pre-existing car-following dynamics, and we aim to design an additional control signal $u_i(t)$ to improve the car-following performance, \emph{i.e.},
\begin{equation} \label{Eq:HDVModel}
    \dot{v}_i(t) = F(s_i(t),\dot{s}_i(t),v_i(t)) + u_i(t),
\end{equation}
where $\dot{s}_i(t) = v_{i-1}(t) - v_i(t)$, and $F(\cdot)$ characterizes the driver's natural car-following behavior (see~\cite{orosz2010traffic} for details). In an equilibrium car-following state, each vehicle moves with the same equilibrium velocity, \emph{i.e.}, $v_i(t) = v_0, \dot{s}_i(t) = 0$, for $i = 1, 2$. Assuming that each vehicle has a small perturbation from the equilibrium state $(s_i^*,v^*)$, we define the error state between actual and equilibrium state of vehicle $i$ as
$
x_i(t)=\begin{bmatrix}\tilde{s}_i(t),\tilde{v}_i(t)\end{bmatrix}^\tr=\begin{bmatrix}s_i(t)-s_i^*,v_i(t)-v^*\end{bmatrix}^\tr.
$
Applying the first-order Taylor expansion to \eqref{Eq:HDVModel}, we can derive a linearized model for each vehicle ($i=1, 2$)
\begin{equation*}
\begin{cases}
\dot{\tilde{s}}_i(t)=\tilde{v}_{i-1}(t)-\tilde{v}_i(t),\\
\dot{\tilde{v}}_i(t)=\alpha_{1}\tilde{s}_i(t)-\alpha_{2}\tilde{v}_i(t)+\alpha_{3}\tilde{v}_{i-1}(t) + u_i(t),\\
\end{cases}
\end{equation*}
with $\alpha_{1} = \frac{\partial F}{\partial s_i}, \alpha_{2} = \frac{\partial F}{\partial \dot{s}_i} - \frac{\partial F}{\partial v_i}, \alpha_{3} = \frac{\partial F}{\partial \dot{s}_i}$ evaluated at the equilibrium state. Assuming that we can measure the relative spacing, we arrive at the following state-space model
\begin{equation} \label{eq:CarDynamics}
    \begin{aligned}
        \dot{x} &= \begin{bmatrix} P_1 & 0 \\
        P_2 & P_1\end{bmatrix}  x  + \begin{bmatrix}B_1 & 0 \\ 0 & B_1\end{bmatrix} \left( u + \delta_u\right), \\
        y &= \begin{bmatrix} C_1 & 0 \\ 0 & C_1 \end{bmatrix}x + \delta_y,
    \end{aligned}
\end{equation}
where  $x = \begin{bmatrix} x_1^\tr & x_2^\tr \end{bmatrix}^\tr, u = \begin{bmatrix} u_1 & u_2 \end{bmatrix}^\tr, y = \begin{bmatrix} \tilde{s}_1(t) & \tilde{s}_2(t) \end{bmatrix}^\tr$, $\delta_u$ and $\delta_y$ are control input noise and measurement noise, respectively, and
$$
    P_1 = \begin{bmatrix} 0 & -1\\ \alpha_1 & -\alpha_2 \end{bmatrix},     P_2 = \begin{bmatrix} 0 & 1\\ 0 & \alpha_3 \end{bmatrix}, B_1 = \begin{bmatrix}0 \\1 \end{bmatrix}, C_1 = \begin{bmatrix}1 & 0 \end{bmatrix}.
$$
The objective is to design $u_i(t)$ to regulate the spacing error $s_i(t)-s_i^*$ and velocity error $v_i(t)-v^*$ based on the output information $y(t)$. This problem can be formulated into~\eqref{eq:H2_s2} in the discrete-time domain.

\begin{figure}[t]
    \centering
     \includegraphics[width=0.34\textwidth]{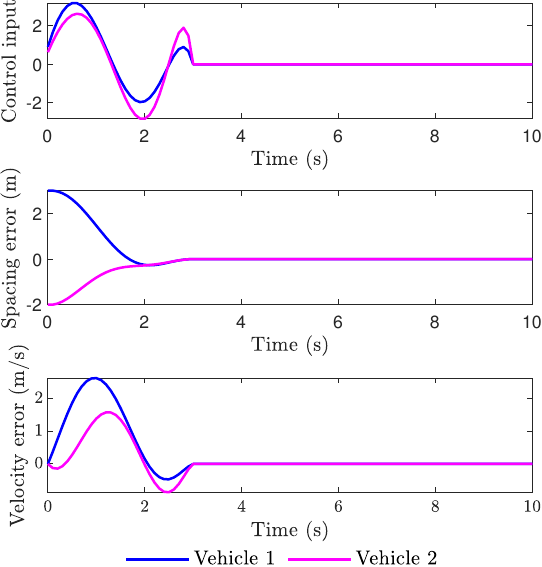}
    \caption{Responses using the controllers from~\eqref{eq:OptimalControlsls}-\eqref{eq:OptimalControlty4} with FIR length $T = 30$.}
    \label{fig:ExampleN30}
\end{figure}

\begin{table}[t]
\caption{$\mathcal{H}_2$ norm for different FIR lengths when solving the car-following problem.}\label{tab:h2norm}
{\small
	\centering
	\begin{tabular}{rrrrrrrr}
	\toprule
		FIR $T$ &10  &  15 &   20&    25 &   30    &50 &   75\\ 
		\midrule
		$\mathcal{H}_2$ norm  &  54.20 &   17.41&    7.56&    4.08&    2.49 &    2.03&    2.02 \\
	\bottomrule
	\end{tabular}
	}
	 \scriptsize
   \newline
   \raggedright
   $^\ddagger$: The $\mathcal{H}_2$ norms from~\eqref{eq:OptimalControlsls}-\eqref{eq:OptimalControlty4} have no difference up to four significant figures. \\
   $\,^\dagger$: The true $\mathcal{H}_2$ norm from \texttt{h2syn} in MATLAB is 2.02. \\
\end{table}

\textit{Numerical results:} In our numerical simulations, the car-following parameters $\alpha_1 = 0.94, \alpha_2 = 1.5, \alpha_3 = 0.9$ are chosen according to the setup in~\cite{zheng2020smoothing}, and the open-loop system is stable. Using a forward Euler-discretization of~\eqref{eq:CarDynamics} with a sampling time of $dT = 0.1$s, we formulate the corresponding optimal controller synthesis problem~\eqref{eq:H2_s2} in discrete-time with $Q = I$ and $R = I$. This can be solved via any of the convex problems~\eqref{eq:OptimalControlsls}-\eqref{eq:OptimalControlty4}. We varied the FIR length $T$ from $10$ to $75$, and the results are listed in Table I. {As expected, when increasing the FIR length, the optimal cost from~\eqref{eq:OptimalControlsls}-\eqref{eq:OptimalControlty4} converges to the true value returned by the standard synthesis \texttt{h2syn} in MATLAB.} Given an initial state $x_0 = [3, 0, -2, 0]^\tr$, Figure~\ref{fig:ExampleN30} shows the time-domain responses\footnote{The responses from~\eqref{eq:OptimalControlsls}-\eqref{eq:OptimalControlty4} have no visible difference.} of the closed-loop system using the resulting controllers from~\eqref{eq:OptimalControlsls}-\eqref{eq:OptimalControlty4} when the FIR length is $T = 30$. By design, the closed-loop system converges to the equilibrium state within 3 seconds.  
For the same initial state, Figure~\ref{fig:ExampleN75} shows the time-domain responses of the closed-loop system when the FIR length is $T = 75$, where the system converges the equilibrium state within 7.5 seconds with lower peak values during the transient process compared to the case $T = 30$.

\begin{figure}[t]
    \centering
     \includegraphics[width=0.34\textwidth]{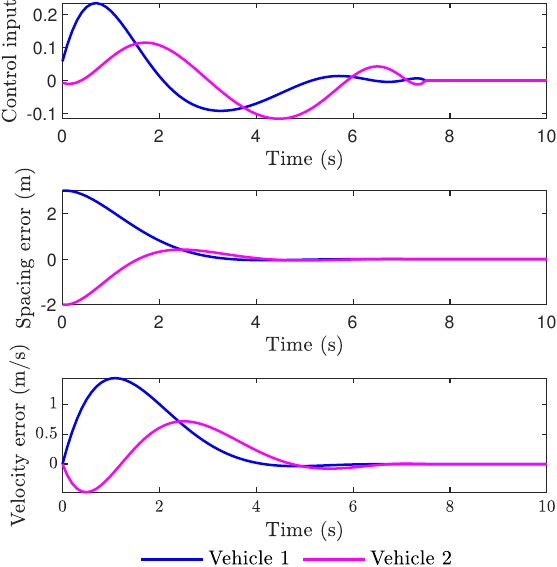}
    \caption{Responses using the controllers from~\eqref{eq:OptimalControlsls}-\eqref{eq:OptimalControlty4} with FIR length $T = 75$.}
    \label{fig:ExampleN75}
\end{figure}

\vspace{5mm}

\section{Conclusions}
\label{section:conclusion}

In this paper, we have characterized all possible parameterizations for the set of stabilizing controllers using closed-loop maps. We have revealed two other parallel choices beyond the recent notions of SLP~\cite{wang2019system} and IOP~\cite{furieri2019input}. In fact, our analysis allows to treat the SLP~\cite{wang2019system} and the IOP~\cite{furieri2019input} {in a unified way}. 
After imposing the FIR approximation, the ability of the four parameterizations for {encoding} $\mathcal{C}_{\text{stab}}$ becomes different, and the IOP enjoys the best approximation ability. These closed-loop parameterizations can avoid computing the doubly co-prime factorization of the plant \emph{a priori}, but instead require imposing a set of affine constraints for achievable closed-loop responses. We have discussed 
two numerically robust scenarios: the SLP in the state feedback case and the IOP for open-loop stable plants.

One future direction is to address decentralized control, \emph{e.g.}, the notion of quadratic invariance (QI)~\cite{rotkowitz2006characterization} and sparsity invariance (SI)~\cite{furieri2019sparsity}, using different parameterizations. Also, similar to SLP~\cite{dean2017sample,dean2018regret} and Youla~\cite{simchowitz2020improper, lale2020logarithmic,furieri2019learning}, it will be extremely interesting to investigate the features of different parameterizations in robust synthesis for uncertain systems and their applications in learning-based control. Finally, we have established that  closed-loop parameterizations are more subtle in practice for open-loop unstable plants with output feedback, and further investigation is needed to unravel a more precise and thorough understanding of related aspects.

\vspace{2mm}
{\small
\noindent\textbf{Acknowledgement:} The authors would like to thank Nikolai Matni, James Anderson and John C Doyle for several insightful discussions, particularly around the robustness of the SLS framework. The authors thank John C Doyle for his encouragement to find a simple SISO example, eventually leading to~\eqref{eq:SLPcounterex}. We also thank three anonymous reviewers and Associate Editor, whose suggestions improved the quality of this work.
}

\vspace{5mm}

\newpage
\appendix
\noindent\textbf{Appendix}

\section{Proof of Proposition~\ref{prop:slp3}} \label{App:proposition3}
\emph{Statement 1}: Given any $\mathbf{K} \in \mathcal{C}_{\text{stab}}$, it is easy to derive that the closed-loop responses~\eqref{eq:slp3} are
        $$
            \begin{aligned}
            \bm{\Phi}_{yx} &= (I - \mathbf{G}\mathbf{K})^{-1}C(zI - A)^{-1},\\ \bm{\Phi}_{yy} &= (I - \mathbf{G}\mathbf{K})^{-1},\\
            \bm{\Phi}_{ux} &=\mathbf{K}(I - \mathbf{G}\mathbf{K})^{-1}C(zI - A)^{-1}, \\ \bm{\Phi}_{uy} &= \mathbf{K}(I - \mathbf{G}\mathbf{K})^{-1},
            \end{aligned}
        $$
        which are all stable by definition. Then, it is not difficult to verify that
        $$
        \begin{aligned}
            \bm{\Phi}_{yx} - \mathbf{G}\bm{\Phi}_{ux}
          & 
            = C(zI - A)^{-1}, \\
            \bm{\Phi}_{yy} - \mathbf{G}\bm{\Phi}_{uy}
            &= (I - \mathbf{G}\mathbf{K})^{-1}-  \mathbf{G}\mathbf{K}(I - \mathbf{G}\mathbf{K})^{-1} = I,
        \end{aligned}
        $$
         and that
        $$
            \begin{aligned}
            \bm{\Phi}_{yx}(zI - A) - \bm{\Phi}_{yy}C &= 0,  \\
            \bm{\Phi}_{ux}(zI - A) - \bm{\Phi}_{uy}C &= 0.
            \end{aligned}
        $$
        Therefore, the closed-loop responses $\bm{\Phi}_{yx}, \bm{\Phi}_{yy},
            \bm{\Phi}_{ux} , \bm{\Phi}_{uy}$ satisfy~\eqref{eq:slp3constraint}.

        \emph{Statement 2}: Consider any $ \bm{\Phi}_{yx}, \bm{\Phi}_{yy},
            \bm{\Phi}_{ux} , \bm{\Phi}_{uy}$ satisfying~\eqref{eq:slp3constraint}. Since $\bm{\Phi}_{yy} = I + \mathbf{G}\bm{\Phi}_{uy}$ and $\mathbf{G}$ is strictly proper, we know that $\bm{\Phi}_{yy}$ is always invertible. Let $\mathbf{K} = \bm{\Phi}_{uy}\bm{\Phi}_{yy}^{-1}$.
            We now verify the resulting closed-loop responses in~\eqref{eq:slp3} are all stable. In particular, we have
             $$
             \mathbf{y} = (I - \mathbf{G}\mathbf{K})^{-1}C(zI - A)^{-1} \bm{\delta}_x,
             $$
             and with $\mathbf{K} = \bm{\Phi}_{uy}\bm{\Phi}_{yy}^{-1}$, we have
             $$
        \begin{aligned}
            &(I - \mathbf{G}\bm{\Phi}_{uy}\bm{\Phi}_{yy}^{-1})^{-1}C(zI - A)^{-1} \\
            =& \bm{\Phi}_{yy}(\bm{\Phi}_{yy} - \mathbf{G}\bm{\Phi}_{uy})^{-1}C(zI - A)^{-1} \\
            = &\bm{\Phi}_{yy}C(zI - A)^{-1} \\
            = &\bm{\Phi}_{yx} \in \mathcal{RH}_{\infty}, \\
            \end{aligned}
            $$
            where the 
            equalities follow from the fact that $ \bm{\Phi}_{yx}, \bm{\Phi}_{yy},
            \bm{\Phi}_{ux}$, $\bm{\Phi}_{uy}$ satisfy~\eqref{eq:slp3constraint}.
            Also, we have that
            $$
                \mathbf{y} = (I - \mathbf{G}\bm{\Phi}_{uy}\bm{\Phi}_{yy}^{-1})^{-1}\bm{\delta}_y = \bm{\Phi}_{yy} \bm{\delta}_y.
            $$
            Similarly, we can show that
        $$
        \begin{aligned}
          \mathbf{u} &= \mathbf{K}(I - \mathbf{P}_{22}\mathbf{K})^{-1}C_2(zI - A)^{-1} \bm{\delta}_x =  \bm{\Phi}_{ux} \bm{\delta}_x,
          \\
            \mathbf{u}&=\mathbf{K}(I - \mathbf{P}_{22}\mathbf{K})^{-1}\bm{\delta}_y =  \bm{\Phi}_{uy}\bm{\delta}_y.
        \end{aligned}
        $$
        Therefore, we have proved that
        $
            \left(\begin{bmatrix}
            \bm{\delta}_x \\
            \bm{\delta}_y  \\
        \end{bmatrix} \rightarrow \begin{bmatrix}
            \mathbf{y} \\
            \mathbf{u} \\
        \end{bmatrix} \right)\in \mathcal{RH}_{\infty},
        $
        using the controller $\mathbf{K} = \bm{\Phi}_{uy}\bm{\Phi}_{yy}^{-1}$.
        By Theorem~\ref{Theo:mainresult}, we know $\mathbf{K} = \bm{\Phi}_{uy}\bm{\Phi}_{yy}^{-1} \in \mathcal{C}_{\text{stab}}.$

\section{Proof of Theorem~\ref{theo:FIR}} \label{App:theorem2}
    The direction $(i) \Rightarrow (ii), (iii), (iv), (v)$ is true by definition. For any controller $\mathbf{K}$, the closed-loop responses are given in~\eqref{eq:closedloop} and~\eqref{eq:responses}.

    We now prove $(ii) \Rightarrow (i)$. Suppose we have
$$
    \left(\begin{bmatrix}
            \bm{\delta}_x \\
            \bm{\delta}_y  \\
        \end{bmatrix} \rightarrow \begin{bmatrix}
            \mathbf{x} \\
            \mathbf{u} \\
        \end{bmatrix}\right) = \begin{bmatrix} \bm{\Phi}_{xx} & \bm{\Phi}_{xy}\\
        \bm{\Phi}_{ux} & \bm{\Phi}_{uy}\end{bmatrix} \in \mathcal{F}_T.
$$
From~\eqref{eq:responses}, it is not difficult to check $\bm{\Phi}_{xu} = \bm{\Phi}_{xx}B \in \mathcal{F}_T$, and
$$
\begin{aligned}
    \bm{\Phi}_{yx} &= C\bm{\Phi}_{xx}\in \mathcal{F}_T, &
    \bm{\Phi}_{yy} &= C\bm{\Phi}_{xy}+ I \in \mathcal{F}_T \\
    \bm{\Phi}_{yu} &= C\bm{\Phi}_{xx}B \in \mathcal{F}_T, &
    \bm{\Phi}_{uu} &= \bm{\Phi}_{ux}B + I \in \mathcal{F}_T. \\
\end{aligned}
$$

\noindent This means that the statement (i) is true. Similar arguments can prove $(iii) \Rightarrow (v)$ and $(iv) \Rightarrow (v)$.

        Finally, if $(A, B, C)$ and $(A_k, B_k, C_k)$ are both controllable and observable, we prove that $(v) \Rightarrow (i)$. According to~\eqref{eq:clstatespace}, we have the following state-space realization
        \begin{equation} \label{eq:clyu2yu}
            \left(\begin{bmatrix}
            \bm{\delta}_y \\
            \bm{\delta}_u  \\
        \end{bmatrix} \rightarrow \begin{bmatrix}
            \mathbf{y} \\
            \mathbf{u} \\
        \end{bmatrix}\right) =  \hat{C}_1 (zI - A_{\text{cl}})^{-1}\hat{B}_2  + \begin{bmatrix}
    I & 0 \\
    D_k & I
    \end{bmatrix},
        \end{equation}
    where $\hat{C}_1$ and $\hat{B}_2$ are defined in~\eqref{eq:clmatrices}.
    We can show that
    $
    (A_{\text{cl}},\hat{B}_2,\hat{C}_1)
    $
    is controllable and observable (see Appendix~\ref{app:controllability}). This means that the eigenvalues of $A_{\text{cl}}$ are the same as the poles of the transfer matrices~\cite[Chapter 3]{zhou1996robust}. Therefore, if the statement $(v)$ 
    is true, then the closed-loop matrix  $A_{\text{cl}}$ only has zero eigenvalues and no eigenvalues of $A_{\text{cl}}$ is hidden from the input-output behavior. 
    This completes the proof.

\section{Proof of stabilizability of~\eqref{eq:statecontrollable}}\label{app:statefeedback}
Consider a feedback gain
$
K = \begin{bmatrix} K_1 & K_2 \end{bmatrix},
$
where $K_1 \in \mathbb{R}^{n \times n},K_2 \in \mathbb{R}^{n \times p}$.
We have
$$
\begin{aligned}
    &\begin{bmatrix}A+BD_k & BC_k \\ B_k & A_k\end{bmatrix} +  \begin{bmatrix}I \\0 \end{bmatrix}\begin{bmatrix} K_1 & K_2 \end{bmatrix} \\
    = &\begin{bmatrix}A+BD_k + K_1 & BC_k + K_2 \\ B_k & A_k\end{bmatrix}.
\end{aligned}
$$
Since $(A_k,B_k)$ is stabilizable, there exists $F_k\in \mathbb{R}^{n \times q}$ such that $A_k + B_kF_k$ is stable.
By choosing
$$
    \begin{aligned}
        K_1 &= - A - BD_k -  I + F_kB_k\\
        K_2 & = -BC_k - (A+BD_k + K_1)F_k + F_k(A_k+B_kF_k),
    \end{aligned}
$$
it can be easily verify that
$$
\begin{aligned}
      &\begin{bmatrix}I & F_k \\ 0 & I\end{bmatrix}^{-1}\begin{bmatrix}A+BD_k + K_1 & BC_k + K_2 \\ B_k & A_k\end{bmatrix}  \begin{bmatrix}I & F_k \\ 0 & I\end{bmatrix} \\
      &=  \begin{bmatrix}-I & 0 \\ B_k & A_k+B_kF_k\end{bmatrix}.
      \end{aligned}
$$
is stable. Since the similarity transformation does not change eigenvalues,
there exist $K_1, K_2$ such that
$$
      \begin{bmatrix}A+BD_k + K_1 & BC_k + K_2 \\ B_k & A_k\end{bmatrix}
$$
is stable. This completes the proof. \hfill \qed

\section{Proof of Proposition~\ref{prop:slp4}} \label{app:proof_p4}
\emph{Statement 1}: Given $\mathbf{K} \in \mathcal{C}_{\text{stab}}$, it  is  easy  to derive that the closed-loop responses~\eqref{eq:slp4} are
\begin{equation*} 
    \begin{aligned}
       \bm{\Phi}_{xy} &= \bm{\Phi}_{xx}B\mathbf{K},  &\bm{\Phi}_{xu} &= \bm{\Phi}_{xx}B,  \\
       \bm{\Phi}_{uy} &= \mathbf{K}(C\bm{\Phi}_{xx}B\mathbf{K} + I),  &\bm{\Phi}_{uu} &= \mathbf{K}C\bm{\Phi}_{xx}B + I,  \\
    \end{aligned}
\end{equation*}
 where  $\bm{\Phi}_{xx} = (zI - A - B\mathbf{K}C)^{-1} \in \mathcal{RH}_{\infty}$. They are all stable by definition. Then, it is not difficult to verify that
 $$
    \begin{aligned}
        &(zI - A) \bm{\Phi}_{xy} - B \bm{\Phi}_{uy} \\
        =\; & (zI - A)\bm{\Phi}_{xx}B\mathbf{K} - B\mathbf{K}(C\bm{\Phi}_{xx}B\mathbf{K} + I) \\
        =\;&((zI - A)\bm{\Phi}_{xx} - B\mathbf{K}C\bm{\Phi}_{xx} - I)B\mathbf{K}
         = \;0,
         \end{aligned}
         $$
         and
         $$
         \begin{aligned}
        &(zI - A) \bm{\Phi}_{xu} - B \bm{\Phi}_{uu} \\
        = \; &(zI - A)\bm{\Phi}_{xx}B - B(C\bm{\Phi}_{xx}B\mathbf{K} + I) \\
        =\;&((zI - A)\bm{\Phi}_{xx} - B\mathbf{K}C\bm{\Phi}_{xx} - I)B
         =\;  0,
    \end{aligned}
 $$
and that
 $$
    \begin{aligned}
       -\bm{\Phi}_{xy}\mathbf{G} + \bm{\Phi}_{xu}
        &=(zI - A)^{-1}B,\\
        -\bm{\Phi}_{uy}\mathbf{G} + \bm{\Phi}_{uu}
        &= I.
    \end{aligned}
 $$
Therefore, the closed-loop responses $\bm{\Phi}_{yx}, \bm{\Phi}_{uy},
            \bm{\Phi}_{xu}, \bm{\Phi}_{uu}$ satisfy~\eqref{eq:slp4constraint}.

 \emph{Statement 2}: Consider any $ \bm{\Phi}_{xy}, \bm{\Phi}_{uy},
            \bm{\Phi}_{xu} , \bm{\Phi}_{uu}$ satisfying~\eqref{eq:slp4constraint}. Since $\bm{\Phi}_{uu} = I + \bm{\Phi}_{uy}\mathbf{G}$, $\bm{\Phi}_{uu}$ is always invertible. Let $\mathbf{K} = \bm{\Phi}_{uu}^{-1}\bm{\Phi}_{uy}$.
            We now verify that the resulting closed-loop responses~\eqref{eq:slp4} are all stable. In particular, we have
            $$
                \mathbf{x} = (zI - A - B\mathbf{K}C)^{-1}B \bm{\delta}_u,
            $$
            and with the controller $\mathbf{K} = \bm{\Phi}_{uu}^{-1}\bm{\Phi}_{uy}$, we have
            $$
                \begin{aligned}
                     &(zI - A - B\bm{\Phi}_{uu}^{-1}\bm{\Phi}_{uy}C)^{-1}B   \\
                     =&(zI -A)^{-1}(I - B\bm{\Phi}_{uu}^{-1}\bm{\Phi}_{uy}C(zI -A)^{-1})^{-1}B  \\
                     =&(zI -A)^{-1}B(I - \bm{\Phi}_{uu}^{-1}\bm{\Phi}_{uy}\mathbf{G})^{-1} \\
                     =&(zI -A)^{-1}B(\bm{\Phi}_{uu} - \bm{\Phi}_{uy}\mathbf{G})^{-1}\bm{\Phi}_{uu} \\
                     =&\bm{\Phi}_{xu} \in \mathcal{RH}_{\infty},
                \end{aligned}
            $$
            where the last equality follows from the fact that $ \bm{\Phi}_{xy}, \bm{\Phi}_{uy},
            \bm{\Phi}_{xu} , \bm{\Phi}_{uu}$ satisfy~\eqref{eq:slp4constraint}.
    Also, it is not difficult to derive that
 $$
 \begin{aligned}
    \mathbf{x} &= (zI - A - B\mathbf{K}C)^{-1}B\mathbf{K}\bm{\delta}_y = \bm{\Phi}_{xy}\bm{\delta}_y, \\
    \mathbf{u} &= \mathbf{K}(C\bm{\Phi}_{xx}B\mathbf{K} + I)\bm{\delta}_y = \bm{\Phi}_{uy}\bm{\delta}_y, \\
    \mathbf{u} &= (\mathbf{K}C(zI - A - B\mathbf{K}C)^{-1}B +I)\bm{\delta}_u = \bm{\Phi}_{uu}\bm{\delta}_u. \\
 \end{aligned}
 $$
 Thus, we have proved that
        $$
            \left(\begin{bmatrix}
            \bm{\delta}_y \\
            \bm{\delta}_u        \end{bmatrix} \rightarrow \begin{bmatrix}
            \mathbf{x} \\
            \mathbf{u} \\
        \end{bmatrix}\right) \in \mathcal{RH}_{\infty}.
        $$
        By Theorem~\ref{Theo:mainresult}, we {conclude that} $\mathbf{K} =\bm{\Phi}_{uu}^{-1}  \bm{\Phi}_{uy}\in \mathcal{C}_{\text{stab}}$.
\qed

\section{Controllability and observability of~\eqref{eq:clyu2yu}} \label{app:controllability}

\begin{lemma}[\cite{zhou1996robust}]
    The following statements are equivalent:
    \begin{enumerate}
        \item $(A,B)$ is controllable;
        \item $(A+BF,B)$ is controllable for any compatible matrix $F$;
        \item $[A-\lambda I, B]$ has full row rank,  $\forall \lambda \in \mathbb{C}$.
    \end{enumerate}
\end{lemma}

We are ready to prove the controllability of $(A_{\text{cl}}, \hat{B}_2)$. First, controllability is invariant under state feedback. We consider
$$
    \begin{bmatrix} A + BD_kC & BC_k \\ B_kC & A_k \end{bmatrix} - \begin{bmatrix} BD_k & B  \\
    B_k & 0
    \end{bmatrix}\begin{bmatrix} C & 0  \\
                     0 & 0
                     \end{bmatrix} =  \begin{bmatrix} A  & BC_k \\ 0 & A_k \end{bmatrix}.
$$
Since $(A, B)$ and $(A_k, B_k)$ are controllable, we have $\text{rank}(\lambda I - A, B) = n$ and $\text{rank}(\lambda I - A_k, B_k) = n_k, \forall \lambda \in \mathbb{C}$. Thus,
$$
    \text{rank}\left(\begin{bmatrix} \lambda I - A  & -BC_k & BD_k & B\\ 0 & \lambda I - A_k & B_k & 0 \end{bmatrix}\right) = n+ n_k, \forall \lambda \in \mathbb{C},
$$
indicating that $(A_{\text{cl}},\hat{B}_2)$ is controllable. The observability of $(A_{\text{cl}},\hat{C}_1)$ can be proved in a similar way.

\section{State space realizations}
\label{app:statespace}

In this section, we use the following system operations very often (see~\cite[Chapter 3.6]{zhou1996robust}). Consider two dynamic systems $$
     \mathbf{G}_i = \left[\begin{array}{c|c} A_i & B_i \\ \hline
    C_i & D_i\end{array}\right], \quad i = 1, 2.
$$
Their inverses are given by
$$
    \mathbf{G}_i^{-1} = \left[\begin{array}{c|c} A_i - B_iD_i^{-1}C_i & -B_iD_i^{-1} \\ \hline
    D_i^{-1}C_i & D_i^{-1}\end{array}\right],\
$$
where we assume $D_i$ is invertible. If the system is strictly proper, then the inverse will be non-proper and there is no state-space realization. The cascade connection of two systems such that $\mathbf{y} = \mathbf{G}_1\mathbf{G}_2\mathbf{u}$ has a realization
\begin{equation} \label{eq:product}
    \mathbf{G}_1\mathbf{G}_2 = \left[\begin{array}{c c|c} A_1 & B_1C_2 & B_1D_2 \\
    0 & A_2 & B_2 \\ \hline
    C_1 & D_1C_2 & D_1D_2 \end{array}\right]
\end{equation}
and a parallel connection $\mathbf{y} = (\mathbf{G}_1 - \mathbf{G}_2)\mathbf{u}$ has a realization
\begin{equation} \label{eq:parallel}
    \mathbf{G}_1 - \mathbf{G}_2 = \left[\begin{array}{c c|c} A_1 & 0 & B_1 \\
    0 & A_2 & B_2 \\ \hline
    C_1 & -C_2 & D_1-D_2 \end{array}\right].
\end{equation}
We note that~\eqref{eq:product} and~\eqref{eq:parallel} are in general not  minimal. 
In addition, we use the following fact: for any invertible matrix $T$ with proper dimension,
\begin{equation} \label{eq:Transformation}
     \mathbf{G}_i = \left[\begin{array}{c|c} A_i & B_i \\ \hline
    C_i & D_i\end{array}\right] =  \left[\begin{array}{c|c} TA_iT^{-1} & TB_i \\ \hline
    C_iT^{-1} & D_i\end{array}\right].
\end{equation}

\subsection*{State space realization of the SLP controller}

    For the SLP controller $\mathbf{K} =  \bm{\Phi}_{uy} - \bm{\Phi}_{ux}\bm{\Phi}_{xx}^{-1}\bm{\Phi}_{xy}$ in Proposition~\ref{prop:slp1}, we assume the system responses $\bm{\Phi}_{uy}, \bm{\Phi}_{ux}, \bm{\Phi}_{xx}, \bm{\Phi}_{xy}$ are FIR transfer matrices of horizon $T$, \emph{i.e.},
    \begin{equation} \label{eq:FIRslp}
    \begin{aligned}
         \bm{\Phi}_{ux} &\!=\! \sum_{t=0}^T M_t \frac{1}{z^t} \in \mathcal{RH}_{\infty}, \bm{\Phi}_{xx} = \sum_{t=0}^T R_t \frac{1}{z^t} \in \mathcal{RH}_{\infty},  \\ \bm{\Phi}_{xy} &\!=\! \sum_{t=0}^T N_t \frac{1}{z^t} \in \mathcal{RH}_{\infty}.
    \end{aligned}
    \end{equation}
Upon defining the following matrices
\begin{equation} \label{eq:hatM}
\begin{aligned}
  \hat{M} &= \begin{bmatrix} M_2 & M_3 & \ldots & M_T\end{bmatrix} \in  \mathbb{R}^{m \times n\hat{T}}, \\
                        \hat{R} &= \begin{bmatrix} R_2 & R_3 & \ldots & R_T\end{bmatrix} \in  \mathbb{R}^{n \times n\hat{T}}, \\
 \hat{N} &= \begin{bmatrix} N_1 & N_2 & \ldots & N_T\end{bmatrix} \in  \mathbb{R}^{n \times pT},
\end{aligned}
\end{equation}
and
$Z_n \in \mathbb{R}^{n\hat{T} \times n \hat{T}}$ as the down shift operator with sub-diagonal containing identity matrices of dimension $n \times n$, and $
    \mathcal{I}_n = [ I_n, 0 , \ldots,  0 ]^\tr \in \mathbb{R}^{n \hat{T} \times n},
$
with $\hat{T} = T - 1$, we have the following result.
\begin{theorem} \label{Theo:SLPstatespace}
Suppose that 
$\bm{\Phi}_{uy}, \bm{\Phi}_{ux}$, $\bm{\Phi}_{xx}, \bm{\Phi}_{xy}$ are FIR transfer matrices with horizon $T$ in~\eqref{eq:FIRiop} and~\eqref{eq:FIRslp}. A state-space realization for the output feedback controller $\mathbf{K} = \bm{\Phi}_{uy} - \bm{\Phi}_{ux}\bm{\Phi}_{xx}^{-1}\bm{\Phi}_{xy}$  is given by 
    \begin{equation} \label{eq:SLScontroller1}
\begin{aligned}
    \mathbf{K} =  \left[\begin{array}{c c |c} Z_n - \mathcal{I}_n\hat{R} & -\mathcal{I}_n\hat{N} & 0  \\  0 & Z_p &  \mathcal{I}_p\\ \hline
 \hat{M}-M_1\hat{R} & \hat{U}-M_1\hat{N} & U_0\end{array}\right].
\end{aligned}
\end{equation}
\end{theorem}

\subsection*{Proof of Theorem~\ref{Theo:IOPstatespace}}

Considering the FIR transfer matrices $\bm{\Phi}_{uy}$ and $\bm{\Phi}_{yy}$ in~\eqref{eq:FIRiop}. By the affine constraint~\eqref{eq:iop}, we always have $Y_0 = I_p$. The following state-space realizations of $\bm{\Phi}_{uy}$ and $\bm{\Phi}_{yy}$ are
   $$
    \bm{\Phi}_{uy} = \left[\begin{array}{c|c} Z_p & \mathcal{I}_p \\ \hline
    \hat{U} & U_0\end{array}\right], \quad   \bm{\Phi}_{yy} = \left[\begin{array}{c|c} Z_p & \mathcal{I}_p \\ \hline
    \hat{Y} & I_p\end{array}\right],
$$
with $\hat{U}$ and $\hat{Y}$ defined in~\eqref{eq:FIRiop}, and $Z_p$ and $\mathcal{I}_p$ defined as 
\begin{equation*} \label{eq:Zp}
\begin{aligned}   Z_p = \begin{bmatrix} 0 & 0 & 0 &\ldots & 0\\
                        I_p & 0 & 0 &\ldots & 0\\
                        0 & I_p & 0 &\ldots &0\\
                        \vdots &\vdots & \ddots & \ddots &\vdots \\
                        0 & 0 & \ldots &  I_p &0\end{bmatrix} \in \mathbb{R}^{pT \times pT},  \mathcal{I}_{p} = \begin{bmatrix} I_p \\0 \\\vdots \\0 \end{bmatrix} \in \mathbb{R}^{pT \times p}.
\end{aligned}
\end{equation*}
Then, we have
$$
\begin{aligned}
   \bm{\Phi}_{uy}\bm{\Phi}_{yy}^{-1} &= \left[\begin{array}{c|c} Z_p & \mathcal{I}_p \\ \hline
    \hat{U} & U_0\end{array}\right] \left[\begin{array}{c|c} Z_p & \mathcal{I}_p \\ \hline
    \hat{Y} & I_p\end{array}\right]^{-1} \\
    &= \left[\begin{array}{c|c} Z_p & \mathcal{I}_p \\ \hline
    \hat{U} & U_0\end{array}\right] \left[\begin{array}{c|c} Z_p -\mathcal{I}_p\hat{Y}  & -\mathcal{I}_p \\ \hline
    \hat{Y} & I_p\end{array}\right] \\
    &= \left[\begin{array}{cc|c} Z_p & \mathcal{I}_p\hat{Y} & \mathcal{I}_p \\
    0 & Z_p - \mathcal{I}_p\hat{Y} & -\mathcal{I}_p \\\hline
    \hat{U} & U_0\hat{Y} & U_0\end{array}\right].
\end{aligned}
$$
By defining a transformation
\begin{equation*} 
    T = \begin{bmatrix}
        I & I \\
        0 & I
    \end{bmatrix}, T^{-1} = \begin{bmatrix}
        I & -I \\
        0 & I
    \end{bmatrix},
\end{equation*}
with compatible dimension, and according to~\eqref{eq:Transformation}, we have \begin{equation*} 
     \bm{\Phi}_{uy}\bm{\Phi}_{yy}^{-1} = \left[\begin{array}{cc|c} Z_p & 0 & 0 \\
    0 & Z_p - \mathcal{I}_p\hat{Y} & -\mathcal{I}_p \\\hline
    \hat{U} & U_0\hat{Y}- \hat{U} & U_0\end{array}\right] = \left[\begin{array}{c|c} Z_p - \mathcal{I}_p\hat{Y} & -\mathcal{I}_p \\\hline
    U_0\hat{Y}- \hat{U} & U_0\end{array}\right].
\end{equation*}
In the last step, we have removed some uncontrollable and unobservable modes.

\subsection*{Proof of Theorem~\ref{Theo:SLPstatespace}}
First, similar to the realization of $\bm{\Phi}_{uy}\bm{\Phi}_{yy}^{-1}$, we have
$$
\begin{aligned}
    \bm{\Phi}_{ux}\bm{\Phi}_{xx}^{-1} = (z\bm{\Phi}_{ux})(z\bm{\Phi}_{xx})^{-1} =  \left[\begin{array}{c |c} Z_n - \mathcal{I}_n\hat{R} & -\mathcal{I}_n  \\ \hline
  M_1\hat{R}-\hat{M} & M_1\end{array}\right],
\end{aligned}
$$
where $\hat{M}$ and $\hat{R}$ are defined in~\eqref{eq:hatM}, and $Z_n$ and $\mathcal{I}_n$ are defined as
\begin{equation*}
\begin{aligned}
    Z_n = \begin{bmatrix} 0 & 0 & 0 &\ldots & 0\\
                       I_n & 0 & 0 &\ldots & 0\\
                        0 & I_n & 0 &\ldots &0\\
                        \vdots &\vdots & \ddots & \ddots &\vdots \\
                        0 & 0 & \ldots &  I_n &0\end{bmatrix} \in \mathbb{R}^{n\hat{T} \times n\hat{T}},  \mathcal{I}_{n} = \begin{bmatrix} I_n \\0 \\\vdots \\0 \end{bmatrix} \in \mathbb{R}^{n\hat{T}\times n}.
    \end{aligned}
\end{equation*}
Then, we have
$$
\begin{aligned}
    &\bm{\Phi}_{uy} - (z\bm{\Phi}_{ux})(z\bm{\Phi}_{xx})^{-1}\bm{\Phi}_{xy} \\
    =&  \left[\begin{array}{c|c} Z_p & \mathcal{I}_p \\ \hline
    \hat{L} & L_0\end{array}\right]-  \left[\begin{array}{c c |c} Z_n - \mathcal{I}_n\hat{R} & -\mathcal{I}_n\hat{N} & 0  \\ 0 & Z_p &  \mathcal{I}_p\\ \hline
  M_1\hat{R}-\hat{M} & M_1 \hat{N} & 0 \end{array}\right]  \\
   =&  \left[\begin{array}{c c c |c} Z_p & 0 & 0 & \mathcal{I}_p \\ 0 & Z_n - \mathcal{I}_n\hat{R} & -\mathcal{I}_n\hat{N} & 0  \\ 0 & 0 & Z_p &  \mathcal{I}_p\\ \hline
 \hat{L} & -M_1\hat{R}+\hat{M} & -M_1 \hat{N} & L_0\end{array}\right]
\end{aligned}
$$
Define a similarity transformation
$$
    \hat{T} = \begin{bmatrix}
    I & 0 & -I \\
    0 & I & 0 \\
    0 & 0 &  I \\
    \end{bmatrix}, \hat{T}^{-1} = \begin{bmatrix}
    I & 0 & I \\
    0 & I & 0 \\
    0 & 0 &  I \\
    \end{bmatrix}.
$$
We derive
\begin{equation*} \label{eq:SLScontroller}
\begin{aligned}
    \bm{\Phi}_{uy} - \bm{\Phi}_{ux}\bm{\Phi}_{xx}^{-1}\bm{\Phi}_{xy}=  \left[\begin{array}{c c |c} Z_n - \mathcal{I}_n\hat{R} & -\mathcal{I}_n\hat{N} & 0  \\  0 & Z_p &  \mathcal{I}_p\\ \hline
 -M_1\hat{R}+\hat{M} & -M_1\hat{N}+\hat{L} & L_0\end{array}\right].
\end{aligned}
\end{equation*}

\section{Derivations of~\eqref{eq:SLPuncertainty}}
\label{app:slperror}
For notational convenience, we use $(\mathbf{R},\mathbf{M},\mathbf{N},\mathbf{L})$ in place of $(\hat{\bm{\Phi}}_{xx},\hat{\bm{\Phi}}_{ux},\hat{\bm{\Phi}}_{xy},\hat{\bm{\Phi}}_{uy})$ here, as used in~\cite{wang2019system}. The following derivations utilize the affine relationship~\eqref{eq:slp_error} multiple times:
$$
    \begin{aligned}
        &(zI - A - B \mathbf{K} C)^{-1} \\
        = &(zI - A - B (\mathbf{L} - \mathbf{M}\mathbf{M}^{-1}\mathbf{N}) C)^{-1} \\
         = &(zI - A - B \mathbf{L}C + B\mathbf{M}\mathbf{R}^{-1}\mathbf{N}C)^{-1}\\
        =&(zI - A - B (\mathbf{M}(zI - A) - \hat{\bm{\Delta}}_4) + B\mathbf{M}\mathbf{R}^{-1}\mathbf{N}C)^{-1} \\
        = &(zI - A - B\mathbf{M}(zI - A)  + B\hat{\bm{\Delta}}_4 + B\mathbf{M}\mathbf{R}^{-1}\mathbf{N}C)^{-1} \\
        \end{aligned}
        $$
        $$
        \begin{aligned}
        = &(zI - A - B\mathbf{M}\mathbf{R}^{-1}(\mathbf{R}(zI - A)  - \mathbf{N}C) + B\hat{\bm{\Delta}}_4 )^{-1}\\
        = &(zI - A - B\mathbf{M}\mathbf{R}^{-1}(I + \hat{\bm{\Delta}}_3) + B\hat{\bm{\Delta}}_4 )^{-1} \\
        = &(zI - A - B\mathbf{M}\mathbf{R}^{-1} - B\mathbf{M}\mathbf{R}^{-1}\hat{\bm{\Delta}}_3 + B\hat{\bm{\Delta}}_4 )^{-1}
        \\
    =& \mathbf{R}((zI - A)\mathbf{R} - B\mathbf{M} - B\mathbf{M}\mathbf{R}^{-1}\hat{\bm{\Delta}}_3\mathbf{R} + B\hat{\bm{\Delta}}_4\mathbf{R})^{-1}
        \\
        = &\mathbf{R}(I + \hat{\bm{\Delta}}_1 - B\mathbf{M}\mathbf{R}^{-1}\hat{\bm{\Delta}}_3\mathbf{R} + B\hat{\bm{\Delta}}_4\mathbf{R})^{-1} \\
        = &\mathbf{R}(I + \hat{\bm{\Delta}}_1 - ((zI-A)\mathbf{R} - I - \hat{\bm{\Delta}}_1)\mathbf{R}^{-1}\hat{\bm{\Delta}}_3\mathbf{R} + B\hat{\bm{\Delta}}_4\mathbf{R})^{-1}. \\
    \end{aligned}
$$
We can further simplified this expression:
$$
    \begin{aligned}
        &\mathbf{R}(I + \hat{\bm{\Delta}}_1 - ((zI-A)\mathbf{R} - I - \hat{\bm{\Delta}}_1)\mathbf{R}^{-1}\hat{\bm{\Delta}}_3\mathbf{R} + B\hat{\bm{\Delta}}_4\mathbf{R})^{-1} \\
         =&((I + \hat{\bm{\Delta}}_1)\mathbf{R}^{-1} - (zI - A)\hat{\bm{\Delta}}_3 + (I + \hat{\bm{\Delta}}_1)\mathbf{R}^{-1}\hat{\bm{\Delta}}_3 + B\hat{\bm{\Delta}}_4)^{-1} \\
         =&(I + \hat{\bm{\Delta}})^{-1}\mathbf{R}(I + \hat{\bm{\Delta}}_1)^{-1},
    \end{aligned}
$$
where $\hat{\bm{\Delta}}$ is defined in~\eqref{eq:Runcertainty}.

\section{Robustness of Mixed I/II parameterizations} \label{app:mixedI/II}
We consider the Mixed I parameterization in Proposition~\ref{prop:slp3}. The transfer matrices $ \hat{\bm{\Phi}}_{yx}, \hat{\bm{\Phi}}_{ux},
            \hat{\bm{\Phi}}_{yy}, \hat{\bm{\Phi}}_{uy}$ only approximately satisfy the affine constraint~\eqref{eq:slp3constraint}, \emph{i.e.}, we have
  \begin{equation} \label{eq:slp3constrainterr}
             \begin{aligned}
           \begin{bmatrix} I & - \mathbf{G} \end{bmatrix} \begin{bmatrix}
            \hat{\bm{\Phi}}_{yx}  & \hat{\bm{\Phi}}_{yy}\\
            \hat{\bm{\Phi}}_{ux}  & \hat{\bm{\Phi}}_{uy}\\
        \end{bmatrix} &= {\begin{bmatrix} C(zI - A)^{-1} + \bm{\Delta}_1  &  I + \bm{\Delta}_2\end{bmatrix}},\\
        \begin{bmatrix}
            \hat{\bm{\Phi}}_{yx}  & \hat{\bm{\Phi}}_{yy}\\
            \hat{\bm{\Phi}}_{ux}  & \hat{\bm{\Phi}}_{uy}\\
        \end{bmatrix}  \begin{bmatrix} zI - A \\ -C \end{bmatrix} &= \begin{bmatrix} \bm{\Delta}_3 \\\bm{\Delta}_4 \end{bmatrix},
        \\
            \hat{\bm{\Phi}}_{yx}, \hat{\bm{\Phi}}_{ux},
            \hat{\bm{\Phi}}_{yy}, \hat{\bm{\Phi}}_{uy} &\in \mathcal{RH}_{\infty},
            \end{aligned}
        \end{equation}
where $
    \bm{\Delta}_1,  \bm{\Delta}_2,
    \bm{\Delta}_3, \bm{\Delta}_4$ are the computational {residuals}.
\begin{theorem} \label{theo:MixIrobustness}
  Let $ \hat{\bm{\Phi}}_{yx}, \hat{\bm{\Phi}}_{ux},
            \hat{\bm{\Phi}}_{yy}, \hat{\bm{\Phi}}_{uy}$ satisfy~\eqref{eq:slp3constrainterr}. Then, we have the following statements.
            \begin{enumerate}
                \item  In the case of $\mathbf{G} \in \mathcal{RH}_{\infty}$, the controller $\mathbf{K} = \hat{\bm{\Phi}}_{uy}\hat{\bm{\Phi}}_{yy}^{-1}$ internally stabilizes the plant $\mathbf{G}$ if and only if $(I + \bm{\Delta}_2)^{-1}$ is stable.
                \item In the case of $\mathbf{G} \notin \mathcal{RH}_{\infty}$, the controller $\mathbf{K} = \hat{\bm{\Phi}}_{uy}\hat{\bm{\Phi}}_{yy}^{-1}$ cannot guarantee the internal stability of the closed-loop system unless $\bm{\Delta}_1 = 0,\bm{\Delta}_2 = 0,\bm{\Delta}_3 = 0,\bm{\Delta}_4 = 0.$
            \end{enumerate}
\end{theorem}
The proof is almost identical to that of Theorem~\ref{theo:IOProbustness}.

We then consider Mixed II parameterization in Proposition~\ref{prop:slp4}. The transfer matrices $             \hat{\bm{\Phi}}_{xy}, \hat{\bm{\Phi}}_{uy},
            \hat{\bm{\Phi}}_{xu}, \hat{\bm{\Phi}}_{uu}$ only approximately satisfy the affine constraint~\eqref{eq:slp4constraint}, \emph{i.e.}, we have
  \begin{equation} \label{eq:slp4constrainterr}
             \begin{aligned}
           \begin{bmatrix} zI - A & - B \end{bmatrix} \begin{bmatrix}
            \hat{\bm{\Phi}}_{xy}  & \hat{\bm{\Phi}}_{xu}\\
            \hat{\bm{\Phi}}_{uy}  & \hat{\bm{\Phi}}_{uu}\\
        \end{bmatrix}  &= \begin{bmatrix} \bm{\Delta}_1 &  \bm{\Delta}_2 \end{bmatrix},\\
       \begin{bmatrix}
            \hat{\bm{\Phi}}_{xy}  & \hat{\bm{\Phi}}_{xu}\\
            \hat{\bm{\Phi}}_{uy}  & \hat{\bm{\Phi}}_{uu}\\
        \end{bmatrix} \begin{bmatrix} -\mathbf{G}\\ I \end{bmatrix} &= \begin{bmatrix}
            (zI - A)^{-1}B + \bm{\Delta}_3 \\
            I + \bm{\Delta}_4
        \end{bmatrix},
        \\
            \hat{\bm{\Phi}}_{xy}, \hat{\bm{\Phi}}_{uy},
            \hat{\bm{\Phi}}_{xu}, \hat{\bm{\Phi}}_{uu} &\in \mathcal{RH}_{\infty},
            \end{aligned}
        \end{equation}
where $
    \bm{\Delta}_1,  \bm{\Delta}_2,
    \bm{\Delta}_3, \bm{\Delta}_4$ are the computational residuals.
\begin{theorem} \label{theo:MixIIrobustness}
  Let $             \hat{\bm{\Phi}}_{xy}, \hat{\bm{\Phi}}_{uy},
            \hat{\bm{\Phi}}_{xu}, \hat{\bm{\Phi}}_{uu}$ satisfy~\eqref{eq:slp4constrainterr}. Then, we have the following statements.
            \begin{enumerate}
                \item  In the case of $\mathbf{G} \in \mathcal{RH}_{\infty}$, the controller $\mathbf{K} = \hat{\bm{\Phi}}_{uu}^{-1}\hat{\bm{\Phi}}_{uy}$ internally stabilizes the plant $\mathbf{G}$ if and only if $(I + \bm{\Delta}_4)^{-1}$ is stable.
                \item In the case of $\mathbf{G} \notin \mathcal{RH}_{\infty}$, the controller $\mathbf{K} = \hat{\bm{\Phi}}_{uu}^{-1}\hat{\bm{\Phi}}_{uy}$ cannot guarantee the internal stability of the closed-loop system unless $\bm{\Delta}_1 = 0,\bm{\Delta}_2 = 0,\bm{\Delta}_3 = 0,\bm{\Delta}_4 = 0.$
            \end{enumerate}
\end{theorem}
The proof is similar to that of Theorem~\ref{theo:IOProbustness}. We just highlight that for the case  $\mathbf{G} \in \mathcal{RH}_{\infty}$, we only need to check the closed-loop response from $\bm{\delta}_y$ to $\bm{u}$, which is
$$
\begin{aligned}
    \bm{u} &= \mathbf{K}(I - \mathbf{G}\mathbf{K})^{-1}\bm{\delta}_y \\
    &= (I - \mathbf{K}\mathbf{G})^{-1}\mathbf{K}\bm{\delta}_y \\
    & = (I - \hat{\bm{\Phi}}_{uu}^{-1}\hat{\bm{\Phi}}_{uy}\mathbf{G})^{-1}\hat{\bm{\Phi}}_{uu}^{-1}\hat{\bm{\Phi}}_{uy}\bm{\delta}_y \\
    &= (I + \bm{\Delta}_4)^{-1}\hat{\bm{\Phi}}_{uy}\bm{\delta}_y.
\end{aligned}
$$
Then, the first statement in Theorem~\ref{theo:MixIIrobustness} becomes obvious.

\section{Proof of Proposition~\ref{prop:initialcontroller}}
\label{app:initialcontroller}
\begin{proof}
    $\Leftarrow$  We first prove that $\forall\, \mathbf{K} \in \hat{\mathcal{C}}_{\text{stab}}$, we have $\mathbf{K} \in {\mathcal{C}}_{\text{stab}}$. Since $\mathbf{K} \in \hat{\mathcal{C}}_{\text{stab}}$, we know
    $$
    \begin{bmatrix}
     (I - \hat{\mathbf{G}}\mathbf{K}_1)^{-1} & (I - \hat{\mathbf{G}}\mathbf{K}_1)^{-1}\hat{\mathbf{G}} \\
      \mathbf{K}_1(I - \hat{\mathbf{G}}\mathbf{K}_1)^{-1} & (I - \mathbf{K}_1\hat{\mathbf{G}})^{-1}
    \end{bmatrix}
   \in  \mathcal{RH}_{\infty},
    $$
    where $\hat{\mathbf{G}} = (I - \mathbf{G}\mathbf{K}_0)^{-1}\mathbf{G}$ and $\mathbf{K} = \mathbf{K}_1 + \mathbf{K}_0$ with $\mathbf{K}_0 \in \mathcal{C}_{\text{stab}} \cap \mathcal{RH}_{\infty}$. Now, we verify that
     $$
        \begin{aligned}
              (I - \mathbf{G}(\mathbf{K}_0 + \mathbf{K}_1))^{-1}
              &=(I - \hat{\mathbf{G}}\mathbf{K}_1)^{-1}(I - \mathbf{G}\mathbf{K}_0)^{-1} \in \mathcal{RH}_{\infty} \\
              (I - \mathbf{G}(\mathbf{K}_0 + \mathbf{K}_1))^{-1}\mathbf{G} &=(I - \hat{\mathbf{G}}\mathbf{K}_1)^{-1}\hat{\mathbf{G}} \in \mathcal{RH}_{\infty}
        \end{aligned}
    $$
    and that
    $$
        \begin{aligned}
              &(\mathbf{K}_0 +\mathbf{K}_1)(I - \mathbf{G}\mathbf{K})^{-1} \\
              =&\mathbf{K}_0 (I - \mathbf{G}\mathbf{K})^{-1} + \mathbf{K}_1(I - \mathbf{G}\mathbf{K})^{-1} \in \mathcal{RH}_{\infty}.
        \end{aligned}
    $$
    Finally, we show
    $$
        \begin{aligned}
              (I - (\mathbf{K}_0 + \mathbf{K}_1)\mathbf{G})^{-1} = (I - \mathbf{K}_1\hat{\mathbf{G}})^{-1}(I - \mathbf{K}_0\mathbf{G})^{-1} \in \mathcal{RH}_{\infty}.
        \end{aligned}
    $$
    By Theorem~\ref{Theo:mainresult}, we have proved $\mathbf{K} = \mathbf{K}_0 + \mathbf{K}_1 \in \mathcal{C}_{\text{stab}}$.

    $\Rightarrow:$ This direction is similar: $\forall \mathbf{K} \in \mathcal{C}_{\text{stab}}$, we prove that $\mathbf{K}_1 = \mathbf{K} - \mathbf{K}_0$ internally stabilizes $\hat{\mathbf{G}}$. 
    For example
    $$
     \begin{aligned}
        (I - \hat{\mathbf{G}}\mathbf{K}_1)^{-1} &= (I - \mathbf{G}\mathbf{K})^{-1} (I + \mathbf{G}\mathbf{K}_0) \\
        &= (I - \mathbf{G}\mathbf{K})^{-1} + (I - \mathbf{G}\mathbf{K})^{-1}\mathbf{G}\mathbf{K}_0 \in \mathcal{RH}_{\infty},
    \end{aligned}
     $$
     and other conditions can be proved similarly. We complete the proof.
\end{proof}

\vspace{10mm}

\balance

{\small
\begin{spacing}{1.1}
\bibliographystyle{IEEEtran}        
\bibliography{references.bib}
\end{spacing}
}

\end{document}

%% file: interconnection.tex
\setlength{\unitlength}{0.008in}
\begin{picture}(243,170)(140,410)
\thicklines

\put(147,440){\vector(1, 0){ 93}}
\put(141,540){\vector( 0, -1){ 95}}
\put(88,440){\vector( 1, 0){ 48}}
\put(141,440){\circle{10}}
\put(170,445){\makebox(0,0)[lb]{$\mathbf{y}$}}
\put(100,445){\makebox(0,0)[lb]{$\delta_{\mathbf{y}}$}}
\put(145,450){\makebox(0,0)[lb]{\tiny +}}
\put(127,445){\makebox(0,0)[lb]{\tiny +}}


\put(385,540){\circle{10}}
\put(380,540){\vector( -1, 0){ 25}}
\put(385,439.5){\vector( 0,1){ 95.5}}
\put(438,540){\vector(-1, 0){ 48}}
\put(410,545){\makebox(0,0)[lb]{$\delta_{\mathbf{u}}$}}
\put(365,545){\makebox(0,0)[lb]{$\mathbf{u}$}}
\put(220,545){\makebox(0,0)[lb]{$\mathbf{x}$}}
\put(395,545){\makebox(0,0)[lb]{\tiny +}}
\put(390,528){\makebox(0,0)[lb]{\tiny +}}
\put(266,440){\line(1, 0){ 119.5}}

\put(240,427){\framebox(25,25){}} 
\put(245,435){\makebox(0,0)[lb]{$\mathbf{K}$}}

\put(330,527){\framebox(25,25){}} 
\put(335,535){\makebox(0,0)[lb]{$B$}}

\put(304,542){\makebox(0,0)[lb]{\tiny +}}
\put(300,547.5){\makebox(0,0)[lb]{\tiny +}}
\put(300,527.5){\makebox(0,0)[lb]{\tiny +}}

\put(329,540){\vector(-1, 0){26}}
\put(298,540){\circle{10}}
\put(298,580){\vector(0, -1){35}}
\put(301,570){\makebox(0,0)[lb]{$\delta_{\mathbf{x}}$}}
\put(293,540){\vector(-1, 0){25}}

\put(243,527){\framebox(25,25){}} 
\put(248,537){\makebox(0,0)[lb]{\tiny  $z^{\text{-}1}$}}

\put(242,540){\vector(-1, 0){46}}
\put(219,540){\line(0, -1){46}}
\put(218.5,494){\vector(1, 0){24}}

\put(243,482){\framebox(25,25){}} 
\put(248,490){\makebox(0,0)[lb]{$A$}}

\put(268.5,494){\line(1, 0){29.7}}
\put(298.5,493.5){\vector(0, 1){42}}

\put(171,527){\framebox(25,25){}} 
\put(176,535){\makebox(0,0)[lb]{$C$}}
\put(170.5,540){\line(-1, 0){30}}


{\color{red}
\put(160,475){\dashbox(203,90){}} 
\put(339,480){\makebox(0,0)[lb]{$\mathbf{G}$}}
}

\end{picture}

%% file: Basic_Loop_GK.tex
\setlength{\unitlength}{0.008in}
\begin{picture}(243,115)(140,410)
\thicklines
\put(179.5,440){\vector(1, 0){ 55}}
\put(180,510){\line( 0, -1){ 70}}
\put(179.5,510){\line( 1, 0){ 53.5}}

\put(315.5,510){\vector( -1, 0){ 50.5}}
\put(315,439.5){\line( 0,1){ 70.5}}
\put(265,440){\line(1, 0){ 50}}

\put(235,424){\framebox(30,30){}}
\put(243,433){\makebox(0,0)[lb]{$\mathbf{K}$}}

\put(235,494){\framebox(30,30){}}
\put(243,503){\makebox(0,0)[lb]{$\mathbf{G}$}}

\put(200,445){\makebox(0,0)[lb]{$\mathbf{y}$}}
\put(285,515){\makebox(0,0)[lb]{$\mathbf{u}$}}



\end{picture}

%% file: Double_Loop.tex
\setlength{\unitlength}{0.008in}
\begin{picture}(243,200)(140,350)
\thicklines
\put(179.5,440){\vector(1, 0){ 55}}
\put(180,510){\line( 0, -1){ 70}}
\put(179.5,510){\line( 1, 0){ 53.5}}

\put(315.5,510){\vector( -1, 0){ 50.5}}
\put(315,444.5){\line( 0,1){ 65.5}}
\put(265,440){\vector(1, 0){ 45}}

\put(235,424){\framebox(30,30){}}
\put(240,433){\makebox(0,0)[lb]{$\mathbf{K}_0$}}

\put(235,494){\framebox(30,30){}}
\put(243,503){\makebox(0,0)[lb]{$\mathbf{G}$}}

\put(200,395){\makebox(0,0)[lb]{$\mathbf{y}$}}
\put(280,515){\makebox(0,0)[lb]{$\mathbf{u}$}}
\put(280,445){\makebox(0,0)[lb]{$\mathbf{u}_0$}}
\put(280,395){\makebox(0,0)[lb]{$\mathbf{u}_1$}}



\put(180,440){\line( 0, -1){ 50}}

\put(235,374){\framebox(30,30){}}
\put(240,383){\makebox(0,0)[lb]{$\mathbf{K}_1$}}
\put(179.5,390){\vector(1, 0){ 55}}
\put(315,440){\circle{10}}

\put(265,390){\line(1, 0){ 50}}
\put(315,390){\vector(0, 1){ 45}}

\put(300,443){\makebox(0,0)[lb]{\tiny +}}
\put(318,428){\makebox(0,0)[lb]{\tiny +}}

{\color{red}
\put(160,417){\dashbox(180,117){}} 
\put(320,510){\makebox(0,0)[lb]{$\hat{\mathbf{G}}$}}
}

\end{picture}

%% file: Basic_Loop_hatGK.tex
\setlength{\unitlength}{0.008in}
\begin{picture}(243,115)(140,410)
\thicklines
\put(179.5,440){\vector(1, 0){ 55}}
\put(180,510){\line( 0, -1){ 70}}
\put(179.5,510){\line( 1, 0){ 53.5}}

\put(315.5,510){\vector( -1, 0){ 50.5}}
\put(315,439.5){\line( 0,1){ 70.5}}
\put(265,440){\line(1, 0){ 50}}

\put(235,424){\framebox(30,30){}}
\put(240,433){\makebox(0,0)[lb]{$\mathbf{K}_1$}}

\put(235,494){\framebox(30,30){}}
\put(242,503){\makebox(0,0)[lb]{$\hat{\mathbf{G}}$}}

\put(200,445){\makebox(0,0)[lb]{$\mathbf{y}$}}
\put(285,515){\makebox(0,0)[lb]{$\mathbf{u}_1$}}



\end{picture}